\documentclass[12pt]{article}

\usepackage{amsthm, amsmath, amssymb, dsfont, mathrsfs, graphicx, color, bbm, float, xcolor, enumitem, subcaption}
\usepackage[sorting=none]{biblatex}
\usepackage[margin=1in]{geometry}
\usepackage[colorlinks=true, linkcolor=blue, citecolor=red, urlcolor=cyan]{hyperref}
\usepackage{authblk}

\newtheorem{definition}{Definition}

\newtheorem{remark}{Remark}
\newtheorem{theorem}{Theorem}
\newtheorem*{theorem*}{Theorem}
\newtheorem{lemma}{Lemma}
\newtheorem{proposition}[lemma]{Proposition}
\newtheorem{claim}{Claim}
\newtheorem{corollary}[lemma]{Corollary}
\theoremstyle{definition}
\newtheorem*{definition*}{Definition}
\newtheorem{example}{Example}

\theoremstyle{definition}
\newtheorem*{example*}{Example}

\newcommand{\rmg}{\mathrm{g}}
\newcommand{\rmgonegtwo}{(\mathrm g_1,\mathrm g_2)}

\title{Interface for variants of the contact process}

\author[1]{Isabella Alvarenga}
\author[2]{Daniel Valesin}

\affil[1]{Universit\'e d'Orl\'eans, Orl\'eans, France}
\affil[2]{University of Warwick, Coventry, United Kingdom}

\affil[]{\textit{Corresponding author:} \texttt{isabella.goncalves-de-alvarenga@cnrs.fr}}
\affil[]{\textit{Email:} \texttt{daniel.valesin@warwick.ac.uk}}

\date{}

\begin{document}

\maketitle

\begin{abstract}
We study two one-dimensional variants of the contact process: the contact-and-barrier process, where the population evolves in a region delimited by a randomly moving barrier, and the multitype contact process, in which two species compete for space. The contact-and-barrier process is started with the barrier at the origin and all sites to its right occupied, while the multitype contact process is started from the Heaviside configuration with species~$1$ to the left of the origin and species~$2$ to the right. We prove that both models exhibit tight interfaces and that, after centring by an appropriate deterministic speed, the interface position satisfies a central limit theorem. Our analysis relies on a renewal-time method based on a novel construction called patchwork construction, in which the processes are built by concatenating space–time evolutions over successive time intervals of random length, providing a more convenient framework for defining the renewal times that drive the proofs.
\end{abstract}

\section{Introduction}

\paragraph{} The classical contact process, introduced in \cite{harris1974contact}, is a continuous-time Markov process on $\{0,1\}^{\mathbb Z}$ modelling the spread of an infection or population growth. Sites in state $1$ are occupied and sites in state $0$ are empty; occupied sites become empty at rate $1$, while empty sites become occupied at rate $\lambda$ times the number of occupied nearest neighbours. We study two variations of this process.

The first one is the contact-and-barrier, a modification in which the population evolves in an environment delimited by a randomly moving barrier. We restrict attention to the case of a single barrier. The process $(\beta_t)_{t\ge0}$ takes values in $\{-1,0,1\}^{\mathbb Z}$, where $1$ denotes an occupied site, $0$ an empty site, and $-1$ the barrier. The barrier attempts to move as a nearest-neighbour continuous-time random walk, while particles evolve according to the usual contact-process dynamics except that births at the barrier location are forbidden. A formal graphical construction is given in Subsection~\ref{GraphicalConstructionCBP}.

The second variation considered is the multitype contact process \cite{Neuhauser1992}, a continuous-time Markov process on $\{0,1,2\}^{\mathbb Z}$ representing two competing species. Each type attempts to evolve according to its own contact-process dynamics, with births onto sites occupied by the other type prohibited. A formal construction is provided in Section~\ref{sec_GCMCP}.

\subsection{Definition of the model and main results}

\subsection*{Classical contact process}

\paragraph*{}

We briefly recall a few fundamental facts about the classical contact process that are needed to state the main results of this paper. The first one is that the identically-zero configuration, denoted~$\underbar{0}$, is an absorbing state for the process. 

Let~$(\zeta_t^0)_{t\geq 0}$ denote the contact process  started from the configuration where a single particle is placed at the origin. 
It is shown in~\cite{harris1974contact, bezuidenhout1990critical} that there exists~$\lambda_c\in(0,\infty)$ such that if~$\lambda\leq\lambda_c$ the process dies out with probability one, i.e.,~$\mathbb P\left(\exists t\geq 0 \,:\,\zeta^0_t=\underbar{0}\right)=1$, whereas if~$\lambda>\lambda_c$, the process has positive probability of survival, meaning that $\mathbb P\left(\zeta^0_t\neq \underbar{0}\; \forall t \ge 0\right)>0$.

Let~$R_t^0 := \sup\{x\in\mathbb Z\,:\,\zeta^0_t(x)=1\}$ be the position of the rightmost occupied site of this process. It is shown in~\cite{durrett1980growth} that if~$\lambda > \lambda_c$, conditioning on the event that the process survives, there exists a deterministic value~$\alpha=\alpha(\lambda)>0$ such that
\begin{equation}
    \label{speedCP}
        \frac{R_t^0}{t}\xrightarrow{t \to \infty}\alpha(\lambda)\quad \text{ almost surely and in }L^1.
\end{equation}

\subsection*{Contact-and-barrier process}

\paragraph{} 

Define the space of configurations
\begin{equation} \label{Eq_Space_Config}
	\mathcal B:= \bigcup_{x \in \mathbb Z}\{\beta \in \{-1,0,1\}^{\mathbb Z}:\; \beta(y)=0 \;\forall y < x,\; \beta(x)=-1,\; \beta(y)\in \{0,1\} \;\forall y > x\}.
\end{equation}

Throughout this work, we consider only contact-and-barrier processes~$(\beta_t)_{t \geq 0}$ with initial configuration~$\beta_0 \in \mathcal{B}$; that is, configurations in which a single site is occupied by the barrier and all particles are located to the right of it. From the configurations in~$\mathcal B$, we highlight the Heaviside configuration of this model, namely the one where the barrier is placed at the origin and all the sites on the right of it are occupied. 

This model depends on the choice of four real parameters:~$\lambda$,~$r_{\leftarrow}$,~$r_{\rightarrow}^0$ and~$r^1_{\rightarrow}$. The first one will dictate the reproduction dynamics of particles; the last three will be the ones that will govern the behaviour of the barrier. 

The dynamics runs as follows. Particles die at rate~$1$, leaving the site empty, and send a copy of themselves to each empty neighbouring site at rate~$\lambda$ (births onto the site occupied by the barrier are forbidden). The barrier jumps one unit to the left with rate~$r_{\leftarrow}$, and jumps one unit to the right with rate~$r_{\rightarrow}^0$ if the destination site is empty and with rate~$r_{\rightarrow}^1$ if the destination site is occupied (in that case, the particle that was in that destination site is destroyed). 

Note that if~$\beta_0\in \mathcal B$, then one has that~$\beta_t\in\mathcal B$ for all~$t\geq 0$. 
\begin{definition}[Interface for contact-and-barrier process] \label{def_interfaceCBP}
    For a contact-and-barrier process~$(\beta_t)_{t\geq 0}$ started from~$\beta_0\in\mathcal B$, define
\begin{align*}
    &B_t:= x \text{ such that }\beta_t(x)=-1;\\
    &\ell_t:=\inf\{x\in\mathbb Z\,:\,\beta_t(x)=1\}.
\end{align*}
    The process~$(I_t)_{t\geq 0}:=(B_t,\ell_t)_{t \ge 0}$ is called the interface process. We call~$i_t:=\frac{B_t+\ell_t}{2}$ the interface position. 
\end{definition}

 In all that follows, we assume that~$\lambda > \lambda_c$ (the critical value of the classical contact process, with no barrier), and let~$\alpha$ be as in~\eqref{speedCP}.

In case~$r^1_\rightarrow \le r^0_\rightarrow$, the barrier can only be slowed down by particles, causing the barrier motion to be subadditive. This leads to the following.
\begin{proposition} \label{prop_real_speed_barrier}
    Assume that~$r^1_\rightarrow\leq r^0_\rightarrow$ and that the contact-and-barrier process~$(\beta_t)_{t\geq 0}$ is started from the Heaviside configuration. Then, there exists a deterministic~$\mathbf{B} \in \mathbb R$ such that
    \begin{equation} \label{real_speed_barrier0}
        \lim_{t\rightarrow\infty}\frac{B_t}{t}=\mathbf{B} \quad \text{almost surely and in }L^1.
    \end{equation}
\end{proposition}

{Our methods allow us to prove the existence of the speed of the barrier under a different set of assumptions, namely:
\begin{equation}\label{eq_A1}\tag{$\mathrm{A1}$}
        \min\{r_\rightarrow^0,r_\rightarrow^1\} - r_{\leftarrow}>-\alpha. 
\end{equation}
Putting this together with Proposition~\ref{prop_real_speed_barrier}, we have the following:}
\begin{theorem}[Speed of the barrier]\label{thmCBP_Speed}
    For the contact-and-barrier process with~$\lambda > \lambda_c$, barrier jump rates satisfying either~$r_\rightarrow^1\le r_\rightarrow^0$ or~$(\mathrm{A1})$, and started from the Heaviside configuration, there exists a deterministic~$\mathbf{B} \in \mathbb R$ such that
    \begin{equation} \label{real_speed_barrier}
        \lim_{t\rightarrow\infty}\frac{B_t}{t}=\mathbf{B} \quad \text{almost surely and in }L^1.
    \end{equation}
\end{theorem}



\begin{theorem}[Tightness of the interface] \label{ThmCBP_Tightness}
    For the contact-and-barrier process with~$\lambda > \lambda_c$, barrier jump rates satisfying either~$(\mathrm{A1})$ or
    \begin{equation}\label{eq_A2} \tag{$\mathrm{A2}$} r_{\rightarrow}^1\leq r_{\rightarrow}^0 \quad \text{ and} \quad \mathbf B >-\alpha,
\end{equation}
    and started from the Heaviside configuration, we have that the size of the interface is tight: 
    \begin{equation*}
        \lim_{L \to \infty} \sup_{t \ge 0} \mathbb P(\ell_t - B_t > L) = 0.
    \end{equation*}
\end{theorem}

\begin{theorem}[CLT for interface position]\label{ThmCBP_CLT_Speed}
     For the contact-and-barrier process with~$\lambda > \lambda_c$, barrier jump rates satisfying either~$(\mathrm{A1})$ or~$(\mathrm{A2})$, and started from the Heaviside configuration, there exists~$\sigma>0$ such that
    \begin{equation*}
             \frac{i_t- \mathbf B \cdot t}{\sqrt{t}}\xrightarrow[t \to \infty]{\mathrm{(dist)}} \mathcal{N}(0,\sigma^2)
    \end{equation*}
\end{theorem}

\begin{remark}
In~\eqref{eq_A2}, we assume that~$r^1_\rightarrow \le r^0_\rightarrow$ to guarantee that~$\mathbf B$ is well defined. Also note that when~$\mathbf B$ is well defined, it satisfies~$\mathbf B \ge \min\{r^0_\rightarrow,r^1_\rightarrow\}- r_\leftarrow$.
\end{remark}

\subsection*{Multitype contact process}
\paragraph{} Consider the space of configurations~$\mathcal C$ defined by:
\begin{equation*}
    \mathcal C:=\bigcup_{x\in\mathbb Z}\left\{\xi\in\{0,1,2\}^{\mathbb Z}\,:\,\xi(y) \in \{0,1\}\;\forall y\leq x\text{ and }\xi(y) \in \{0,2\}\;\forall y>x\right\}
\end{equation*}
In words, configurations in~$\mathcal C$ are such that there exists a site for which all individuals of type~$1$ are on the left of that site (or on that site included) and all individuals of type~$2$ are on the right of that site.
Throughout all this work, we only consider multitype contact process whose initial configurations lie in~$\mathcal C$. From this space of configurations, we single out the \emph{Heaviside configuration}, obtained by placing an individual of type~$1$ at every site~$x \leq 0$, and an individual of type~$2$ at every site~$x \geq 1$.

The model is governed by two parameters,~$\lambda_1, \lambda_2 \in [0,\infty)$. For fixed values of these parameters, the dynamics evolves as follows: every occupied site, regardless of type, becomes empty at rate~$1$. A site occupied by an individual of type~$i\in\{1,2\}$ sends a copy of itself to each empty nearest neighbour at rate~$\lambda_i$. Births onto sites already occupied by the opposite type are prohibited.

\begin{definition}[Interface for multitype contact process] \label{def_interfaceMCP}
    For a multitype contact process~$(\xi_t)_{t\geq 0}$ started from~$\xi_0\in\mathcal B$, the process~$(I_t)_{t\geq 0}$ given by the pair~$I_t:=(r_t,\ell_t)$ where~$r_t := \sup\{x\in\mathbb Z\,:\,\xi_t(x)=1\}$ and~$\ell_t := \inf\{x\in\mathbb Z\,:\,\xi_t(x)=2\}$ is called the interface process. Let~$i_t:=\frac{r_t+\ell_t}{2}$ be the interface position. 
\end{definition}

Note that if~$\xi_0\in\mathcal C$, it follows that~$\xi_t\in\mathcal C$ for all~$t\geq 0$. Therefore,~$r_t$ and~$\ell_t$ are well-defined for all~$t\geq 0$, and also satisfy~$r_t<\ell_r$.

\begin{theorem}[Tightness of the interface] \label{Thm_MCP_Tight}
    For a multitype contact process started from the Heaviside configuration with parameters~$\lambda_1,\lambda_2>\lambda_c$, we have that the size of the interface is tight:
    \begin{equation} \label{eq_tightnessMCP}
        \lim_{L \to \infty} \sup_{t \ge 0} \mathbb P(\ell_t - r_t > L) = 0
    \end{equation}
\end{theorem}

\begin{theorem}[CLT for interface position]\label{Thm_MCP_CLT}
    For a multitype contact process started from the Heaviside configuration with parameters~$\lambda_1,\lambda_2>\lambda_c$, there exists real constants~$\mu$ and~$\sigma>0$ such that:
    \begin{equation*}
             \frac{i_t- \mathbf \mu t}{\sqrt{t}}\xrightarrow[t \to \infty]{\mathrm{(dist)}} \mathcal{N}(0,\sigma^2)
    \end{equation*}
\end{theorem}

We have used the same notation~$(I_t)_{t\geq 0}$ for both the interface process of the contact-and barrier process and of the multitype contact process as it shall be clear from the context as to which interface we are referring to. The same notation highlights the fact that, in some sense, those two interfaces are the same object, and the strategy developed in Sections~\ref{section_PatchworkCBP} (for the contact and barrier process) and~\ref{section_patchworkMCP} (for the multitype contact process) is the same: although the results are not directly transferable and the proofs require non-trivial adjustments, the core structure of the argument is unchanged.

\subsection{Interfaces and related work}

\paragraph{} Many spatial stochastic systems exhibit two large homogeneous regions separated by a narrow transition zone, or \emph{interface}, whose long-term behaviour—whether it expands, fluctuates diffusively, or remains localized—is a central problem in interacting particle systems. One of the first works to highlight this phenomenon in a probabilistic setting is \cite{CoxDurrett1995}, where the authors study the one-dimensional voter model and prove that the interface started from the Heaviside configuration remains tight. Alternative proofs for the same result were later obtained using inversion-counting arguments \cite{Sturm2008}.

Interface questions also arise naturally in the multitype contact process. For nearest-neighbour interactions on $\mathbb Z$ and symmetric infection rates $\lambda_1=\lambda_2>\lambda_c$, the interface is tight \cite{Valesin2010Multitype}, and its position converges under diffusive scaling to Brownian motion \cite{Mountford2016}. These results rely on duality and ancestor-process techniques, which become considerably more difficult to apply in asymmetric settings \cite{Neuhauser1992}. Interface tightness has also been established in related competitive growth models such as the grass–bushes–trees process \cite{Andjel2018}, while several fundamental questions for the multitype contact process still remain open.

In one dimension, the rightmost particle of a species in the multitype contact process effectively acts as a moving boundary restricting the region accessible to the competing type. This observation motivates the introduction of the \emph{contact-and-barrier} process, which preserves the key geometric mechanism governing the interface while avoiding some of the technical difficulties of the fully multitype system. The model also places the problem within the broader framework of stochastic processes evolving in dynamic random environments.

More precisely, the contact-and-barrier process connects two active research directions: contact processes in random dynamical environments and random walks in random dynamical environments. Previous studies of contact processes in evolving environments typically assume that the environment evolves independently of the particle system \cite{broman2007stochastic, steif2007critical, remenik2008contact, cardona2024contact, leite2024contact, schapira2023contact, linker2020contact, hilario2022results}, whereas in the present model the environment is \emph{interdependent}. From another perspective, the barrier may be viewed as a random walk in a dynamic random environment, linking the model to the extensive random walks in random dynamical environments literature \cite{Boldrighini1992, Madras1992, Sznitman2002, BolthausenSznitman2002}. Related works where the environment is generated by a contact process include \cite{Bethuelsen2016, den2014scaling}, whose infection-path constructions will play an important role in our analysis. A distinctive feature of the present setting is the absence of uniform ellipticity (which guarantees a positive chance for the walker to move in any direction at all times), since the barrier’s transition probabilities may vanish in certain local configurations.

\subsection{Overview}

\paragraph{} We conclude with a brief outline of the paper. In Section~\ref{SectionClassicalCP}, we review the classical contact process. Section~\ref{section_lemmasRenewal} presents two renewal lemmas: the first identifies an i.i.d.\ structure within a stochastic process, while the second yields a central limit theorem once suitable control of temporal and spatial increments between these renewal times is established. The proofs of these lemmas are deferred to Section~\ref{sec_proofRenewalLemmas}.

Sections~\ref{sec_CBP} and~\ref{sec_MCP} contain the main contributions. In Section~\ref{sec_CBP}, we analyse the contact-and-barrier process and prove Theorems~\ref{thmCBP_Speed},~\ref{ThmCBP_Tightness}, and~\ref{ThmCBP_CLT_Speed}, while Section~\ref{sec_MCP} establishes Theorems~\ref{Thm_MCP_Tight} and~\ref{Thm_MCP_CLT} for the multitype contact process. In both settings, after presenting the graphical construction, we introduce the \emph{patchwork construction}, which yields renewal times with i.i.d.\ temporal and spatial increments and allows the renewal lemmas to be applied.

\section{Preliminaries on the contact process} \label{SectionClassicalCP}

\paragraph{} 
The classical contact process can be constructed using a Harris (or graphical) construction. This consists of a collection of independent Poisson point processes
\begin{equation*}
    \mathcal{H} = \left((D^x)_{x \in \mathbb{Z}},\ (D^{x,y})_{x,y \in \mathbb{Z},\, |x - y| = 1}\right),
\end{equation*}
where each~$D^x$ is a Poisson process on~$\mathbb R$ of rate~1 (representing death events at site~$x$), and each~$D^{x,y}$ is a Poisson process on~$\mathbb R$ of rate~$\lambda$ (representing potential reproduction events from site~$x$ to its neighbour~$y$).

We understand this graphical construction as follows. For each point~$x\in\mathbb Z$, we append a temporal axis which is just a half-line~$[0,\infty)$. For each non-negative realisation~$t\in D^x$, we draw a~$\times$ at~$(x,t)$, and for each non-negative realisation~$t\in D^{x,y}$ we draw an arrow~$\rightarrow$ from~$(x,t)$ to~$(y,t)$. 

\begin{definition}[Infection path] \label{def_InfectionPath}
    Consider a graphical construction~$\mathcal H$. Let~$I\subset [0,\infty)$ be an interval. We say that a function~$\gamma:I\rightarrow \mathbb Z$ is an infection path (in~$\mathcal H$ in case we want to highlight the construction used) if the following conditions are satisfied:
    \begin{itemize}
        \item~$t\notin D^{\gamma(t)}$ for all~$t\in I$
        \item if~$\gamma(t) \neq \gamma(t-):=\lim_{s\uparrow t}\gamma(s)$, then~$t\in D^{\gamma(t-), \gamma(t)}$
    \end{itemize}
    
    For~$x,y\in\mathbb Z$ and~$s,t\in[0,\infty)$ with~$s\leq t$, we write~$(x,s)\rightsquigarrow(y,t)$ if there exists~$\gamma:[s,t]\rightarrow \mathbb Z$ infection path with~$\gamma(s)=x$ and~$\gamma(t)=y$. If such a path does not exist, we write~$(x,s)\not\rightsquigarrow(y,t)$. We write~$A\times \{s\}\rightsquigarrow (y,t)$ to indicate that there exists~$x\in A$ such that~$(x,s)\rightsquigarrow(y,t)$. Similarly, we write~$(x,s)\rightarrow B\times \{t\}$ if there exists~$y\in B$ such that~$(x,s)\rightsquigarrow(y,t)$. We also write~$(x,t)\rightsquigarrow \infty$ in case there exists an infinite infection path starting from~$(x,t)$. 
    
    In all of those cases, we may replace~$\rightsquigarrow$ by~$\overset{\mathcal H}{\rightsquigarrow}$ in case we want to highlight the graphical construction used. Moreover, for the cases where there are more than one graphical construction being considered, we also write~$\overset{\mathcal H_1\text{ or }\mathcal H_2}{\rightsquigarrow}$ in order to denote the existence of an infection path either in the graphical construction~$\mathcal H_1$ or in the graphical construction~$\mathcal H_2$. 
\end{definition}

From a graphical construction alongside with the notion of infection path, one can construct the classical contact process~$(\zeta_t)_{t\geq 0}$ started from an initial configuration~$\zeta_0\in\{0,1\}^{\mathbb Z}$ in the following way: for any~$t\in[0,\infty)$ and for any~$x\in \mathbb Z$, we claim that~$\zeta_t(x)=1$ if and only if there exists an infection path~$(y,0)\rightsquigarrow(x,t)$ from some~$y\in \mathbb Z$ with~$\zeta_0(y)=1$. 

Let~$\delta_{\underbar{0}}$ be the probability measure on the space of configurations of the contact process that attributes mass~$1$ to the configuration~$\underbar{0}$. Clearly,~$\delta_{\underbar{0}}$ is an invariant measure for the contact process. To characterise other invariant measures for this process, we recall the complete convergence theorem~\cite{liggett1985interacting}. For $\lambda>\lambda_c$, there exists a probability measure $\nu_\lambda$ on $\{0,1\}^{\mathbb Z}$, called the upper invariant measure, such that the process started from the fully occupied configuration converges in distribution to $\nu_\lambda$. More generally, if the process is started from $\mathds{1}_A$ for $A\subset\mathbb Z$, then
\[
\zeta_t^A \xrightarrow[t\to\infty]{(\mathrm{dist})}
\mathbb P(\exists t\ge0:\zeta_t^A=\underline{0})\,\delta_0
+ \mathbb P(\forall t\ge0:\zeta_t^A\neq\underline{0})\,\nu_\lambda.
\]

A useful characterization of $\nu_\lambda$ follows from the self-duality of the contact process (see Section~III.4 and Example~4.18 in \cite{liggett1985interacting}). For any finite set $A\subset\mathbb Z$,
\[
\mathbb P\big(\zeta_t^1(x)=1\ \forall x\in A\big)
=\mathbb P\big((x,0)\rightsquigarrow \mathbb Z\times\{t\}\ \forall x\in A\big)
\xrightarrow[t\to\infty]{}
\mathbb P\big((x,0)\rightsquigarrow \infty\ \forall x\in A\big).
\]

Consequently, a configuration sampled from $\nu_\lambda$ can be obtained from the graphical construction by declaring $\zeta(x)=1$ if and only if $(x,0)\rightsquigarrow\infty$. This observation will play an important heuristic role in Section~\ref{section_PatchworkCBP}.

Let $A\subset\mathbb Z$ be finite and let $(\zeta_t^A)_{t\ge0}$ be the contact process with parameter $\lambda>\lambda_c$ started from $\mathds{1}_A$. Denote by~$\tau_A^{\mathrm{ext}}=\inf\{t\ge0:\zeta_t^A(x)=0\ \forall x\in\mathbb Z\}$ its extinction time. Then there exist constants $c,C>0$, depending only on $\lambda$, such that (Theorems~3.23 and~3.29 in \cite{liggett1985interacting}):
\begin{align}
\label{eq_tau_ext_A}
\mathbb P\big(\tau_A^{\mathrm{ext}}=\infty\big) &\ge 1-Ce^{-c|A|},\\
\label{eq_small_cluster}
\mathbb P\big(t<\tau_A^{\mathrm{ext}}<\infty\big) &\le Ce^{-ct}.
\end{align}

Combining Equation 33 of~\cite{griffeath1983basic} and Theorem 4 of~\cite{durrett1983supercritical}, we state the following result:

\begin{proposition}[Bounds for the position of the rightmost particle \cite{griffeath1983basic,durrett1983supercritical}] \label{RightmostWellBehaved}
    Let~$(\zeta^h_t)_{t\geq 0}$ be a contact process with parameter~$\lambda>\lambda_c$ started from the Heaviside configuration~$\mathds{1}_{(-\infty,0]}$ and let~$\alpha$ be its associated speed as given in~\eqref{speedCP}. Let~$R_t^h := \sup\{x\in \mathbb Z\,:\,\zeta_t^h(x)=1\}$. Then, for any~$\epsilon > 0$ there exists~$c > 0$ such that
    \begin{equation}\label{eq_R_within_bounds}
        \mathbb P\left(\left|\frac{R_t^h}{t}  - \alpha\right| \leq  \epsilon\right) \geq 1-e^{-ct} \quad \text{for all } t > 0.
    \end{equation} 
\end{proposition}

\begin{corollary} \label{CorRightmostWellBehaved} Assume that~$\lambda > \lambda_c$. Let~$\zeta'$ be a configuration in~$\{0,1\}^{\mathbb Z}$ drawn from~$\nu_{\lambda}$, and assume we are given a graphical construction~$\mathcal H$ of the contact process with rate~$\lambda$, independently of~$\zeta'$. Then, there exist~$c,C>0$ such that the following holds  for all~$L > 0$ and~$t > 0$:
    \begin{equation}\label{eq_for_inside_the_prob}
        \mathbb P\left( \exists x \in [-L,0]:\; \zeta'(x)=1 \text{ and } (x,0) \overset{\mathcal H}{\rightsquigarrow} [-L+(\alpha-\epsilon) t,\infty) \times \{t\}  \right)\geq 1-Ce^{-cL}-e^{-ct}.
    \end{equation}
\end{corollary}

\begin{proof}
Let~$\rho:= \mathbb{P}((0,0) \rightsquigarrow \infty)$, which is positive since~$\lambda > \lambda_c$. Theorem~1 in Durrett and Schonmann~\cite{durrett1988large} (a large deviations principle for the density of~$\nu_\lambda$) implies that there exist~$c,C > 0$ such that
\[\mathbb{P}(|\zeta' \cap [-L,0]| > \tfrac{\rho}{2}L) > 1- Ce^{-cL}\]
for all~$L > 0$. Hence,
\begin{align}
\label{eq_for_inside_the_prob1}&\mathbb{P}( \exists x \in \zeta' \cap [-L,0]:\; (x,0) \overset{\mathcal H}{\rightsquigarrow} \infty) \\
\nonumber&\ge 1- Ce^{-cL} - \mathbb{P}(|\zeta' \cap [-L,0]| > \tfrac{\rho}{2}L,\; \tau^{\mathrm{ext}}_{\zeta' \cap [-L,0]} < \infty) \stackrel{\eqref{eq_tau_ext_A}}{\ge} 1-Ce^{-cL}-e^{-c\rho L/2}.
\end{align}
Next, Proposition~\ref{RightmostWellBehaved} gives
\begin{equation}\label{eq_for_inside_the_prob2}
\mathbb{P}(\exists y,z \in \mathbb Z: y \le -L, \; z \ge -L+(\alpha - \epsilon)t,\; (y,0) \overset{\mathcal H}{\rightsquigarrow} (z,t) ) > 1-e^{-ct}
\end{equation}
for all~$t > 0$.
To conclude, we observe that the intersection of the events inside the probabilities in~\eqref{eq_for_inside_the_prob1} and~\eqref{eq_for_inside_the_prob2} is contained in the event inside the probability in~\eqref{eq_for_inside_the_prob}\qedhere.
\end{proof}

At last, we introduce the notion of an infection path that is random, in the sense that it depends on the realization of the process, yet can be identified through a procedure that preserves the available information inside a certain region. This framework was originally introduced in~\cite{den2014scaling}, and we adapt several key concepts from their work. Most importantly, we state Lemma~\ref{lem_RPRIP_property} from their work (without reproof) as the key technical tool for our analysis. To formalize this idea, we begin with the following definition.


\begin{definition} \label{def_LeftRightPlus}
Let~$t \ge 0$, and let~$\pi:[0,t] \to \mathbb Z$ be a c\`adl\`ag function. Let~$\mathcal{H}_{[0,t]}$ be the graphical construction obtained from~$\mathcal H$ by deleting all recovery marks and transmission arrows outside the interval~$[0,t]$. Define:

\begin{itemize}
    \item $\mathrm{Left}_\pi(\mathcal H)$ the graphical construction obtained from~$\mathcal H_{[0,t]}$ by keeping recovery marks of~$\mathcal H_{[0,t]}$ inside the space-time set~$\{(x,s) \in \mathbb Z \times [0,t]: x \le \max\{\pi(s-),\pi(s)\}\}$, and transmission arrows of~$\mathcal H_{[0,t]}$ with both start- and end-points in the same space-time set (all other recovery marks and transmission arrows are deleted);
    \item ~$\mathrm{Right}^+_\pi(\mathcal H)$ the graphical construction obtained from~$\mathcal H_{[0,t]}$ by keeping all recovery marks and transmission arrows of~$\mathcal H_{[0,t]}$ inside~$\mathbb Z \times [0,t]$ that are \emph{not} included in~$\mathrm{Left}_\pi(\mathcal H)$ (and deleting all other ones).
\end{itemize}

\end{definition}

With regards to this definition, note that if~$\pi$ is itself an infection path of~$\mathcal H$, then the transmission arrows that it traverses are included in~$\mathrm{Left}_\pi(\mathcal H)$, but not in~$\mathrm{Right}_\pi^+(\mathcal H)$ (which is why we include the~`$+$' in the notation).

\begin{definition}
    Let~$t \ge 0$, and let~$\pi:[0,t] \to \mathbb Z$ be a c\`adl\`ag function. Define:

\begin{itemize}
    \item $\mathrm{Right}_\pi(\mathcal H)$ the graphical construction obtained from~$\mathcal H_{[0,t]}$ by keeping recovery marks of~$\mathcal H_{[0,t]}$ inside the space-time set~$\{(x,s) \in \mathbb Z \times [0,t]: x \ge \max\{\pi(s-),\pi(s)\}\}$, and transmission arrows of~$\mathcal H_{[0,t]}$ with both start- and end-points in the same space-time set (all other recovery marks and transmission arrows are deleted);
    \item ~$\mathrm{Left}^+_\pi(\mathcal H)$ the graphical construction obtained from~$\mathcal H_{[0,t]}$ by keeping all recovery marks and transmission arrows of~$\mathcal H_{[0,t]}$ inside~$\mathbb Z \times [0,t]$ that are \emph{not} included in~$\mathrm{Left}_\pi(\mathcal H)$ (and deleting all other ones).
\end{itemize}
\end{definition}


We will encounter random objects of the form~$\Gamma:[0,T]\to \mathbb Z$, where~$T$ is itself random, so we now introduce a measurability structure for such objects. The space where such objects take values is
\begin{align*}
    &\mathcal D([0,*],\mathbb Z):=\bigcup_{t \ge 0} \{\gamma:[0,t]\to \mathbb Z,\; \gamma \text{ c\`adl\`ag}\}.
\end{align*}

We endow this space with the following metric: the distance between~$\gamma_1:[0,t_1] \to \mathbb Z$ and~$\gamma_2:[0,t_2]\to \mathbb Z$ is set to be~$\max \left\{|t_1-t_2|,\; \mathrm{dist}_{\mathrm{Sk}}(\bar{\gamma}_1,\bar{\gamma}_2), \right\}$ where~$\bar{\gamma}_i(t):=\gamma_i(\min\{t,t_i\})$, with~$t \ge 0$ and~$i \in \{1,2\}$, and~$\mathrm{dist}_{\mathrm{Sk}}$ is the Skhorohod metric on c\`adl\`ag functions~$\gamma:[0,\infty) \to \mathbb Z$. This is then a Polish space, which we endow with the Borel~$\sigma$-algebra. We will abuse the terminology and still refer to this as the Skhorohod~$\sigma$-algebra.

\begin{definition}[Right-preserving random infection path] \label{def_RPRIP}
    Let~$\mathcal H$ be a graphical construction for the contact process. We say that a random element~$\Pi$ of~$\mathcal D([0,*],\mathbb Z)$ is a \emph{right-preserving random infection path (RPRIP) with respect to~$\mathcal H$} if the following conditions are satisfied: 
    \begin{itemize}
        \item $\Pi$ is an infection path in~$\mathcal H$;
        \item for all deterministic paths~$\pi \in \mathcal D([0,*],\mathbb Z)$, the event
        \[
        \{\text{the domain of $\Pi$ is contained in the domain of~$\pi$ and~$\Pi(s) \le \pi(s)$ for all $s$}\}
        \]
        is measurable with respect to the~$\sigma$-algebra generated by~$\mathrm{Left}_\pi(\mathcal H)$.
    \end{itemize}
\end{definition}

\begin{example} \label{example_Rightmost}
    Fix~$t \ge 0$. Define~$\Pi$ as the leftmost infection path started from~$\mathbb N_0$ and reaching time~$t$. That is,~$\Pi$ is the almost surely unique infection path that satisfies~$\Pi(s) \le \gamma(s)$ for every~$s \in [0,t]$ and every infection path~$\gamma:[0,t] \to \mathbb Z$ with~$\gamma(0)\ge 0$ (to check that the leftmost infection path indeed exists, define a total order on infection paths started in~$\mathbb N_0 \times \{0\}$, by saying that~$\gamma$ is smaller than~$\gamma'$ when~$\gamma(\bar{s})<\gamma'(\bar{s})$, where~$\bar{s}$ is the fist time when they are different; then, let~$\Pi$ be the minimal path for this order). Then,~$\Pi$ is an RPRIP. In fact, all RPRIP's we will encounter are essentially variants of this example.
\end{example}

\begin{definition}[Left-preserving random infection path] \label{def_LPRIP}
    Let~$\mathcal H$ be a graphical construction for the contact process. We say that a random element~$\Pi$ of~$\mathcal D([0,*],\mathbb Z)$ is a \emph{left-preserving random infection path (LPRIP) with respect to~$\mathcal H$} if the following conditions are satisfied: 
    \begin{itemize}
        \item $\Pi$ is an infection path in~$\mathcal H$;
        \item for all deterministic paths~$\pi \in \mathcal D([0,*],\mathbb Z)$, the event
        \[
        \{\text{the domain of $\Pi$ is contained in the domain of~$\pi$ and~$\Pi(s) \ge \pi(s)$ for all $s$}\}
        \]
        is measurable with respect to the~$\sigma$-algebra generated by~$\mathrm{Right}_\pi(\mathcal H)$.
    \end{itemize}
\end{definition}

\begin{lemma}[\cite{den2014scaling}] \label{lem_RPRIP_property}
    Let~$\mathcal  H$ be a graphical construction for the contact process and let~$\Pi$ be an RPRIP with respect to~$\mathcal H$. On the same probability space where~$\mathcal H$ is defined, let~$\mathcal H'$ be an independent graphical construction. Then, it follows that:
\begin{equation*}
    \mathrm{Law}\left(\Pi, \mathrm{Left}_\Pi(\mathcal H), \mathrm{Right}_\Pi^+(\mathcal H)\right) = \mathrm{Law}\left(\Pi, \mathrm{Left}_\Pi(\mathcal H), \mathrm{Right}_\Pi^+(\mathcal H')\right)
\end{equation*}
\end{lemma}

For when we treat the contact-and-barrier process, Definitions~\ref{def_LeftRightPlus} and~\ref{def_RPRIP} alongside with Lemma~\ref{lem_RPRIP_property} will suffice for our studies. However, for when we consider the multitype contact process, we will also require the following.

\section{Lemmas for renewal-type processes} \label{section_lemmasRenewal}

\paragraph{} In this section, we state two lemmas concerning renewal processes. The first, Lemma~\ref{lem_renewal}, provides a mechanism for extracting an i.i.d.\ sequence from a sequence of stopping times. Although abstract in formulation, it encapsulates the renewal-time strategy introduced in \cite{kuczek1989central} for the classical contact process. To clarify its role, we include an example immediately after the statement illustrating how it will be applied. The second result, Lemma~\ref{lem_CLT_RenewalProcess}, allows us to derive a central limit theorem from this embedded i.i.d. structure, provided these renewal times occur fast enough.

\begin{lemma}\label{lem_renewal}
	Let~$(\Omega,\mathcal F, \mathbb P)$ be a probability space with a filtration~$(\mathcal F_n)_{n \in \mathbb N_0}$. Let~$(\kappa_n)_{n \in \mathbb N_0}$ be a sequence of stopping times with respect to~$(\mathcal F_n)_n$, and~$(Y_n)_{n \in \mathbb N_0}$ be a stochastic process adapted to~$(\mathcal F_n)_n$ taking values on a measurable space~$(E,\mathcal E)$. Assume that~$(\kappa_n)_n$ and~$(Y_n)_n$ satisfy the following properties:
	\begin{itemize}
		\item[$\mathrm{(i)}$] $\hat \kappa_n := \kappa_n - n \ge 0$ for each~$n$;
		\item[$\mathrm{(ii)}$] if~$m < n$ and~$\kappa_m \ge n$, then~$\kappa_m \ge \kappa_n$;
		\item[$\mathrm{(iii)}$] for every bounded and measurable function~$g$ and every~$n \ge 1$, we have
			\begin{equation} \label{eq_property_iii}
				\begin{split}
					&\mathbb E[g(Y_n,\hat \kappa_n, Y_{n+1}, \hat \kappa_{n+1},\ldots )) \cdot \mathds{1}\{\kappa_n = \infty\} \mid \mathcal F_{n-1}] \\[.1cm]&\hspace{3cm} = \mathbb E[g(Y_0,\hat \kappa_0, Y_1,\hat \kappa_1,\ldots) \cdot \mathds{1}\{\kappa_0 = \infty\}] \text{ a.s.;}
				\end{split}
			\end{equation}
			\item[$\mathrm{(iv)}$] $\mathbb P(\kappa_0 = \infty) > 0$.
	\end{itemize}
	Then, almost surely there are infinitely many values of~$n \in \mathbb N_0$ such that~$\kappa_n = \infty$. Moreover, letting~$N_0 < N_1 < \ldots$ denote these values of~$n$ in increasing order, we have that the random sequences
	\begin{align*}
		(N_1-N_0,Y_{N_0},\ldots, Y_{N_1-1}),\; (N_2-N_1, Y_{N_1},\ldots, Y_{N_2-1}), \ldots
	\end{align*}
	are independent and identically distributed, all with the law of~$(N_1,Y_0,\ldots,Y_{N_1-1})$ conditioned on~$\{\kappa_0=\infty\}$.
\end{lemma}

\begin{figure}[H]
\begin{center}
\setlength\fboxsep{0.1cm}
\setlength\fboxrule{0.01cm}
\fbox{\includegraphics[width=0.7\textwidth]{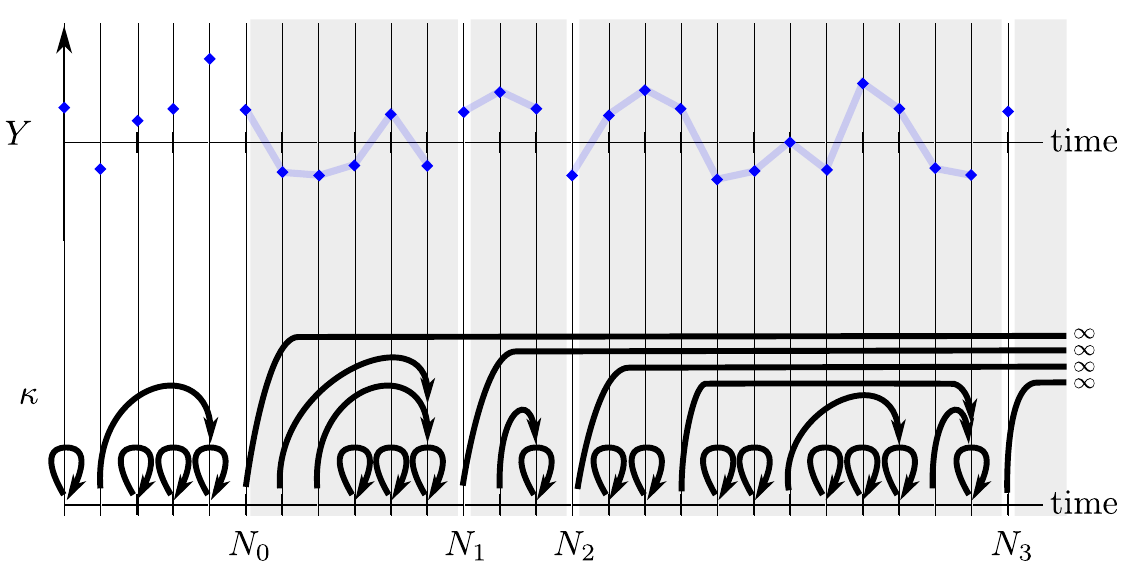}}
\end{center}
	\caption{Schematic representation of the process~$(Y_n)_n$, the stopping times~$(\kappa_n)_n$, and the random times~$(N_k)_k$ from Lemma~\ref{lem_renewal}. In the graph representing~$(\kappa_n)_n$, an arrow from~$m$ to~$n$ indicates that~$\kappa_m=n$.}
\end{figure}

\begin{example}
    Let $(\Omega,\mathcal F,\mathbb P)$ be a probability space on which is defined a one-dimensional classical contact process~$\zeta_t^h$ with parameter~$\lambda>\lambda_c$ started from the configuration $\mathds{1}_{(-\infty,0]}$. For each $t\ge0$, let $\mathcal F_t$ be the natural filtration generated by the process up to time $t$, and recall $R_t^h=\sup\{x\in\mathbb Z:\zeta_t^{h}(x)=1\}$. For each $n\in\mathbb N$, set:
\begin{equation}
\kappa_n=\inf\big\{ s\ge n :
(R_n^h,n)\not\rightsquigarrow (y,s)\ \text{for all } y\in\mathbb Z \big\}.
\end{equation}

In words, $\kappa_n$ is the first time at which all descendants of the rightmost particle at time $n$ have died out. By definition, $\kappa_n\ge n$, and therefore $\hat{\kappa}_n\ge0$, which verifies~(i). Moreover, since $\lambda>\lambda_c$, the process started from a single occupied site survives with positive probability, implying $\mathbb P(\kappa_0=\infty)>0$, which establishes~(iv).

Verification of (ii) requires a brief argument. If $\kappa_m=\infty$, the claim is immediate. Suppose instead that $\kappa_m$ is finite but larger than~$n$. Then, there exists $y\in\mathbb Z$ such that $(R_m^h,m)\rightsquigarrow (y,n)$. Since sites can only give birth to nearest-neighbours, a crossing-paths argument implies $(R_m^h,m)\rightsquigarrow (R_n^h,n)$. If there existed $z\in\mathbb Z$ with $(R_n^h,n)\rightsquigarrow (z,\kappa_m)$, we would obtain $(R_m^h,m)\rightsquigarrow (z,\kappa_m)$ by a concatenation of infection paths, contradicting the definition of $\kappa_m$. Hence $\kappa_n\le \kappa_m$.

Let also $Y_n=(R_s^h)_{n\le s<n+1}$ denote the process recording the trajectory of the rightmost particle during the time interval $[n,n+1)$. The verification of (iii) is more involved and relies on the main idea introduced in \cite{kuczek1989central}. Informally, once the rightmost particle survives indefinitely, all future rightmost particles are its descendants by a crossing-paths argument. Hence, this particle effectively governs the future evolution of the interface. Consequently, the process observed after such a time behaves, in distribution, as a fresh copy of the original process conditioned on the rightmost particle at the initial time also having this property.

\end{example}

\begin{lemma}\label{lem_CLT_RenewalProcess}
    Consider an increasing sequence of non-negative random variables~$\left(\tau^n\right)_{n\in \mathbb N_0}$ defined under some probability measure~$\mathbf{P}$ and assume that~$\tau_0$ is finite almost surely. Associated to this sequence, consider a sequence of real-valued stochastic processes~$(Z^0_t)_{0\leq t<\tau_0}$ and~$(Z^n_t)_{\tau^{n-1}\leq t<\tau^n}$ for~$n\geq 1$. For a given~$t\in[0,\infty)$, let~$X(t)=Z^0_t$ for~$t\in[0,\tau_0)$ and ~$X(t) = Z^{n}_t$ where $n$ is such that~$t\in[\tau^{n}, \tau^{n+1})$ for~$n\geq 1$.
    
    Assume that the following is true:
    \begin{enumerate}
        \item $(Z^0_t)_{0\leq t <\tau^0}$ is independent of~$\left(\tau^n,(Z^n_t)_{\tau^{n-1}\leq t<\tau^n}\right)_{n\geq 1}$
        \item The increments~$\left((\tau^{n}-\tau^{n-1}),(Z^{n}_t)_{\tau^{n-1}\leq t< \tau^{n}}\right)_{n\geq 1}$ are i.i.d.
        \item For all~$n\geq 2$, it follows that:
        \begin{equation*}
            \mathbf P\left(\max\{\tau^{n}-\tau^{n-1},\sup\{|Z^{n}_t - Z^{n}_{\tau_n}|\,:\,\tau^{n-1}\leq t< \tau^{n}\}\}>x\right)\leq Ce^{-cx^p}
        \end{equation*}
        for some choice of constants~$c,C,p>0$
    \end{enumerate}
    
    Then, there exists real constants~$\mu$ and~$\sigma>0$ such that~$\frac{X(t)-\mu t}{\sigma \sqrt{t}}\xrightarrow[t \to \infty]{\mathrm{(dist)}} \mathcal{N}(0,1)$ as~$t\rightarrow\infty$. 
\end{lemma}

\section{Contact-and-barrier process} \label{sec_CBP}

\paragraph{}This section contains the most novel contribution of this work, which is to construct the interface process through what we call the \emph{patchwork construction}. This is a step-by-step procedure in which the process is built up to a certain stopping time, after which we sew together successive pieces to obtain the full interface process for all times. The main advantage of this approach is that it leads to the definition of an observable called \emph{depth}, which depends on this construction. This, in turn, enables us to identify an embedded i.i.d. sequence of times at which the increments of the interface position are themselves i.i.d.

This section is organized as follows. In Section~\ref{GraphicalConstructionCBP}, we introduce the contact-and-barrier process via an augmented graphical construction. Section~\ref{sec_Subadditivity} proves Proposition~\ref{prop_real_speed_barrier} and establishes regularity properties of the barrier’s asymptotic speed. Section~\ref{section_PatchworkCBP} contains the core argument, where we develop the patchwork construction and identify an embedded i.i.d.\ structure in time and space. These ingredients culminate in the proofs of Theorems~\ref{thmCBP_Speed}, \ref{ThmCBP_Tightness}, and \ref{ThmCBP_CLT_Speed}, presented in Section~\ref{sec_proofsThmsCBP}.

\subsection{Augmented graphical construction} \label{GraphicalConstructionCBP}

\paragraph{} In this section, we describe a Poisson construction for the contact-and-barrier process. For the classical contact process, one defines the notion of infection path (which depends on the Poisson processes in the graphical construction, but \emph{not} on the initial configuration), and then uses these infection paths in conjunction with the initial configuration to define the process, by means of the prescription
\[\zeta_t(x) = \mathds{1}\{\exists y \in \mathbb Z:\; \zeta_0(y)=1,\; (y,0) \rightsquigarrow (x,t) \},\quad x \in \mathbb Z,\; t \ge 0.\]

The construction we are about to give for the contact-and-barrier process will rely on the same family~$\mathcal H$ as the one for the classical process, and in addition, on an \emph{flight plan}~$\mathcal I$, describing the (attempted) jumps of the barrier (Definition~\ref{def_BfreeInfPath} below). We will use the expression \emph{augmented graphical construction} to refer to the pair~$(\mathcal H, \mathcal I)$. Infection paths of~$\mathcal H$  will not always be effective in carrying the infection, because they may overlap with the barrier. With this in mind, we will introduce a notion of \emph{barrier-free infection path}, which is an infection path that carries the infection by avoiding the barrier. It would then be tempting to mimic the above formula by writing
\begin{equation}\label{eq_posteriori}\begin{split}
&\beta_t(x)=1 \text{ if and only if }\\&\exists y \in \mathbb Z:\; \beta_0(y)=1 \text{ and $\exists$ a barrier-free infection path from $(y,0)$ to $(x,t)$}.
\end{split}
\end{equation}

While this formula will turn out to be correct \textit{a posteriori}, it cannot be used to \emph{construct} the process. The issue is one of circularity: in order to know whether a path is barrier-free, we need to ask whether it ever overlaps with the barrier, which in turn requires the process to already be constructed. To circumvent this issue, we construct the process in time steps, delimited by the times of attempted jumps of the barrier.

We start by defining the {flight plan} for the barrier. This is a collection~$\mathcal I = (J_{\leftarrow},(J_{\rightarrow},\mathfrak{M}))$, where:
\begin{itemize}
    \item $J_\leftarrow$ and~$J_\rightarrow$ are independent Poisson point process on~$[0,\infty)$ with intensities~$r_\leftarrow$ and~$r_{\rightarrow}:=\max\{r^0_\rightarrow, r^1_\rightarrow\}$, respectively;
    \item $\mathfrak{M}=(\mathfrak{M}(t):t \in J_\rightarrow)$ are random marks assigned to arrivals of~$J_\rightarrow$, where marks are independent and uniformly distributed on~$[0,1]$.
\end{itemize}

Their effect will be:
\begin{itemize}
	\item for each~$t \in J_{\leftarrow}$, the barrier jumps to the left at time~$t$; 
	\item for each~$t \in J_{\rightarrow}$ such that the site at the right of the barrier is empty at time~$t-$, the barrier jumps to the right at time~$t$ if~$\mathfrak{M}(t) \le r^0_\rightarrow/r_\rightarrow$; 
	\item for each~$t \in J_{\rightarrow}$ such that the site at the right of the barrier is occupied at time~$t-$, the barrier jumps to the right if~$\mathfrak{M}(t) \le r^1_\rightarrow/r_\rightarrow$ (overwriting the particle that was occupying that site).
\end{itemize}

Consider also~$\mathcal{H}$ a graphical construction for the contact process with rate~$\lambda$ on~$\mathbb{Z}\times[0,\infty)$. We will now construct the process for a given initial configuration~$\beta_0\in \mathcal{B}$. Let~$\mathcal{J} = J_{\leftarrow}\cup J_{\rightarrow}$ be the potential jump times for the barrier. Let~$\sigma_0:=0$ and inductively, define~$\sigma_{k+1}:=\inf\{t> \sigma_k\,:\,t\in\mathcal{J}\}$ 

Assume that the process is built until time~$\sigma_k$ for some~$k\in\mathbb{N}_0$. 
We define the process on~$(\sigma_k,\sigma_{k+1})$ by letting the barrier stand still and letting the contact process evolve on the right of it, so that we set, for each~$s \in [\sigma_k,\sigma_{k+1})$:
\begin{equation*}
    \beta_s(x) = 
    \begin{cases}
        -1 & \text{if } x = B_{\sigma_k} \\
        0 & \text{if } x < B_{\sigma_k} \\
        1 & \text{if } x > B_{\sigma_k} \text{ and there exists an infection path } \gamma: [\sigma_k, s] \rightarrow \mathbb{Z} \text{ s.t.} \\
          & \qquad \beta_{\sigma_k}(\gamma(\sigma_k)) = 1, \ \gamma(s) = x, \text{ and } \gamma(u) > B_{\sigma_k} \text{ for all } u \in [\sigma_k, s] \\
        0 & \text{otherwise}
    \end{cases}
\end{equation*}

To define the process at time~$\sigma_k$, we separate it into three possibilities:

\begin{itemize}
    \item If~$\sigma_k\in J_{\leftarrow}$, then the barrier jumps to the left, that is, we let:
    \begin{eqnarray}
        \beta_{\sigma_k}(x) = \begin{cases}
            -1 &\text{ if } x=B_{\sigma_k-}-1; \\
            0 &\text{ if }x=B_{\sigma_k-} \text{ or }x\leq B_{\sigma_k-}-2; \\
            \beta_{\sigma_k-}(x) &\text{ if }x\geq B_{\sigma_k-}+1.
        \end{cases}
    \end{eqnarray}
    \item If either~$[\sigma_k\in J_{\rightarrow},\;\beta_{\sigma_k-}(B_{\sigma_k-}+1)=0$,\;~$\mathfrak{M}(\sigma_k) \le r^0_\rightarrow/r_\rightarrow$] or~[$\sigma_k\in J_{\rightarrow},\;\beta_{\sigma_k-}(B_{\sigma_k-}+1)=1,\;\mathfrak{M}(\sigma_k) \le r^1_\rightarrow/r_\rightarrow$], then the barrier jumps to the right, that is, we let:
    \begin{eqnarray}
        \beta_{\sigma_k}(x) = \begin{cases}
            0 &\text{ if }x\leq B_{\sigma_k-}; \\
            -1 &\text{ if }x=B_{\sigma_k-}+1; \\
            \beta_{\sigma_k-}(x) &\text{ if }x\geq B_{\sigma_k-}+2.
        \end{cases}
    \end{eqnarray}
\end{itemize}

This constructs the process for the time interval~$[0,\sigma_{k+1}]$, so recursively, it makes it well-defined for all times.

\begin{definition}[Barrier-free infection path]\label{def_BfreeInfPath}
Let~$\beta_0 \in \mathcal B$ and assume that the contact-and-barrier process~$(\beta_t)_{t \ge 0}$ starts from~$\beta_0$ and is obtained from the augmented graphical construction~$(\mathcal H, \mathcal I)$ introduced above. 
A \emph{barrier-free infection path} for this process is an infection path~$\gamma:[s,t] \to \mathbb Z$ for~$\mathcal H$ such that~$\gamma(r) \neq B_r$ for all~$r \in [s,t]$. 
\end{definition}

Since~$\beta_0\in\mathcal B$, one could replace the condition~$\gamma(r)\neq B_r$ for all~$r\in[s,t]$ in Definition~\ref{def_BfreeInfPath} by the condition~$B_r <\gamma(r)$ for all~$r\in[s,t]$. With this observation, it is easy to see the following: if~$\gamma:[s,t]\rightarrow\mathbb{Z}$ is a barrier-free infection path and~$\gamma':[s,t]\rightarrow\mathbb{Z}$ is an infection path completely contained on~$\mathrm{Right}(\gamma)$ (i.e., such that~$\gamma(u)\leq \gamma'(u)$) for all~$u\in [s,t]$), then~$\gamma'$ is also a barrier-free infection path. 

For the rest of this subsection, we assume that the contact-and-barrier process is obtained from an augmented graphical construction~$(\mathcal H, \mathcal I)$.

\begin{lemma}\label{OccupiediffActivePath}
    For every~$s<t$ and every~$x \in \mathbb Z$, we have~$\beta_t(x)=1$ if and only if there exists~$y \in \mathbb Z$ such that~$\beta_s(y)=1$ and there exists a barrier-free infection path from~$(y,s)$ to~$(x,t)$. In particular,~\eqref{eq_posteriori} holds.
\end{lemma} 

\begin{proof}
This is readily checked by induction: we assume that~\eqref{eq_posteriori} has been proved for all~$s,t$ with~$s < t \le \sigma_k$ and all~$x$, and using concatenation of barrier-free paths from time~$0$ to~$\sigma_k$ with barrier-free paths from~$\sigma_k$ to~$t \in (\sigma_k,\sigma_{k+1}]$, we obtain the induction step.
\end{proof}

The following definition introduces a partial order on the space of configurations for the contact-and-barrier processes. In the subsequent lemma, we demonstrate that this order is preserved under the dynamics.

\begin{definition}
    Let~$\beta, \beta'\in\mathcal B$, and let~$x$ and~$x'$ be the locations of the barrier in~$\beta$ and~$\beta'$, respectively. We say that~$\beta \lesssim\beta '$ if~$\{y: \beta(y)=1\} \subseteq \{y: \beta'(y)=1\}$ and~$x \ge x'$. 
\end{definition}

\begin{lemma}\label{prop_coupling}
Assume that~$r_{\rightarrow}^1\leq r_{\rightarrow}^0$. Let~$\beta, \beta' \in \mathcal B$ with~$\beta \lesssim \beta'$. Let~$(\beta_t)_{t\geq 0}$ and~$(\beta_t')_{t\geq 0}$ be contact-and-barrier processes constructed using the same augmented graphical construction started from~$\beta$ and~$\beta'$, respectively. Then,~$\beta_t\lesssim \beta_t'$ for all~$t\geq 0$.
\end{lemma}

\begin{proof}
We denote by~$(B_t)_{t \ge 0}$ and~$(B_t')_{t \ge 0}$ the processes describing the positions of the barriers in~$(\beta_t)$ and~$(\beta_t')$, respectively.

We argue by induction on~$k$ to show that the~$\lesssim$ relation is maintained up to time~$\sigma_k$. For~$k=0$, this is given by the assumption. Assume that we have proved it for~$k$. To prove that~$\beta_t \lesssim \beta_t'$ for all~$t \in (\sigma_k,\sigma_{k+1})$, we observe that there is no motion of the barriers in this time interval (so~$B_t \ge B_t'$), and we obtain the inclusion~$\{x:\beta_t(x)=1\} \subseteq \{x:\beta_t'(x)=1\}$ by considering infection paths of~$\mathcal H$ restricted to~$[\sigma_k,\sigma_{k+1})$. 

To check that~$B_{\sigma_{k+1}} \ge B_{\sigma_{k+1}}'$, the only case that needs a moment's thought is when
\[B_{\sigma_k}=B_{\sigma_k}',\quad \beta_{\sigma_{k+1}-}(B_{\sigma_k}+1) = 0,\quad \beta_{\sigma_{k+1}-}'(B_{\sigma_k}+1) = 1,\]
and the jump instruction at time~$\sigma_{k+1}$ is to the right. In these circumstances, recalling that~$r_\rightarrow = \max\{r^0_\rightarrow,r^1_\rightarrow\} = r^0_\rightarrow$, we have that:
\begin{itemize}
    \item if~$\mathfrak{M}(\sigma_{k+1}) \le r^1_\rightarrow/r^0_\rightarrow$, then both barriers jump to the right at time~$\sigma_{k+1}$;
    \item otherwise, the barrier of~$(\beta_t)$ jumps to the right at time~$\sigma_{k+1}$, while the barrier of~$(\beta_t')$ does not move.
\end{itemize}

Having established~$B_{\sigma_{k+1}} \ge B_{\sigma_{k+1}}'$, we obtain~$\{x:\beta_{\sigma_{k+1}}(x)=1\} \subseteq \{x:\beta'_{\sigma_{k+1}}(x)=1\}$ by taking a limit as~$t \nearrow \sigma_{k+1}$.
\end{proof}

We will also require an additional auxiliary result concerning the evolution of two contact-and-barrier processes, started from possibly different initial configurations but agreeing on a region ahead of the barrier.

\begin{lemma} \label{lem_SameInterface}
Let~$(\beta_t^1)_{t\geq 0}$ and~$(\beta_t^2)_{t\geq 0}$ be two contact-and-barrier processes started from initial configurations~$\beta_1, \beta_2 \in \mathcal B$, constructed using the same augmented graphical construction~$(\mathcal H, \mathcal I)$. Let~$B_t^1$ and~$B_t^2$ denote the respective positions of the barriers at time~$t$. Suppose there exist integers~$b < x$ such that:
\begin{itemize}
    \item $\beta_1$ and~$\beta_2$ coincide on the interval $[b, x]$, i.e.,~$\beta_1 \mathds{1}_{[b, x]} = \beta_2 \mathds{1}_{[b, x]}$;
    \item $\beta_1(b) = -1$, i.e., there is a barrier at site~$b$;
    \item $\beta_1(x) = 1$, i.e., site~$x$ is occupied by a particle.
\end{itemize}

Fix any~$t \geq 0$, and for each~$s \in [0, t]$, define
\begin{align*}
r_s^x := \sup\left\{ y \in \mathbb{Z} :\ 
\begin{array}{l}
\text{either for the process started from~$\beta_1$ or from~$\beta_2$, there exists} \\
\text{a barrier-free infection path in that process from } (x, 0) \text{ to } (y, s)
\end{array}
\right\}.
\end{align*}
with the convention that the supremum of the empty set is~$-\infty$. If~$r_t^x \in \mathbb{Z}$, then for all~$s \in [0, t]$, we have
\begin{equation} \label{eq_propertyInterfaceDifferentConfig}
    B_s^1 = B_s^2 \quad \text{and} \quad \beta_s^1(y) = \beta_s^2(y) \quad \text{for all } y \in \{B_s^1, \dots, r_s^x\}.
\end{equation}
\end{lemma}

One may argue by contradiction, letting $P\in[0,t]$ be the first time at which~\eqref{eq_propertyInterfaceDifferentConfig} fails. The conclusion then follows by a straightforward argument, which we omit.

\subsection{Subadditivity} \label{sec_Subadditivity}

\paragraph{}The end goal of this subsection is to prove Proposition~\ref{prop_real_speed_barrier}. Under the assumption that~$r^0_\rightarrow \ge r^1_\rightarrow$, we will be able to use the subadditive ergodic theorem to prove that the barrier has an asymptotic speed. 

Recall the Heaviside configuration where the barrier is placed at the origin and all sites to the right of it are declared to be occupied. The following result guarantees that the barrier has a speed for the process started from that configuration:

\begin{proof}[Proof of Proposition~\ref{prop_real_speed_barrier}]
    Consider a contact-and-barrier process~$(\beta_t)_{t\geq 0}$ started from the Heaviside configuration and constructed using an augmented graphical construction~$(\mathcal H,\mathcal I)$. 
    
    Set~$X_{0,t}:=-B_t$ for~$t \ge 0$. For~$s \ge 0$, define the auxiliary process~$(\Tilde{\beta}^s_u)_{u\geq s}$ as follows: let~$\Tilde{\beta}^s_s$ be defined by modifying~$\beta_s$, so that all the sites to the right of the barrier are artificially changed to state~1, and then let the process evolve for times~$u \ge s$ as a contact-and-barrier process, using the same augmented graphical construction as before (restricted to the time interval~$[s,\infty)$). Set~$X_{s,t}:=B_s-\Tilde{B}^s_t$ for all~$t \ge s$.
    
Restricting to integer times, it is straightforward to verify that $\{X_{m,n}: m\le n\}$ satisfies assumptions~(a)–(e) of the Subadditive Ergodic Theorem (Theorem~2.6 in \cite{liggett1985interacting}). To make the transition from convergence along integer times to convergence along real times, we can proceed as in the proof of Theorem~2.19 in~\cite{liggett1985interacting}, using the Borel-Cantelli lemma and the bound~$\mathbb P(\max_{n\leq t \leq n+1}|B_{t}-B_n|>\epsilon n)\leq \mathbb P(Z > \epsilon n)$, where~$Z \sim \mathrm{Poisson}(r_\leftarrow + r_\rightarrow)$. 
\end{proof}

\begin{lemma}\label{lem_bar_Behaves_Well}
Assume that~$r^1_\rightarrow\leq r^0_\rightarrow$, and let~$\mathbf{B}$ be the constant given in Proposition~\ref{prop_real_speed_barrier}. Assume that the contact-and-barrier process~$(\beta_t)_{t\geq 0}$ is started from an arbitrary configuration~$\beta_0 \in \mathcal B$, and let~$(B_t)_{t \ge 0}$ denote the barrier process. 
Then, for all~$\epsilon>0$, there exist~$c,C>0$ such that
    \begin{equation}\label{eq_BarrierRealSpeedWell}
        \mathbb P\left(B_t \geq \left(\mathbf{B}-\epsilon\right)t\right) \geq 1-Ce^{-ct } \quad \text{for all }t > 0.
    \end{equation}
\end{lemma}
\begin{proof}
By monotonicity (using Lemma~\ref{prop_coupling}), it suffices to prove the statement under the assumption that the contact-and-barrier process starts from the Heaviside configuration.
Since~$B_t/t \xrightarrow{t \to \infty} \mathbf B$ in~$L^1$, we can choose~$r$ large enough that~$\mathbb E[B_{r}] \ge (\mathbf B- \tfrac{\epsilon}{4}) r$.

For each~$s \ge 0$, we consider the same auxiliary process~$(\Tilde{\beta}_u^s)_{u \ge s}$ that was introduced in the proof of Proposition~\ref{prop_real_speed_barrier}. For each~$u \ge s$, let~$\Tilde{B}^s_u$ denote the position of the barrier in~$\Tilde{\beta}^s_u$. Noting that~$\Tilde{B}^{kr}_{(k+1)r} \le B_{(k+1)r}$ for all~$k \in \mathbb N_0$, we bound
\begin{align*}
    \mathbb P(B_{rm} \le (\mathbf B - \tfrac{\epsilon}{2})rm) &= \mathbb P\left(\sum_{k=1}^m (B_{kr}-B_{(k-1)r}) > (\mathbf B - \tfrac{\epsilon}{2})rm \right) \\
    &\le \mathbb P\left(\sum_{k=1}^m (\Tilde{B}^{(k-1)r}_{kr}-B_{(k-1)r}) \le (\mathbf B - \tfrac{\epsilon}{2})rm \right).
\end{align*}
The random variables~$(\Tilde{B}^{kr}_{(k+1)r}-B_{kr})_{k\in \mathbb N_0}$ are independent and identically distributed, all with expectation equal to~$\mathbb E[B_r] > (\mathbf B - \tfrac{\epsilon}{4})r$.  They also have finite exponential moments, since they are bounded by the number of attempted jumps of the barrier in a time interval of length~$r$. Hence, a large deviations bound yields
\[\mathbb P(B_{rm} \le (\mathbf B - \tfrac{\epsilon}{2})rm)  \le Ce^{-cm}\]
for some~$c,C > 0$ and all~$m \in \mathbb N$.

Next, if~$t \ge 0$, we can bound
\begin{align*}
    \mathbb P(B_t \le (\mathbf B - \epsilon)t) &\le \mathbb P(B_{r \lfloor t/r \rfloor} \le (\mathbf B - \tfrac{\epsilon}{2})r \lfloor t/r\rfloor)\\ &\quad + \mathbb P(B_t - B_{r\lfloor t/r \rfloor} > (\mathbf B - \epsilon)t - (\mathbf B - \tfrac{\epsilon}{2}) r \lfloor t/r\rfloor)\\
    &\le Ce^{-c\lfloor t/r\rfloor} + \mathbb P(B_t - B_{r\lfloor t/r \rfloor} > (\mathbf B - \epsilon)t - (\mathbf B - \tfrac{\epsilon}{2}) r \lfloor t/r\rfloor)
\end{align*}
Now, if~$t$ is large enough, we have~$|(\mathbf B - \epsilon)t - (\mathbf B - \tfrac{\epsilon}{2}) r \lfloor t/r\rfloor| \ge \tfrac{\epsilon}{4}t$, so the probability on the right-hand side above can be bounded by another factor of the form~$Ce^{-ct}$, again by comparing the barrier displacement with a Poisson random variable. The desired bounds now follow from adjusting the constants~$C,c$, so that small values of~$t$ are also covered.
\end{proof}

\subsection{Patchwork construction of interface process} \label{section_PatchworkCBP}

\paragraph{}This subsection is the core of this work, as it contains all the elements necessary for the patchwork construction and the auxiliary results concerning them. Given its high level of complexity, we begin with a brief overview to provide a clearer picture of what will follow.

In Section~\ref{subsec_trailInter}, we introduce a method for constructing initial configurations in $\mathcal{C}$ for the multitype contact process. This method relies on the concept of trails, which are marks on $\mathbb{Z} \times (-\infty,0]$ that give rise to particles. Once the initial setting is established, Section~\ref{patchworkElements} defines the main objects needed for the patchwork construction: a stopping time $T$, a special random infection path $\Gamma$, the observable called depth $D$, and the interface process $(I_t)_{0 \leq t \leq T}$. We state some of their key properties, which are then proved in Section~\ref{sss_proofs}.

These objects enable the construction of a single patch of the patchwork, defining the interface up to a certain special stopping time. Extending this to all times requires sewing these patches together, which is done in Section~\ref{subsub_sewingCBP}. We then establish the renewal structure in Section~\ref{subsec_renewalsCBP}, making use of the previously defined depth observable. Finally, Section~\ref{ss_coupled} is devoted to the proof of an important technical lemma.

\subsubsection{Trails and interface measures} \label{subsec_trailInter}

\paragraph{} We begin by defining a \emph{trail}, a subset of the space $\mathbb{Z} \times (-\infty, 0]$, interpreted as a discrete set of marked points (or ``footprints'') on this lattice-halfplane.

\begin{definition}[Trails] \label{def_trail}
Let~$\mathscr A$ be the collection of all sets~$A \subseteq (-\infty,0]$ such that~$\mathds{1}_A$ is c\`adl\`ag. Define
\[
\mathcal R:= \left\{\begin{array}{l}
\mathrm g \in \mathrm P\left(\mathbb Z \times (-\infty, 0]\right)\,:\, \mathrm g \text{ is of the form }\mathrm g = \bigcup_{x \in \mathbb Z} (\{x\} \times A_x), \\ \text{with } A_x \in \mathscr A \text{ for all }x,  \text{ and moreover, } (1,0) \in \mathrm g
\end{array}\right\}.
\]
\end{definition}
We endow~$\mathscr A$ with the Skhorohod~$\sigma$-algebra, and then~$\mathcal R$ with the infinite product~$\sigma$-algebra.

\begin{definition}
    For a c\`adl\`ag function~$\gamma:(-\infty,t]\rightarrow\mathbb Z$ with~$t > 0$, let~$\mathrm{g}\sqcup\gamma$  be the subset of~$\mathbb Z\times (-\infty,0]$ given by:

\begin{equation*}
    \mathrm{g}\sqcup\gamma = \{(x-\gamma(t)+1,s-t)\,:\,(x,s)\in\mathrm{g}\cup\overline{\mathrm{Graph}(\gamma)}\}
\end{equation*}
where~$\mathrm{Graph}(\gamma) = \{(\gamma(s),s)\,:\,s\in(-\infty,t]\}$. 
\end{definition}

In words,~$\mathrm{g}\sqcup \gamma$ is the subset of~$\mathbb Z \times (-\infty, 0]$ obtained by appending the closure of the graph of~$\gamma$ to the trail~$\mathrm{g}$ and translating this set so that the point~$(\gamma(t),t)$ now occupies the position of~$(1,0)$. In that way, it is clear that~$\mathrm{g}\sqcup\gamma\in\mathcal{R}$. 

Given~$\mathrm g \in \mathcal R$ and a graphical construction~$\mathcal H$  of the contact process on~$\mathbb Z \times (-\infty, 0]$, define~$\beta \in \{-1,0,1\}^\mathbb{Z}$ by setting
\begin{equation}\label{eq_def_mu_g}
\beta(x)=\begin{cases}
0&\text{if } x < 0;\\
-1&\text{if } x =0;\\
\mathds{1}_{(\{-\infty\}\cup \mathrm g)\rightsquigarrow (x,0)}&\text{if } x > 0.
\end{cases}
\end{equation}
where for~$A \subseteq \mathbb Z \times \mathbb R$ and~$(x,t) \in \mathbb Z \times \mathbb R$, we will write~$\{-\infty\} \cup A \rightsquigarrow (x,t)$ to denote the event that either~$-\infty \rightsquigarrow (x,t)$ or~$A \rightsquigarrow (x,t)$.

In words,~$x > 0$ is set to state 1 either if there is an infection path from~$-\infty$ ending at~$(x,0)$, or if there is some~$(y,t) \in \mathrm g$ so that~$(y,t)$ is connected to~$(x,0)$ by an infection path. If neither of these things happen, then~$x$ is set to state 0. Note that, since~$(1,0) \in \mathrm g$ and~$(1,0) \rightsquigarrow (1,0)$, we have~$\beta(1)=1$.

\begin{figure}[ht]
    \centering
    \begin{subfigure}[b]{0.48\textwidth}
        \centering
        \includegraphics[width=1.05\linewidth]{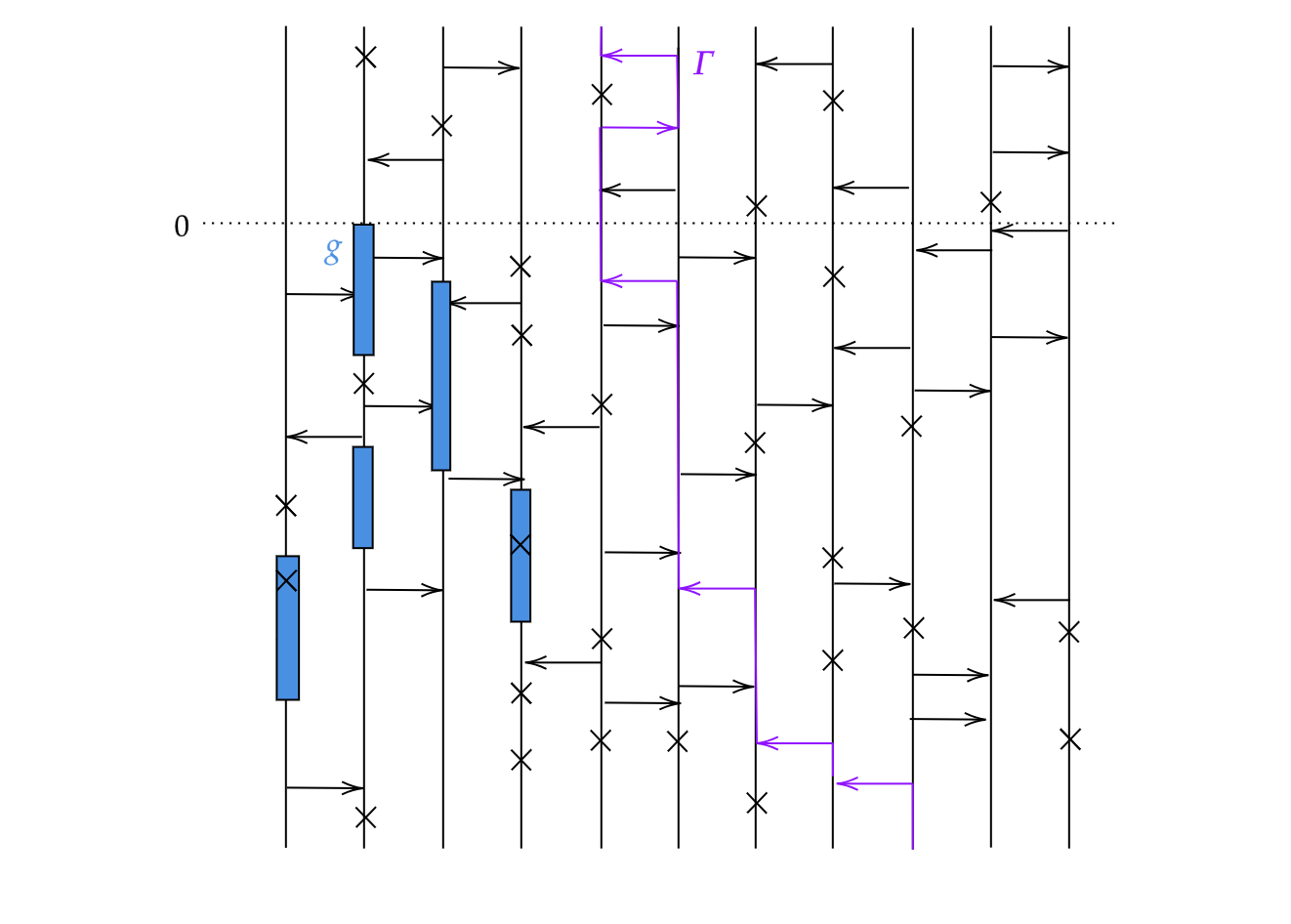}
        \caption{Example of a trail~$\mathrm{g}$ (in blue) and an infection path~$\Gamma$ (in purple), with the trail drawn on top of the graphical construction in which~$\Gamma$ appears as an infection path.}
        \label{fig:first}
    \end{subfigure}
    \hspace{0.012\textwidth}
    \begin{subfigure}[b]{0.48\textwidth}
        \centering
        \includegraphics[width=1.05\linewidth]{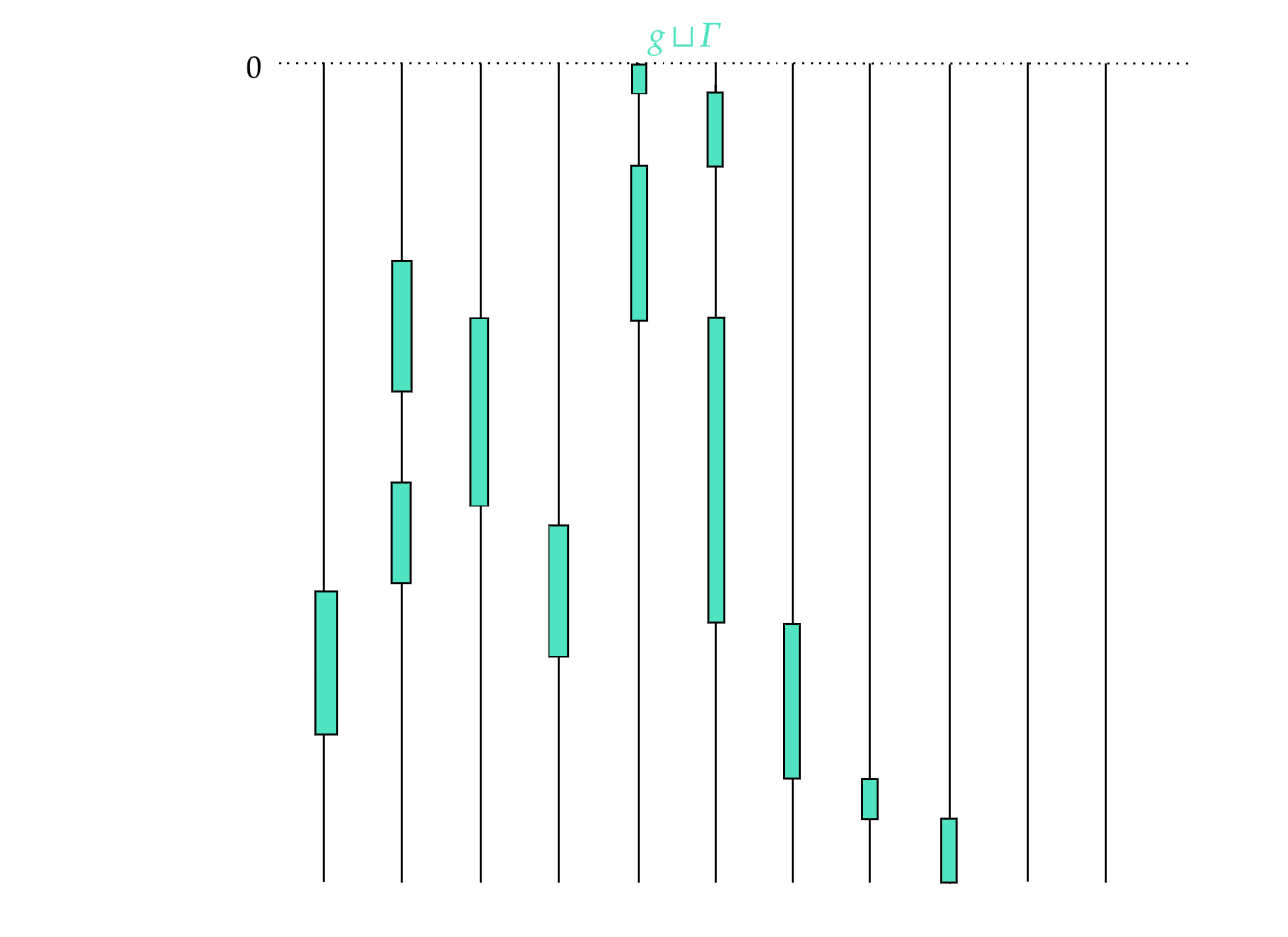}
        \caption{Illustration of the trail~$\mathrm{g}\sqcup\Gamma$ (in turquoise). The graphical construction is discarded at this stage, and a recentralisation from Figure~\ref{fig:first} is required to obtain~$\mathrm{g}\sqcup\Gamma$. }
        \label{fig:second}
    \end{subfigure}
    \caption{Figure~\ref{fig:first} illustrates both the trail~$\mathrm{g}$ and the infection path~$\Gamma$ over the graphical construction, while Figure~\ref{fig:second} shows only the trail~$\mathrm{g}\sqcup\Gamma$. We emphasize that a trail itself does not depend on any graphical construction, as it is simply a subset of~$\mathbb{Z}\times(-\infty,0]$.}
    \label{fig:two_figs}
\end{figure}

\begin{definition}
Given~$\mathrm g\in \mathcal R$, we let~$\mu_{\mathrm g}$ denote the law of the configuration~$\beta$ obtained from~$\mathrm g$ as in~\eqref{eq_def_mu_g}.   
\end{definition}

\begin{definition}[Law~$P^0_{\beta}$ of interface process]\label{def_P0} We let~$P^0_{\beta}$ be the distribution of the interface process~$(I_t)_{t \ge 0}$ for the contact-and-barrier process~$(\beta_t)_{t \ge 0}$ started from~$\beta_0=\beta$.
\end{definition}

\begin{definition}[Law~$P_{\mathrm g}$ of interface process induced by a trail] \label{def_P} We let~$P_{\mathrm g}$ be the distribution of the interface process~$(I_t)_{t \ge 0}$ for the contact-and-barrier process~$(\beta_t)_{t \ge 0}$ started from a random configuration~$\beta_0$ with distribution~$\mu_{\mathrm g}$, so that
\begin{equation}
    \label{eq_P0_and_P}
    P_{\rmg} = \int P^0_{\beta}\; \mu_{\rmg}(\mathrm{d}\beta).
\end{equation}
\end{definition}

\subsubsection{Patchwork elements} \label{patchworkElements}

\paragraph{} Our goal is to give a construction of~$P_{\mathrm g}$ by ``sewing'' together pieces of trajectory, in a patchwork scheme. Fix~$\mathrm g \in \mathcal R$. We work under a probability measure~$\mathbb P$ under which we have defined a graphical construction~$(\mathcal H, \mathcal I)$ of the contact-and-barrier process; we take~$\mathcal H$ defined for both negative and positive times. We assume~$\beta_0 \sim \mu_{\mathrm g}$ is obtained from~$\mathcal H$ and~$\mathrm g$ as in~\eqref{eq_def_mu_g}. We let~$(\beta_t)_{t\ge 0}$ be the contact-and-barrier process started from~$\beta_0$ and constructed from~$(\mathcal H, \mathcal I)$.

We will define a few of central concepts associated with this process. We will highlight the most important of these below: the \emph{adjacency time}~$T$ (Definition~\ref{defT}), the \emph{depth}~$D$ (Definition~\ref{DefDepth}), and the \emph{special infection path}~$\Gamma$ (Definition~\ref{GammaT}). We emphasize that all these objects (as well as some of the other auxiliary ones we will define along the way) will depend on the trail~$\mathrm g$ we have fixed, even though we do not incorporate this in the notation.

In order to concentrate many important definitions close together, we postpone the proofs of all statements that appear in this subsection to the next subsection.

We let
\begin{equation}\label{eq_def_of_mathcal_S}
\mathcal S:= \{x \in \mathbb N: \; -\infty \rightsquigarrow (x,0)\} \subseteq \{x \in \mathbb N: \; \beta_0(x) = 1\},\end{equation}
where the inclusion follows from~\eqref{eq_def_mu_g}. Also, for~$x\in\mathcal S$, let
\begin{equation}\label{eq_def_sigma_x}
\mathcal T_x := \inf\left\{\begin{array}{l} t \ge 1:\; \text{in $(\beta_t)_{t \ge 0}$ there is a barrier-free}\\ \text{infection path from $(x,0)$ to~$(B_t+1,t)$} \end{array}\right\}
\end{equation}
In words,~$\mathcal T_x$ is the first time~$t \ge 1$ at which the site to the right of the barrier is occupied by a particle which descends, by a barrier-free infection path, from the particle at~$x$ at time~0. We then let
\begin{equation}\label{eq_def_bold_X}
\mathbf X := \inf\{x \in \mathcal S:\; \mathcal T_x < \infty\}. 
\end{equation}
We will soon prove that~$\mathbf X < \infty$ almost surely. Using an argument involving the crossing of paths, it is easy to see that 
\begin{equation}\label{eq_minimality_sigma_x}\mathcal T_{\mathbf{X}} = \min\{ \mathcal T_x: \; x \in \mathcal S\}.\end{equation}
\begin{definition}[Adjacency time $T$]\label{defT} Define~$T := \mathcal{T}_{\mathbf X}$.
\end{definition}
Clearly,~$T$ is a stopping time with respect to the filtration generated by the augmented graphical construction. Regarding~$\mathbf X$ and~$T$, we state the following result:


\begin{lemma} \label{TBehavesWell}
    Let~$\alpha$ denote the speed of the contact process with rate~$\lambda$, as in~\eqref{speedCP}.
Assume that either of conditions~\eqref{eq_A1} or~\eqref{eq_A2} is satisfied. 
    Then, there exist constants $c,C > 0$ (independent of the trail~$\rmg$) such that:
    \begin{equation*}
        \mathbb P( \max\{\mathbf X, T \} > s)\leq Ce^{-cs}\quad\text{ for all }s > 0.
    \end{equation*}
\end{lemma}
The proof is given in Section~\ref{sss_proofs}.

\begin{definition}[Depth]\label{DefDepth}
    Define the \emph{depth} as the random variable
    \[
    D:=\sup\{t \ge 0:\; \mathbb Z \times \{-t\} \rightsquigarrow (\{1,\ldots,\mathbf X\}\backslash \mathcal S) \times \{0\}\}, 
    \]
    with~$D = - \infty$ in case~$\{1,\ldots, \mathbf X\}\backslash \mathcal S = \varnothing$.
\end{definition}

\begin{lemma}   \label{DepthBehavesWell} 
    There exists~$C > 0$ (independent of the trail~$\mathrm g$) such that
    \begin{equation*}
        \mathbb P\left(D>d\right)\leq Ce^{-\sqrt{d}}\quad \text{for all }d > 0.
    \end{equation*} 
\end{lemma}
The proof is again postponed to Section~\ref{sss_proofs}. At last, we also define the following random infection path:

\begin{definition}[The special infection path~$\Gamma$]\label{GammaT}
Let~$\Gamma:(-\infty,T]$ be the infection path defined as follows:
\begin{itemize}
    \item in~$(-\infty,0]$,~$\Gamma$ is the leftmost infection path from~$-\infty$ to~$(\mathbf X,0)$;
    \item in~$[0,T]$,~$\Gamma$ is the leftmost barrier-free infection path from~$(\mathbf X,0)$ to~$(B_T+1,T)$.
\end{itemize}
\end{definition}

An illustration of the patchwork elements defined until this point is done in Figure~\ref{fig_patchworkElemets}. The next proposition allow us to somehow re-sample the configuration at moment~$T$ given the information generated by the objects we have defined.

\begin{figure}[H]
    \centering
    \includegraphics[width=0.7  \linewidth]{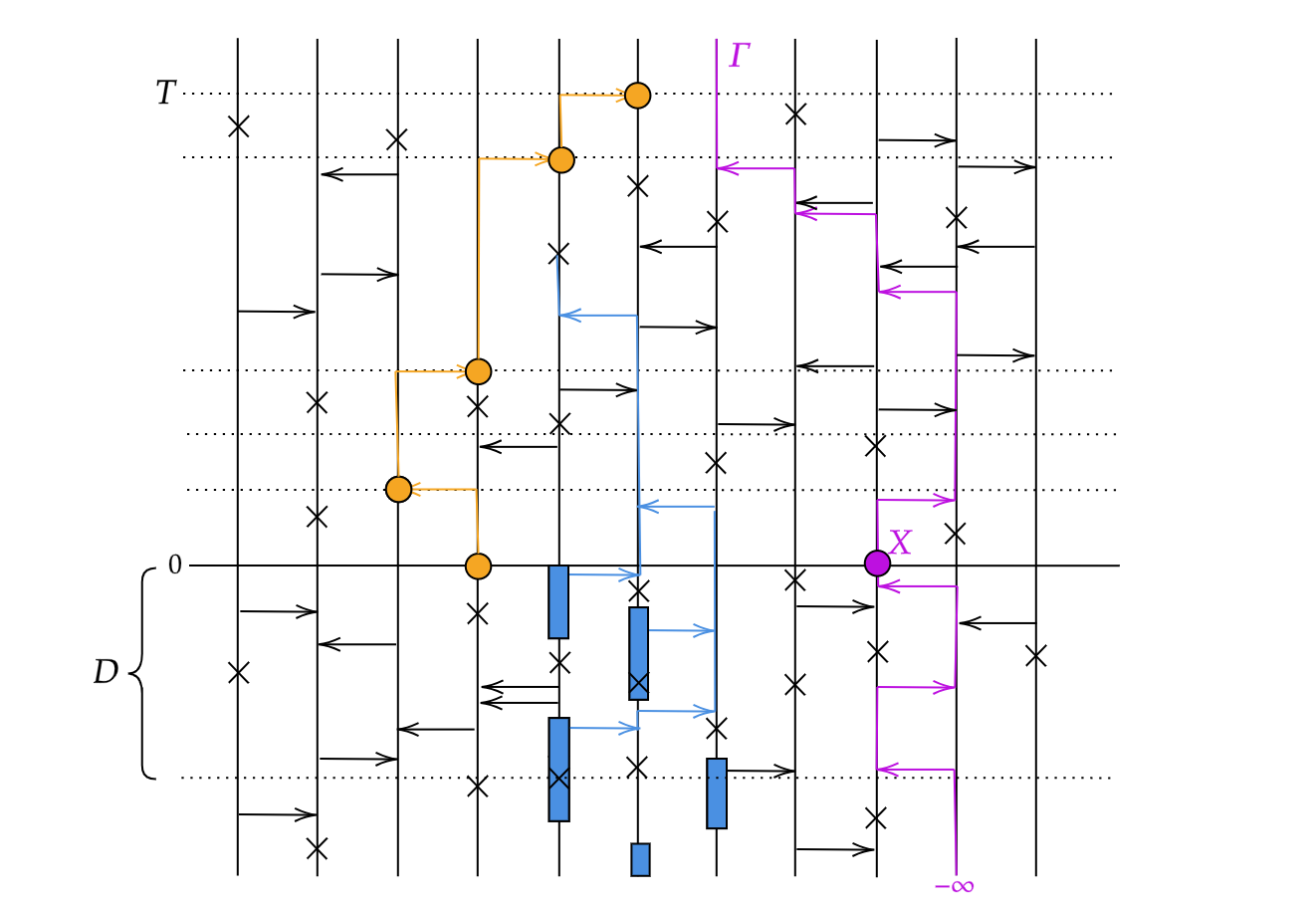}
    \caption{Illustration of the patchwork elements $T$, $D$, and $\Gamma$. The barrier and its trajectory are shown in orange, the trail $\mathrm{g}$ in blue, and the special infection path $\Gamma$ in purple. The random variable $X$ is shown in purple as well.
}
    \label{fig_patchworkElemets}
\end{figure}


\begin{proposition}\label{lem_first_law}
	The law of~$\{\beta_T(x-B_T): x \in \mathbb Z\}$ conditioned on $\sigma\left(\mathcal I, T, \Gamma, \mathrm{Left}_\Gamma(\mathcal H)\right)$ 
     is~$\mu_{\mathrm g\sqcup\Gamma}$, almost surely.
\end{proposition}
The proof is again given in Section~\ref{sss_proofs}.

\begin{definition}[Law~$\mathrm{Q_g}$]
We let~$\mathrm{Q_g}$ be the law of the 4-tuple

\begin{equation} \label{XiDef}
    \Xi := (T,\Gamma,D,(I_t)_{0 \le t \le T})
\end{equation}
for the contact-and-barrier process~$(\beta_t)_{t \ge 0}$ started from~$\beta_0 \sim \mu_{\mathrm g}$.
 \end{definition}
 \begin{definition} \label{def_Iprime_CBP}
     For~$T$ as in Definition~\ref{defT}, we let:
        \begin{equation}\label{eq_def_Itprime}
            I_t':=I_{T+t}-I_T, \quad t \ge 0.
    \end{equation}
\end{definition}

\begin{corollary} \label{FirstPatchworkInterface}
   Conditionally on~$\Xi$, the law of~$(I_t')_{t \ge 0}$ is~$P_{\mathrm g_\sqcup\Gamma}$.
\end{corollary}

The proof of this corollary will also be done in Section~\ref{sss_proofs}.

\begin{definition}
    Let~$Q_{\mathrm g}'$ denote the law of the pair~$(\Xi,(I_t')_{t \ge 0})$ for the contact-and-barrier process~$(\beta_t)_{t \ge 0}$ started from~$\beta_0 \sim \mu_{\mathsf g}$.
\end{definition}

\begin{lemma} \label{onePatch}
    If $(\Xi,(I_t')_{t\ge 0}) = ((T,\Gamma,D,(I_t)_{0 \le t \le T}),(I_t')_{t \ge 0}) \sim Q_{\mathrm{g}}'$, then it follows that
    \[\mathrm{Sew}((I_t)_{0 \le t \le T}, (I_t')_{t \ge 0}) \sim P_{\mathrm{g}}.\]  
\end{lemma}

\begin{proof}
    This follows readily from the definition of these objects as we have used the same graphical construction to build them. 
\end{proof}


\subsubsection*{Heuristics of patchwork elements}

\paragraph{} We will give a brief intuition behind the definitions and results stated in Section~\ref{section_PatchworkCBP}, starting by understanding an initial configuration drawn from~$\mu_{\mathrm g}$. When we start from~$\beta_0\sim\mu_{\mathrm{g}}$, we can think of the particles at the initial time as classified into two types: the ones that are present in~$\beta_0$ \textit{only} because there exists an infection path going backwards in time starting from it that reaches the trail~$\mathrm{g}$, but they have no infinite infection path starting from it; and the ones that do possess an infinite infection path starting from it, and we call those particles special particles. They are gathered then in the set~$\mathcal S$ as defined by~\eqref{eq_def_of_mathcal_S}. Note that the collection of special particles is distributed precisely according to~$\nu_{\lambda}$ restricted to~$\mathbb N_0$, whereas the collection of particles in~$\beta_0$ dominates~$\nu_\lambda$. 

In that perspective, we have that~$T$ is the first moment after one unit of time where the leftmost particle and the barrier are neighbours and the leftmost particle is a descendant of an special particles. The special path~$\Gamma$ is the leftmost one realising this encounter in between the barrier and the leftmost particle which is a descendant of a special particle. 

Next, we defined depth, an observable of the contact-and-barrier process that is associated with the stopping time~$T$ and with the special path~$\Gamma$; this observable will serve as an auxiliary tool in order for us to find a sequence of times such that the sequence of increments of the interface process in between those times are i.i.d.

Roughly, the heuristics behind the definition of depth is the following. At moment~$T$, we have that now the descendants of the special particles are stochastically larger than~$\nu_\lambda$, so we can again repeat the same procedure reclassifying which particles are special and wait for an adjacency moment with a particle that now does not need the advantage of neither ~$\mathrm{g}$ nor~$\text{Graph}(\gamma)$ to exist.

The observable depth will help us to find adjacency moments where the barrier has never seen a particle that was not a descendant of a special particle related to that moment. This will give us an i.i.d. structure roughly because those will be moments where we can restart the process from the barrier followed by a configuration drawn from~$\nu_\lambda$ as the interface process would never feel the difference since it would never notice the presence of particles that were not special. 

\subsubsection{Proof of properties of patchwork elements}
\label{sss_proofs}

\begin{proof}[Proof of Lemma~\ref{TBehavesWell}]
We first prove the lemma under the assumption~\eqref{eq_A2}. Using this assumption, we have~$\epsilon :=\tfrac12(\mathbf B + \alpha) > 0$. Let~$a > 4$, and let~$t := \frac{2a}{\mathbf B + \alpha - 2\epsilon}$, so that~$2a-(\alpha - \epsilon)t = (\mathbf B-\epsilon)t$. We now define three good events. First,
\[G_1:= \{B_t > (\mathbf B - \epsilon)t\}.\]

Second,
\[G_2:= \{ (\mathcal S \cap [a,2a]) \times \{0\} {\rightsquigarrow} (-\infty, (\mathbf B - \epsilon) t] \times \{t\}  \};\]
we emphasize that the notation~`${\rightsquigarrow}$' employed in this event pertains to infection paths in~$\mathcal H$, regardless of the behaviour of the barrier. Third,
\[G_3 := \left\{ \max_{s \in [0,1]} |B_s| < a/4\right\} \cap \left\{\min\{x \in \mathbb Z:\; \exists s \le 1\text{ s.t. }\; [a,\infty) \times \{0\} \rightsquigarrow (x,s)\} > 3a/4\right\}.\]

We claim that~$G_1 \cap G_2 \cap G_3 \subseteq \{\mathbf X < 2a,\; T \le t\}$. Indeed, if~$G_2$ occurs, then there is~$x \in \mathcal S \cap [a,2a]$ and an infection path~$\gamma:[0,t] \to \mathbb Z$ with~$\gamma(0) = x$ and~$\gamma(t) < (\mathbf B-\epsilon)t$. If~$G_1$ also occurs, then we can consider the first time~$t' < t$ when this infection path is neighbouring the barrier, that is,~$\gamma(t') = B_{t'} + 1$. If~$G_3$ also occurs, using~$a > 4$, we have~$B_s + 1< \gamma(s)$ for all~$s \in [0,1]$, so~$t' \ge 1$. Recalling the definition of~$\sigma_x$ in~\eqref{eq_def_sigma_x}, this shows that~$\sigma_x \le t' < t$. Using the definitions of~$\mathbf X$ and~$T$, as well as~\eqref{eq_minimality_sigma_x}, we have  now proved the desired inclusion of events.

We now turn to giving lower bounds for the probabilities of the good events. Lemma~\ref{lem_bar_Behaves_Well} implies that~$\mathbb P(G_1) > 1- e^{-ct}$ for some~$c > 0$. The choice of~$t$ and Corollary~\ref{CorRightmostWellBehaved} imply that
\begin{align*}\mathbb P(G_2) \ge \mathbb P((\mathcal S \cap [a,2a]) \times \{0\} \rightsquigarrow (-\infty, 2a-(\alpha-\epsilon)t] \times \{t\} )> 1 - Ce^{-ca}-e^{-ct}
\end{align*}
for some~$c,C > 0$. Finally, elementary bounds using Poisson random variables give~$\mathbb P(G_3) > 1- e^{-ca}$ for some~$c > 0$. Changing the constants, we have thus proved that
\[\mathbb P(\mathbf X < 2a,\; T \le 2a/(\mathbf B + \alpha -2\epsilon)) > 1 - Ce^{-ca}- Ce^{-ct}.\]

By adjusting the constants if needed, we obtain the desired bound.

The proof under assumption~\eqref{eq_A1} is the same, except that~$\mathbf B$ should be replaced everywhere by~$\min\{r^0_\rightarrow,r^1_\rightarrow\}- r_\leftarrow$, and the bound on the probability of~$G_1$ is then given by a standard large deviations bound for a random walk.\qedhere
\end{proof}

\begin{proof}[Proof of Lemma~\ref{DepthBehavesWell}]
    For each~$x \in \mathbb Z$, define
    \[\sigma_x':=\sup\{t\geq 0\,:\,\mathbb Z \times \{-t\} \stackrel{\mathcal H}{\rightsquigarrow} (x,0)\}.\]
    By duality and~\eqref{eq_small_cluster}, we have~$\mathbb P(t<\sigma_x'<\infty) < e^{-ct}$ for some~$c > 0$, all~$x$ and all~$t > 0$.
    
    The event that the depth is larger than~$d$ is equal to the event that there exist~$t > d$ and~$x \in \{1,\ldots, \mathbf X\}$ such that~$t<\sigma'_x<\infty$. Using a union bound and Lemma~\ref{TBehavesWell}, we have
    \[\mathbb P(D>d) \le \mathbb P(\mathbf X > d) + d \cdot e^{-cd} < Ce^{-\sqrt{d}}\]
    for all~$d > 0$ and some large constant~$C>0$, completing the proof.
\end{proof}

For the proof of Proposition~\ref{lem_first_law}, we will need a few auxiliary definitions and statements. We continue working on the same probability space as in the previous subsection, where an augmented graphical construction~$(\mathcal H, \mathcal I)$ is defined (with~$\mathcal H$ being defined for all times in~$\mathbb R$). 

Several of the objects we have defined in the previous subsection depend on the graphical construction~$\mathcal H$ for times down to~$-\infty$. We will now define ``truncated'' versions of these objects, where the dependence on negative times is forced to stop at some point. 
Recall the construction of~$\beta_0$ in~\eqref{eq_def_mu_g}.
We now define, for each~$n \in \mathbb N$,
\[
\beta^{(n)}_0(x) = \begin{cases}
    0&\text{if } x < 0;\\
    -1&\text{if } x = 0;\\
    \mathds{1}_{((\mathbb Z \times \{-n\}) \cup \mathrm g) \rightsquigarrow (x,0)}&\text{if } x > 0.
\end{cases}
\]
In analogy with~\eqref{eq_def_of_mathcal_S}, define
\[\mathcal S^{(n)}:= \{x \in \mathbb N:\;\mathbb Z \times \{-n\}  \rightsquigarrow (x,0)\}.\]
The following is readily seen.
\begin{claim}\label{cl_Nk}
    Almost surely, for all $k \in \mathbb N$ there exists $N(k)$ such that for all $n \ge N(k)$ we have
    \[\beta_0^{(n)} \cdot \mathds{1}_{[1,k]} = \beta_0 \cdot \mathds{1}_{[1,k]} \qquad \text{and} \qquad \mathcal S^{(n)} \cap [1,k] = \mathcal S \cap [1,k].\]
\end{claim}

We now let~$(\beta^{(n)}_t)_{t \ge 0}$ denote the contact-and-barrier process started from~$\beta^{(n)}_0$ at time~$0$, evolving according to~$(\mathcal H, \mathcal I)$. We let~$(B^{(n)}_t)_{t \ge 0}$ denote the associated barrier trajectory.

In analogy with~\eqref{eq_def_sigma_x}, we define
\begin{equation}\label{eq_def_sigma_x_trunc}
\mathcal T^{(n)}_x := \inf\left\{\begin{array}{l} t \ge 1:\; \text{in $(\beta^{(n)}_t)_{t \ge 0}$ there is a barrier-free}\\ \text{infection path from $(x,0)$ to~$(B_t^{(n)}+1,t)$} \end{array}\right\},\quad x \in \mathcal S^{(n)}.
\end{equation}

We emphasize that~$\mathcal T^{(n)}_x$ could differ from~$\mathcal T_x$, since the trajectory of the barrier and the set of barrier-free infection paths depend not only on~$(\mathcal H, \mathcal I)$, but also on the configuration at time zero. However, we make the following observation, whose proof is elementary and thus omitted:
\begin{claim}\label{cl_Nk2}
    Assume that~$x \in \mathcal{S}$ and~$n \ge N(x)$, where~$N(x)$ is as in Claim~\ref{cl_Nk}. Then,~$\mathcal{T}^{(n)}_x = \mathcal{T}_x$.
\end{claim}
Similarly to~\eqref{eq_def_bold_X} and Definition~\ref{defT}, let
\[\mathbf{X}^{(n)}:=\inf\{x \in \mathcal{S}^{(n)}:\; \mathcal{T}^{(n)}_x < \infty\}, \qquad T^{(n)}:= \mathcal{T}^{(n)}_{\mathbf{X}^{(n)}}.\]

By combining the previous two claims with an extra argument involving the crossing of paths, we have:
\begin{claim}\label{cl_Nk3}
    Almost surely, if~$n \ge N(\mathbf X)$ we have
    \[\mathbf X^{(n)} = \mathbf X,\qquad T^{(n)} = T,\qquad B^{(n)}_t = B_t \text{ for all } t \in [0,T].\]
    
Moreover, the leftmost barrier-free infection path in~$(\beta^{(n)}_t)$ from~$(\mathbf X^{(n)},0)$ to~$(B^{(n)}_T+1,T)$ is equal to the leftmost barrier-free infection path in~$(\beta_t)$ from~$(\mathbf X,0)$ to~$(B_T+1,T)$.
\end{claim}

Finally, in analogy with Definition~\ref{GammaT}, we let~$\Gamma^{(n)}:[-n,T^{(n)}] \to \mathbb Z$ be the infection path defined as follows:
\begin{itemize}
    \item in~$[-n,0]$,~$\Gamma^{(n)}$ is the leftmost infection path from~$\mathbb Z \times \{-n\}$ to~$(\mathbf X^{(n)},0)$;
    \item in~$[0,T^{(n)}]$,~$\Gamma^{(n)}$ is the leftmost barrier-free infection path (with respect to~$(\beta^{(n)}_t)$) from~$(\mathbf X^{(n)},0)$ to~$(B^{(n)}_{T^{(n)}}+1,T^{(n)})$.
\end{itemize}

We then have:
\begin{claim}\label{clNk4}
    Almost surely, for every~$s < 0$, there exists~$N' \ge N(\mathbf X)$ such that for all~$n \ge N'$ we have~$\Gamma^{(n)}\vert_{[s,T]} = \Gamma\vert_{[s,T]}$.
\end{claim}
\begin{proof}
    It follows from Claim~\ref{cl_Nk3} that if~$n \ge N(\mathbf X)$, then~$T^{(n)} = T$ and~$\Gamma^{(n)}\vert_{[0,T]}=\Gamma\vert_{[0,T]}$. Fix~$s < 0$. Almost surely, when~$n$ is large enough, the leftmost infection path from~$\mathbb Z \times \{-n\}$ to~$(\mathbf X,0)$ and the leftmost infection path from~$-\infty$ to~$(\mathbf X,0)$ agree on~$[s,0]$. This concludes the proof.
\end{proof}

\begin{lemma}
    \label{lem_change_H_patchwork}
On the same probability space where the augmented graphical construction~$(\mathcal H, \mathcal I)$ is defined, let~$\mathcal H'$ be another graphical construction of the contact process with parameter~$\lambda$, defined for all times in~$\mathbb R$, and independent of~$(\mathcal H, \mathcal I)$. Then, almost surely,
\[
\mathrm{Law}(\mathrm{Right}_\Gamma^+(\mathcal H) \mid \mathcal{I}, T, \Gamma, \mathrm{Left}_\Gamma(\mathcal H)) = \mathrm{Law}(\mathrm{Right}_\Gamma^+(\mathcal H') \mid \mathcal{I}, T, \Gamma, \mathrm{Left}_\Gamma(\mathcal H)).
\]
\end{lemma}
\begin{proof}
Since~$\Gamma^{(n)}$ is a right-preserving random infection path, we have that the distributions of
\[(\mathcal I, T^{(n)}, \Gamma^{(n)},\mathrm{Left}_{\Gamma^{(n)}}(\mathcal H),\mathrm{Right}_{\Gamma^{(n)}}^+(\mathcal H)) \qquad \text{and} \qquad (\mathcal I, T^{(n)}, \Gamma^{(n)},\mathrm{Left}_{\Gamma^{(n)}}(\mathcal H),\mathrm{Right}_{\Gamma^{(n)}}^+(\mathcal H'))\]
are the same, by {Lemma~\ref{lem_RPRIP_property}}.\footnote{In fact, this follows from a slightly stronger version of {Lemma~\ref{lem_RPRIP_property}}, as here we  include the flight plan~$\mathcal I$ in the random vectors. A quick inspection shows that the proof of {Lemma~\ref{lem_RPRIP_property}} allows for this. Also note that the inclusion of~$T^{(n)}$ makes no difference, as~$T^{(n)}$ is measurable with respect to~$\sigma(\Gamma^{(n)})$, since~$T^{(n)}$ is the endpoint of the domain of~$\Gamma^{(n)}$.}

As~$n \to \infty$, these converge to
\[(\mathcal I, T, \Gamma,\mathrm{Left}_{\Gamma}(\mathcal H),\mathrm{Right}_{\Gamma}^+(\mathcal H)) \qquad \text{and} \qquad (\mathcal I, T, \Gamma,\mathrm{Left}_{\Gamma}(\mathcal H),\mathrm{Right}_{\Gamma}^+(\mathcal H'))\]
respectively, almost surely\footnote{We endow the space of c\`adl\`ag paths with the Skhorohod topology, and the space of graphical constructions with the infinite product of the Skhorohod topology for Poisson processes over~$\mathbb R$.}. This implies that the two limiting random vectors have the same distribution, which readily gives the statement of the lemma.
\end{proof}

\begin{proof}[Proof of Proposition~\ref{lem_first_law}] 
	Throughout this proof, we fix a realization of~$T$,~$\Gamma$,~$\mathrm{Left}_\Gamma(\mathcal H)$, and~$\mathcal I$, so these are all treated as deterministic.

	Fix~$x \in \mathbb Z$ with~$x > B_T$. We claim that
	\begin{equation}\label{eq_claim_key}
		\beta_T(x)=1 \qquad \text{if and only if} \qquad (\{-\infty\}  \cup \mathrm{Graph}(\Gamma)\cup \mathrm{g}) \, \rightsquigarrow \, (x,T) \text{ in } \mathrm{Right}_\Gamma^+(\mathcal H)
	\end{equation}
    (the latter condition meaning that in the infection path in question, all jumps are due to elements of~$\mathrm{Right}_\Gamma^+(\mathcal H)$). 
    
	To prove this, first assume that~$\beta_T(x)=1$. This implies that there exists~$y \in \mathbb N$ such that~$\beta_0(y)=1$  and~$(y,0) {\rightsquigarrow} (x,t)$ with a (barrier-free) infection path~$\gamma'$. Having~$\beta_0(y)=1$ is the same as saying that there is an infection path~$\gamma''$ from~$\{-\infty\} \cup \mathrm g$ to~$(y,0)$, by the construction of~$\beta_0$ in~\eqref{eq_def_mu_g}. By concatenating~$\gamma''$ with~$\gamma'$, we obtain an infection path~$\gamma$ from~$\{-\infty\} \cup \mathrm g$ to~$(x,t)$. By deleting the portion of~$\gamma$ below its (chronologically) last intersection with~$\mathrm{Graph}(\Gamma)$, we obtain a path from~$\{-\infty\} \cup \mathrm{Graph}(\Gamma) \cup \mathrm g$ to~$(x,t)$ that only uses reproduction marks of~$\mathrm{Right}_\Gamma^+(\mathcal H)$.

	For the converse, assume that~$(\{-\infty\}  \cup \mathrm{Graph}(\Gamma)\cup \mathrm g) \rightsquigarrow (x,T)$  in $\mathrm{Right}_\Gamma^+(\mathcal H)$. Let~$\gamma$ denote the infection path that achieves this connection. We now distinguish between two cases. First, assume that~$\gamma$ starts at a time~$t_0 > 0$. Since~$\mathrm g \subset \mathbb Z \times (-\infty,0]$, we must then have~$(\gamma(t_0),t_0) \in \mathrm{Graph}(\Gamma)$.  Since~$\beta_t(\Gamma(t))=1$ for every~$t \in [0,T]$, we obtain~$\beta_{t_0}(\gamma(t_0))=1$. Any infection path that lies on the right of a barrier-free infection path is also barrier-free, so we obtain that~$\gamma$ is barrier-free and hence~$\beta_T(\gamma(T))=1$.

	The second case is when~$\gamma$ starts at~$t_0 \le 0$. Then there are further (non-mutually exclusive) sub-cases:
	\begin{itemize}
		\item if~$\gamma$  either starts at~$\mathrm g$ or comes from~$-\infty$, then we have~$\beta_0(\gamma(0))=1$ by the construction of~$\beta_0$ in~\eqref{eq_def_mu_g}. Again using the fact that~$\gamma\vert_{[0,T]}$ stays on the right of~$\Gamma\vert_{[0,T]}$ and~$\Gamma$ is barrier-free, we obtain that~$\gamma\vert_{[0,T]}$ is barrier-free, so~$\beta_T(\gamma(T))=1$;
		\item if~$(\gamma(t_0),t_0) \in \mathrm{Graph}(\Gamma)$, we then see that~$-\infty \rightsquigarrow (\gamma(0),0)$, by concatenating~$\Gamma$ on~$(-\infty,t_0]$ with~$\gamma$ on~$[t_0,0]$. This gives~$\beta_0(\gamma(0))=1$, and we conclude that~$\beta_T(\gamma(T))=1$ as in the previous item.
	\end{itemize}
    This concludes the proof of~\eqref{eq_claim_key}, that is, we now have
        \[
(\beta_T(x))_{x > B_T} =\left(\mathds{1}\{ (\{-\infty\}  \cup \mathrm{Graph}(\Gamma)\cup \mathrm{g}) \, \rightsquigarrow \, (x,T) \text{ in } \mathrm{Right}_\Gamma^+(\mathcal H) \}\right)_{x > B_T}.
    \]

    Now let~$\mathcal H'$ be a graphical construction of the contact process (for both negative and positive times) independent of~$(\mathcal H, \mathcal I)$. By Lemma~\ref{lem_change_H_patchwork}, we have
    \[
    \mathrm{Law}(\mathrm{Right}_\Gamma^+(\mathcal H)\mid \mathcal{I}, T,\Gamma, \mathrm{Left}_\Gamma(\mathcal H)) = \mathrm{Law}(\mathrm{Right}_\Gamma^+(\mathcal H')\mid \mathcal{I}, T,\Gamma, \mathrm{Left}_\Gamma(\mathcal H)).
    \]
    
    In particular, almost surely, the law of
    \[
 \left(\mathds{1}\{ (\{-\infty\}  \cup \mathrm{Graph}(\Gamma)\cup \mathrm{g}) \, \rightsquigarrow \, (x,T) \text{ in } \mathrm{Right}_\Gamma^+(\mathcal H) \}\right)_{x > B_T}
    \]
    conditionally on~$\mathcal{I}, T,\Gamma, \mathrm{Left}_\Gamma(\mathcal H)$ is equal to the law of
\[
\left(\mathds{1}\{ (\{-\infty\}  \cup \mathrm{Graph}(\Gamma)\cup \mathrm{g}) \, \rightsquigarrow \, (x,T) \text{ in } \mathrm{Right}_\Gamma^+(\mathcal H') \}\right)_{x > B_T}
\]
    conditionally on~$\mathcal{I}, T,\Gamma, \mathrm{Left}_\Gamma(\mathcal H)$.
Now, for each~$x > B_T$, we have
\begin{align*}&(\{-\infty\}  \cup \mathrm{Graph}(\Gamma)\cup \mathrm{g}) \, \rightsquigarrow \, (x,T) \text{ in } \mathrm{Right}_\Gamma^+(\mathcal H') \\
&\hspace{2cm} \text{if and only if} \qquad  (\{-\infty\}  \cup \mathrm{Graph}(\Gamma)\cup \mathrm{g}) \, \stackrel{\mathcal H'}{\rightsquigarrow} \, (x,T)
\end{align*}
(the ``only if'' part is obvious, and the ``if'' part comes from ignoring the portion of the infection path up to the last intersection with~$\mathrm{Graph}(\Gamma)$). 

We have thus proved that 
\begin{align*}
    &\mathrm{Law}\left(\left(\beta_T(B_T+x)\right)_{x > 0} \mid \mathcal I, T, \Gamma, \mathrm{Left}_\Gamma(\mathcal H) \right)\\
    &\quad =\mathrm{Law}\left((\mathds 1 \{ (-\infty \cup \mathrm{Graph}(\Gamma) \cup \mathrm{g}) \stackrel{\mathcal H'}{\rightsquigarrow} (B_T+x,T) \})_{x > 0} \mid \mathcal I, T, \Gamma, \mathrm{Left}_\Gamma(\mathcal H)\right).
\end{align*}
Since the indicator functions on the right-hand side can be obtained as functions of~$\mathcal H'$ and~$\Gamma$ only, the right-hand side equals
\[\mathrm{Law}\left((\mathds 1 \{ (-\infty \cup \mathrm{Graph}(\Gamma) \cup \mathrm{g}) \stackrel{\mathcal H'}{\rightsquigarrow} (B_T+x,T) \})_{x > 0} \mid \Gamma\right).\]

By comparison with~\eqref{eq_def_mu_g}, this is equal to the law of~$(\beta(x))_{x > 0}$ when~$\beta \sim \mu_{\mathrm{g} \sqcup \Gamma}$. This concludes the proof.
\end{proof}

\begin{proof}[Proof of Corollary~\ref{FirstPatchworkInterface}]
    We need to prove that, for every bounded and measurable function~$f$, we have
    \begin{equation*}
        \mathbb E[f((I_t')_{t \ge 0}) \mid \Xi] = \int f \;\mathrm{d} P_{\rmg \sqcup \Gamma}.
    \end{equation*}
    
    Fix the function~$f$, and for each configuration~$\beta$ for the contact-and-barrier process, define
    \[
    F(\beta):=\int f \; \mathrm{d}P^0_\beta,
    \]
    where~$P^0_\beta$ is given in Definition~\ref{def_P0}.

    Let~$\mathcal F_{(-\infty, T]}$ denote the~$\sigma$-algebra generated by both the graphical construction restricted to the time interval~$(-\infty,T]$ and the random walk flight plan restricted to the time interval~$[0,T]$. By the strong Markov property, we have
    \[ \mathbb E[f((I_t')_{t \ge 0}) \mid \mathcal F_{(-\infty, T]}] = F(\beta_T).\]
    Using the fact that~$\Xi$ is measurable with respect to~$\mathcal F_{(-\infty, T]}$, we then have
    \[\mathbb E[f((I_t')_{t \ge 0}) \mid \Xi] = \mathbb E[ \mathbb E[f((I_t')_{t \ge 0}) \mid \mathcal F_{(-\infty, T]}] \mid \Xi] = \mathbb E[F(\beta_T) \mid \Xi].\]
    
    Next, let~$\mathcal G$ be the~$\sigma$-algebra generated by~$\mathcal I$,~$T$,~$\Gamma$, and~$\mathrm{Left}_\Gamma(\mathcal H)$ as in the statement of Proposition~\ref{lem_first_law}. That lemma says that the law of~$\beta_T$ conditionally on~$\mathcal G$ is~$\mu_{\rmg \sqcup \Gamma}$; hence,
    \[ \mathbb E[F(\beta_T) \mid \mathcal G] = \int F(\beta) \; \mu_{\rmg \sqcup \Gamma}(\mathrm d \beta) = \int \left( \int f \; \mathrm{d}P^0_\beta\right)\; \mu_{\rmg \sqcup \Gamma}(\mathrm d \beta) \stackrel{\eqref{eq_P0_and_P}}{=} \int f \; \mathrm{d} P_{\rmg \sqcup \Gamma}.\]
    
    Since~$\Xi$ is measurable with respect to~$\mathcal G$, we have
    \[\mathbb E[F(\beta_T) \mid \Xi] = \mathbb E[\mathbb E[ F(\beta_T) \mid \mathcal G] \mid \Xi] = \mathbb{E} \left[\left. \int f \;\mathrm d P_{\rmg \sqcup \Gamma}  \;\right| \;\Xi\right] =  \int f \mathrm d P_{\rmg \sqcup \Gamma}.
    \qedhere
    \] 
\end{proof}

\subsubsection{Sewing the patchwork} \label{subsub_sewingCBP}

\paragraph{} In the previous sections, all objects and results were formulated through the graphical construction of the process. We now adopt a more abstract viewpoint, working directly with the probability laws of the observables introduced earlier and treating them as given random objects, independent of their explicit construction. We proceed by piecing together the interface via the following function:

\begin{definition}[Sewing of paths]
Letting~$(\gamma^1(s): 0 \le s \le t^1)$ and~$(\gamma^2(s): 0 \le s \le t^2)$ be trajectories on some vector space, we define~$\mathrm{Sew}(\gamma^1,\gamma^2)$ as the trajectory~$(\gamma(s): 0 \le s \le t^1+t^2)$ given by
\[
\gamma(s) = \begin{cases}
    \gamma^1(s)&\text{if } s \le t^1;\\[.2cm]
    \gamma^1({t^1}) + \left(\gamma^2({s-t^1}) - \gamma^2({0})\right) &\text{if } t^1 < s \le t^1+t^2.
\end{cases}
\]

We extend this notion in the obvious way to sequences of more than two paths -- i.e.,
\[\mathrm{Sew}(\gamma^1,\ldots, \gamma^n):=\mathrm{Sew}(\mathrm{Sew}(\gamma^1,\ldots, \gamma^{n-1}),\gamma^n)\] -- and also to sequences of infinitely many paths. In the case of a sequence of finitely many paths, it is also allowed that the last path in the sequence is defined for~$0 \le t < \infty$ (in which case the output of the sewing is also a path defined for all times).
\end{definition}

To improve clarity, we will adopt the following. Given random objects~$X$ and~$Y$, we write~$\mathrm{Law}(X)$ for the law of~$X$, and~$\mathrm{Law}(X \mid Y)$ for the law of~$X$ conditionally on~$Y$. At last, we state and prove the following proposition. 

\begin{proposition}[Patchwork construction of~$P_{\mathrm{g}}$] \label{prop_patchwork}
Fix~$\mathrm g \in \mathcal R$. Let~$(\Xi^n)_{n \ge 0}$ be a sequence with~$\Xi^n=(T^n,\Gamma^n,D^n,(I^n_t)_{t \le T^n})$, and distribution specified inductively as follows:
\begin{align}
&\mathrm{Law}(\Xi^0) = Q_{\rmg};\\[.1cm]
&\mathrm{Law}(\Xi^{n+1} \mid \Xi^0,\ldots, \Xi^n) = Q_{\rmg \sqcup \Gamma^0 \sqcup \cdots \sqcup \Gamma^n} \; \text{ for all }n \in \mathbb N_0. \label{eq_nome_xi_n}
\end{align}
Then,~$\mathrm{Sew}((I^0_t)_{t \le T^0}, (I^1_t)_{t \le T^1},\ldots)$ has law~$P_{\mathrm g}$.
\end{proposition}

\begin{proof}
In an augmented probability space, we construct the sequence~$(\Xi^0,\Xi^1,\ldots)$ together with an extra sequence~$((\hat I^1_t)_{t \ge 0}, (\hat I^2_t)_{t \ge 0}, \ldots)$, where each~$(\hat I^n_t)_{t \ge 0}$ is a random c\`adl\`ag function taking values in~$\{(a,b) \in \mathbb Z^2: a < b\}$. For this argument, we abbreviate~$I^i=(I^i_t)_{t \le T^i}$ and~$\hat{I}^i = (\hat I_t^i)_{t \ge 0}$.

We  specify the joint distribution of these two sequences inductively as follows:

\begin{itemize}
    \item $\mathrm{Law}(\Xi^0) = Q_{\mathrm g}$;
    \item $\Xi^{n+1}$ and~$\hat I^{n+1}$ are independent conditionally on $(\Xi^0, \hat I^1,\ldots, \Xi^{n}, \hat I^n)$, with
    \begin{align*}
    &\mathrm{Law}(\hat I^{n+1} \mid \Xi^0, \hat I^1,\ldots \Xi^n, \hat I^n) = P_{\mathrm g \sqcup \Gamma^0 \sqcup \cdots \sqcup \Gamma^n}, \\
    &\mathrm{Law}(\Xi^{n+1} \mid \Xi^0, \hat I^1,\ldots \Xi^n, \hat I^n) = Q_{\mathrm g \sqcup \Gamma^0 \sqcup \cdots \sqcup \Gamma^n}.
    \end{align*}

For the sake of clarity, we include a diagram representing the dependence structure of the first few random elements in this sequence:
\begin{figure}[H]
\begin{center}
\setlength\fboxsep{0.1cm}
\setlength\fboxrule{0.0cm}
\fbox{\includegraphics[width=0.4\textwidth]{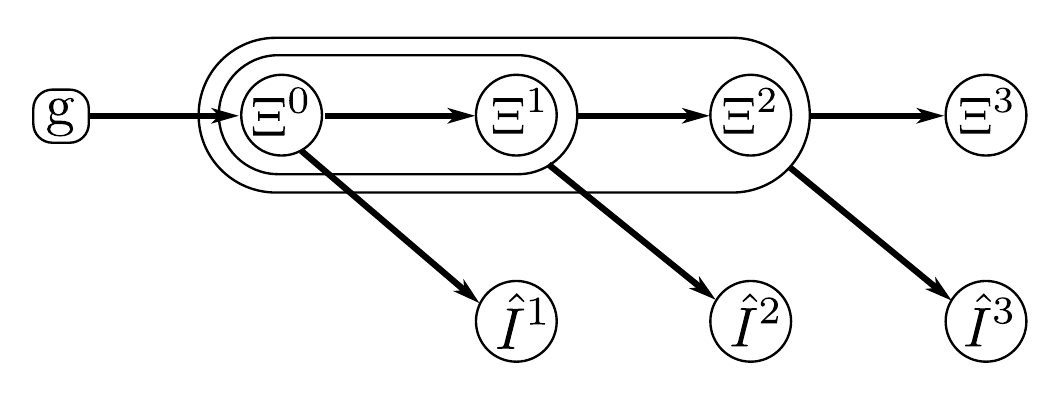}}
\end{center}
\end{figure}\vspace{-.7cm}

We claim that
\begin{equation}\label{eq_inductive_sew}
    \mathrm{Law}( \mathrm{Sew} (I^0, I^1, \ldots, I^n, \hat I^{n+1})) = P_{\mathrm g} \quad \text{for all } n \in \mathbb N.
\end{equation}
We prove this by induction. The case~$n=1$ follows from noting that the distribution of the pair~$(\Xi^0, \hat I^1)$ is~$Q_{\mathrm g}'$, and then using Lemma~\ref{onePatch}. For the induction step, note that
\[\mathrm{Law}((\Xi^{n+1}, \hat I^{n+2}) \mid \Xi^1,\ldots, \Xi^n) = Q_{\mathrm g \sqcup \Gamma^1 \sqcup \cdots \sqcup \Gamma^n}',\]
so again by Lemma~\ref{onePatch}, 
\[
\mathrm{Law}(\mathrm{Sew}(I^{n+1},\hat I^{n+2} )\mid \Xi^1, \ldots, \Xi^n) = P_{\mathrm g \sqcup \Gamma^0 \sqcup \cdots \sqcup \Gamma^n} = \mathrm{Law}(\hat I^{n+1} \mid \Xi^1,\ldots, \Xi^n).
\]

This then gives
\[
\mathrm{Law}(\mathrm{Sew}(I^1,\ldots, I^n, I^{n+1}, \hat I^{n+2})) = \mathrm{Law}(\mathrm{Sew}(I^1,\ldots, I^n, \hat I^{n+1})) = P_{\mathrm g},
\]
the last equality being the induction hypothesis. This completes the proof of~\eqref{eq_inductive_sew}.

To conclude, note that the trajectories given by
\[
\mathscr{I}^n:=\mathrm{Sew}(I^1,\ldots, I^n, \hat I^{n+1}) \quad \text{and} \quad \mathscr{I}^\infty := \mathrm{Sew}(I^1,I^2,\ldots)
\]
agree on the interval~$[0,T^1+\cdots + T^n] \supseteq [0,n]$ since~$T_i\geq 1$ for all~$i\in\mathbb N_0$. In particular,~$\mathscr I^n$ converges to~$\mathscr I^\infty$ almost surely, hence also in distribution. This gives~$\mathrm{Law}(\mathscr I^\infty) = P_{\mathrm g}$ as required. \qedhere
\end{itemize}
\end{proof}

\begin{definition}
    We call a sequence~$(\Xi^n)_{n \ge 0}$ with law as prescribed in the statement of Proposition~\ref{prop_patchwork}, corresponding to~$\mathrm g \in \mathcal R$, a \emph{patchwork sequence with initial trail $\mathrm g$}.
\end{definition}

\subsubsection{Renewals of the patchwork construction} \label{subsec_renewalsCBP}

\paragraph{}This section constructs a sequence of stopping~$(\kappa_n)_{n\in\mathbb N_0}$ satisfying the conditions stated in Lemma~\ref{lem_renewal}. Using the depth concept from Definition~\ref{DefDepth}, we first establish these key times and then develop their fundamental properties through a sequence of lemmas.

\begin{definition} \label{def_kappa_n}
    Let~$\mathrm g \in \mathcal R$ and let~$(\Xi^n)_{n \ge 0}$ be a patchwork sequence with initial trail~$\mathrm g$. We define the stopping times
\begin{equation}\label{eq_def_kappa_n}
        \kappa_n := \inf\{n' \ge n:\; D^{n'} \ge T^n + \cdots + T^{n'-1}\},\quad n \in \mathbb N_0,
\end{equation}
    where we interpret~$T^n + \cdots + T^{n'-1}$ as zero if~$n'=n$.
\end{definition}

Figure~\ref{fig:lemRenewalConnection} illustrates the first few elements of the sequence~$\kappa_n$ for a realisation of the contact-and-barrier process.

\begin{figure}[h]
\centering
\includegraphics[width=0.8\linewidth]{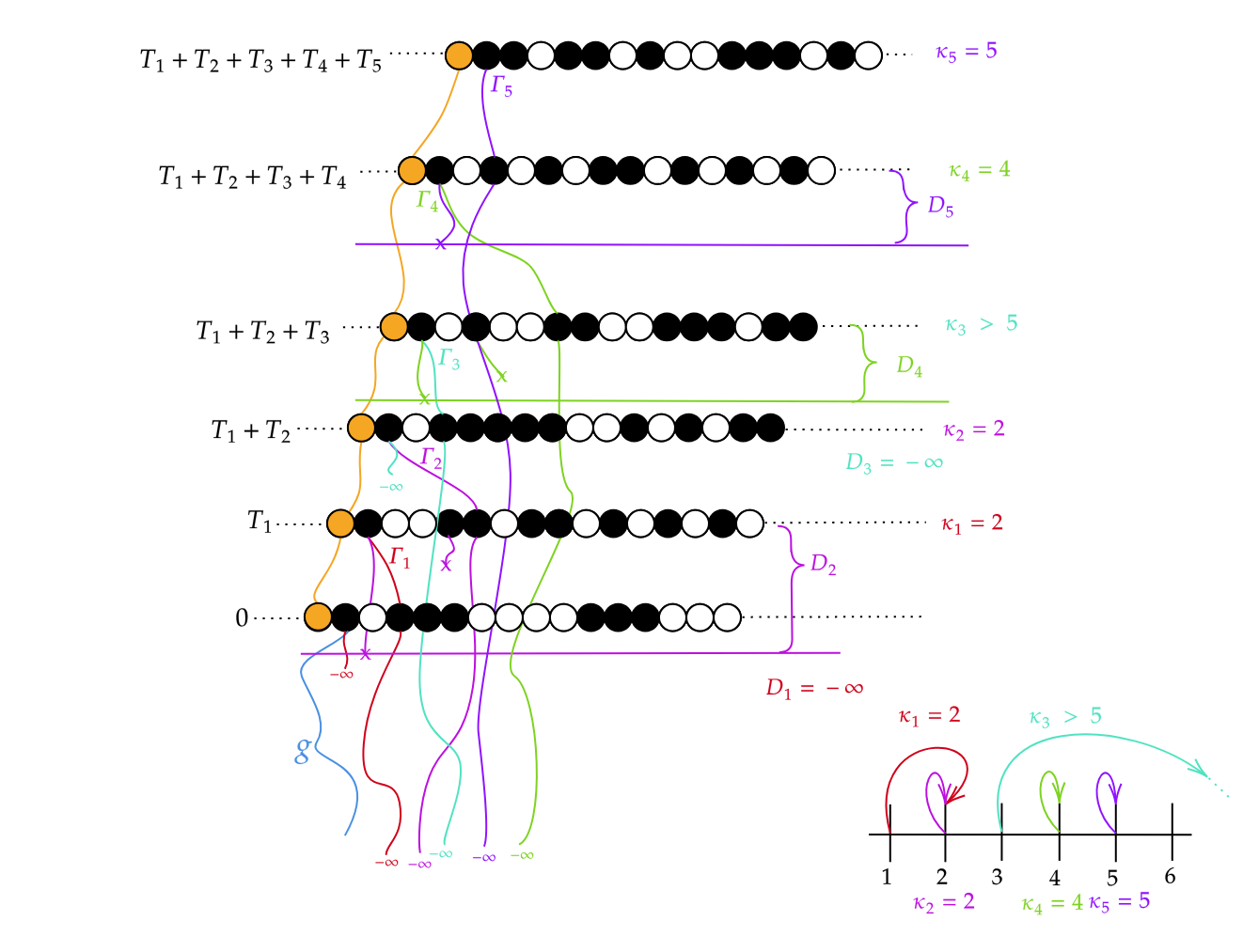}
    \caption{Illustration of the times $\kappa_n$ for a realization of the contact-and-barrier process obtained via the patchwork construction. The barrier and its trajectory are shown in orange; occupied sites are shown in black and empty sites in white. The figure illustrates the connection between Definition~\ref{def_kappa_n} and Lemma~\ref{lem_renewal}.}
    \label{fig:lemRenewalConnection}
\end{figure}

\begin{lemma} \label{kappa_m_kappa_n}
    If~$m < n$ and~$\kappa_m \ge n$, then~$\kappa_m \ge \kappa_n$.
\end{lemma}
\begin{proof}
The statement is trivial if~$\kappa_n = n$, so we assume that~$\kappa_n > n$. We observe that~$\kappa_m \ge n$ is equivalent to
    \begin{equation}
        \label{eq_first_kappa_n}
    D^{\ell} < T^{m} + \cdots T^{\ell -1} \text{ for all } \ell \in \{m,\ldots, n-1\}.
    \end{equation}
We also have
\begin{equation}
    D^{\ell} < T^{n} + \cdots T^{\ell -1} \text{ for all } \ell \in \{n,\ldots, \kappa_n-1\},
    \end{equation}
    which implies that
    \begin{equation}
        \label{eq_sec_kappa_n}
        D^{\ell} < T^m + \cdots T^{\ell -1} \text{ for all } \ell \in \{n,\ldots, \kappa_n-1\}.
    \end{equation}
Then,~\eqref{eq_first_kappa_n} and~\eqref{eq_sec_kappa_n} taken together are equivalent to~$\kappa_m \ge \kappa_n$.
\end{proof}

\begin{lemma} \label{lem_kappaPoswithPosProb}
    Let~$(\Xi^n)_{n\geq 0}$ be a patchwork sequence with initial trail~$\mathrm g$ which is defined under some probability measure~$\mathbb P$. Then,~$\mathbb P(\kappa_0=\infty)>0$. 
\end{lemma}

\begin{proof}
     Let~$\mathcal G$ be the space of the 4-tuples~$\Xi$ as in~\eqref{XiDef} such that~$T=1$ and~$D=-\infty$. Let~$n\in \mathbb N$ be fixed (it will be made large enough later in the proof) and let~$G(n)$ be the good event defined by~$G(n):=\left\{\Xi^0\in \mathcal G,\dots,\Xi^{n}\in\mathcal G\right\}$. Note that~$\mathbb P (G(n))>0$:

    \begin{align*}
        \mathbb P\left(\Xi^0\in \mathcal G,\dots,\Xi^{n}\in\mathcal G\right) &= \mathbb E\left[\mathbb E\left[\mathds{1}_{\{\Xi^0\in \mathcal G,\dots,\Xi^{n}\in\mathcal G\}}\mid \Xi^0,\dots,\Xi^{n-1}\right]\right] \\
        &= \mathbb E\left[\mathds{1}_{\{\Xi^1\in\mathcal G,\dots,\Xi^{n-1}\in\mathcal G\}}\mathbb P\left(\Xi^n\in \mathcal G\,\mid\,\Xi^1,\dots,\Xi^{n-1}\right)\right] \\
        &= \mathbb E\left[\mathds{1}_{\{\Xi^1\in\mathcal G,\dots,\Xi^{n-1}\in\mathcal G\}}\mathrm{Q}_{\mathrm g\sqcup \Gamma^0\sqcup \dots \sqcup \Gamma^{n-1}}(T=1\cap D=\infty)\right]\\
        &=p\mathbb P\left(\{\Xi^1\in\mathcal G,\dots,\Xi^{n-1}\in\mathcal G\}\right) = \dots = p^{n+1}
    \end{align*}
    for some~$p>0$ where the last equality follows by repeating this procedure iteratively. 
    
    To see that this number~$p$ is positive, just note that the event~$\{T=1\}\cap\{D=-\infty\}$ could be implied by the following conditions: there exists~$\gamma:(-\infty,0]\rightarrow \mathbb Z$ infection path with~$\gamma(0)=1$,~$D^1\cap[0,1]=\emptyset$ and~$J_{\rightarrow}\cap [0,1]=\emptyset$. 

    We note that~$ \mathbb P \left(\kappa_0=\infty\right) \geq  \mathbb P \left(\kappa_0=\infty\,\mid\,G(n)\right)\mathbb P\left(G(n)\right)$ and therefore we will be done once we show that~$\mathbb P\left(\kappa_0<\infty\,\mid\,G(n)\right)<1$ instead. We have the following inclusion of events:
    \[
\{\kappa_0<\infty\}\cap G(n)
\subset \left\{\bigcup_{m>n}\{D^m>T^0+\dots+T^{m-1}\}\right\}\cap G(n)
\subset \left\{\bigcup_{m>n}\{D^m>m-1\}\right\}\cap G(n),
\]
and therefore~$\mathbb P\left(\kappa_0<\infty\,\mid\,G(n)\right)\leq \mathbb P \left(\bigcup_{m>n}\{D^m>m-1\}\,\mid\,G(n)\right)$. We can bound this last term by:
    \begin{equation} \label{summation_kappaInfty}
        \sum_{m>n}\frac{\mathbb P(\{D^m>m-1\}\cap G(n))}{\mathbb P(G(n))}
    \end{equation}
    and to control this term we work on the numerator in the following way. Let~$m>n$ be given. We have:

    \begin{align*}
        \mathbb P\left(\{D^m>m-1\}\cap G(n)\right) &= \mathbb E\left[\mathbb E\left[\mathds{1}_{\{D^m>m-1\}}\mathds{1}_{\{G(n)\}}\,\mid\,\Xi^0,\dots,\Xi^{m-1}\right]\right] \\
        &=\mathbb E\left[\mathds{1}_{\{G(n)\}}\mathbb E\left[\mathds{1}_{\{D^m>m-1\}}\,\mid\,\Xi^0,\dots,\Xi^{m-1}\right]\right] \\
        &=\mathbb E\left[\mathds{1}_{\{G(n)\}}\mathrm{Q}_{\mathrm{g}\sqcup\Gamma^0\dots\sqcup\Gamma^{m-1}}\left(D^m>m-1\right)\right] \\
        &\leq Ce^{-c(m-1)^p}\mathbb P(G(n)),
    \end{align*}
    where the last inequality follows from Lemma~\ref{DepthBehavesWell}. Therefore, the sum in~\eqref{summation_kappaInfty} can be bounded by~$\sum_{m>n}Ce^{-c(m-1)^p}$ which can be made smaller than~$1$ by choosing~$n$ sufficiently large. 
\end{proof}

\begin{lemma}\label{lem_coupled_patchworks}
    Let~$(\Xi^n)_{n \ge 0}$ be a patchwork sequence with initial trail~$\mathrm g$, defined under a probability measure~$\mathbb P$. Then, for every bounded and measurable function~$f$, the value of
    \[\mathbb E[f(\Xi^0,\Xi^1,\ldots) \cdot \mathds{1}\{\kappa_0 = \infty\}]\]
    does not depend on the initial trail~$\mathrm g$.
\end{lemma}
The proof of this lemma is somewhat technical and lengthy and therefore we postpone it to Section~\ref{ss_coupled}.

\begin{corollary}\label{cor_sameAsInitial}
    Let~$(\Xi^n)_{n \ge 0}$ be a patchwork sequence with initial trail~$\mathrm g$, defined under a probability measure~$\mathbb P$. Then, for every bounded and measurable function~$f$ and every~$n \ge 1$,
\begin{align*}
        &\mathbb E[f(\Xi^{n},\Xi^{n+1},\ldots)\cdot \mathds{1}\{\kappa_{n}=\infty\} \mid \Xi^0,\ldots, \Xi^{n-1}] \\[.1cm]&\hspace{4cm}= \mathbb E[f(\Xi^0,\Xi^1,\ldots) \mid \kappa_0=\infty] \text{ a.s.}
\end{align*}
\end{corollary}
\begin{proof}
    Write~$\mathbb E_{\rmg}$ for the expectation operator corresponding to a probability measure~$\mathbb P_{\rmg}$ under which a patchwork construction with initial trail~$\rmg$ is defined. Then, by~\eqref{eq_nome_xi_n} in the definition of the law of the patchwork construction,
    \begin{align*}
         &\mathbb E_{\rmg}[f(\Xi^{n},\Xi^{n+1},\ldots)\cdot \mathds{1}\{\kappa_{n}=\infty\} \mid \Xi^0,\ldots, \Xi^{n-1}] \\[.1cm]
         &= \mathbb E_{\rmg \sqcup \Gamma^0 \sqcup \cdots \sqcup \Gamma^{n-1}} [f(\Xi^0,\Xi^1,\ldots) \cdot \mathds{1}\{\kappa_0 = \infty\}].
    \end{align*}
    
    By Lemma~\ref{lem_coupled_patchworks}, the right-hand is unchanged if we replace the random trail~$\mathrm g \sqcup \Gamma^0 \sqcup \cdots \sqcup \Gamma^{n-1}$ by any other trail; in particular, we can replace it by~$\rmg$. This completes the proof.
\end{proof}

At last, we observe that we are in the conditions of Lemma~\ref{lem_renewal}. Indeed,~$\kappa_n\geq n$ by definition so that condition~$(i)$ is satisfied. Because of Lemmas~\ref{kappa_m_kappa_n} and~\ref{lem_kappaPoswithPosProb}, we are also in condition~$(ii)$  and~$(iv)$ of Lemma~\ref{lem_renewal}. Moreover, Corollary~\ref{cor_sameAsInitial} guarantees that we are under condition~$(iii)$ of~\ref{lem_renewal}. It remains the work of controlling the interface process in between those renewals, and we do so in the next lemma. 

\begin{lemma} \label{time_and_space_betweenRenewals}
    Let~$(\Xi^n)_{n\in \mathbb N_0}$ be a patchwork construction with some initial trail~$\mathrm g$ defined under some law~$\mathbb P$. Let~$N_0<N_1<\dots$ denote the indexes~$n\in\mathbb N$ for which one has~$\kappa_n=\infty$. Let~$m\in \mathbb N_0$ be arbitrary but fixed. 
    Then, there exist constants~$c,C,p>0$ independent of the trail~$\mathrm g$ such that the following is true for all~$t\geq 0$:
    \begin{equation}\label{timeBetweenRenBehavesWell}
        \mathbb P\left(\sum_{n=N_m}^{N_{m+1}-1}T^n >t\right) \leq Ce^{-ct^p}
    \end{equation}
    
    Moreover, if we let~$\tau_m:=\sum_{n=0}^{N_m-1}T^n$, then it follows that there exist constants~$c,C,p>0$ (again independent of the trail~$\mathrm g$) such that the following is true for all~$x\geq 0$:

    \begin{equation} \label{spaceBetweenRenBehavesWell}
        \mathbb P \left(\sup\left\{|i_t - i_{\tau_m}|\,:\,t\in[\tau_m,\tau_{m+1}]\right\}>x\right) \leq Ce^{-cx^p}
    \end{equation}
\end{lemma}

\begin{proof}
     We start by observing that it would be enough to show that~\eqref{timeBetweenRenBehavesWell}  and~\eqref{spaceBetweenRenBehavesWell} hold for~$m=1$ conditioning on~$N_0=0$ because of Lemma~\ref{lem_renewal}. 

    Let~$t\in[0,\infty)$ and let~$E_1:=\left\{\sum_{n=1}^{N_1}T^n>t\right\}$. Since~$\mathbb P(E_1\mid\{N_0=0\})=\frac{\mathbb P(E_1\cap\{N_0=0\})}{\mathbb P(N_0=0)}$, the proof of~\eqref{timeBetweenRenBehavesWell} follows as long as we have the desired sub exponential bound for~$\mathbb P(E_1\cap\{N_0=0\})$. To study this term, let~$G_1$ be the good event~$G_1=\{N_1\leq t^{1/2}\}$. It follows that:

    \begin{equation*}
    \mathbb P\left(E_1 \cap \{N_0=0\}\right) \leq \underbrace{\mathbb P\left(E_1 \cap \{N_0=0\} \cap G_1\right)}_{(1)} + \underbrace{\mathbb P(G_1^c)}_{(2)}
\end{equation*}

    The term in~$(1)$ is bounded by~$\mathbb P\left(\bigcup_{n=1}^{\lceil t^{1/2}\rceil}\{T^n>t^{1/2}\}\right)\leq \lceil t^{1/2}\rceil Ce^{-c(t^{1/2})^p}$ due to Lemma~\ref{TBehavesWell}. For~$(2)$, note that~$G_1^c\subset\left\{\bigcap_{i=1}^{\lfloor t^{1/2}\rfloor}\{\kappa_i<\infty\}\right\}$. We then have:

\[
\mathbb P\!\left(\bigcap_{i=1}^{\lfloor t^{1/2}\rfloor}\{\kappa_i<\infty\}\right)
= \mathbb E\!\left[\prod_{i=1}^{\lfloor t^{1/2}\rfloor-1}\mathds{1}_{\{\kappa_i<\infty\}}
\,\mathbb E\!\left[\mathds{1}_{\{\kappa_{\lfloor t^{1/2}\rfloor}<\infty\}}
\mid \Xi^0,\dots,\Xi^{\lfloor t^{1/2}\rfloor-1}\right]\right]
\le q^{\lfloor t^{1/2}\rfloor}.
\]
for some~$q\in(0,1)$ where we have obtained the last term by repeating the previous argument iteratively. 

   Let~$x\geq 0$ and let~$E_2:=\left\{\sup\{|i_t-i_0|\,:\,t\in[0,\tau_1]\}>x\right\}$. Again it is enough to show that the term~$\mathbb P \left(E_2\cap \{N_0=0\}\right)$ has the desired subexponential bound. To control this term, let~$G_2$ be the good event defined by~$G_2:=\{\tau_1 \leq x^{1/2}\}$. Then we have:

   \begin{equation} \label{eq_boundTimeDisplacementRenewals_Term3}
    \mathbb P\left(E_2 \cap \{N_0=0\}\right) \leq \underbrace{\mathbb P\left(E_2 \cap \{N_0=0\} \cap G_2\right)}_{(3)} + \underbrace{\mathbb P(G_2^c)}_{(4)}
\end{equation}

    Because of~\eqref{timeBetweenRenBehavesWell}, we have that~$(4)$ has the desired bound. To check that the same is true for~$(3)$, simply note that this term is bounded by the probability of a Poisson random variable with parameter~$(r_{\leftarrow} + r_{\rightarrow})x^{1/2}$ to be bigger than~$x$. 
\end{proof}

\subsubsection{Coupled patchworks: proof of Lemma~\ref{lem_coupled_patchworks}}\label{ss_coupled}

\paragraph{} We now return momentarily to the setup of Section~\ref{patchworkElements}, where we worked on a probability space with measure~$\mathbb P$ under which an augmented graphical construction~$(\mathcal H, \mathcal I)$ was defined (with~$\mathcal H$ for both negative and positive times). Rather than fixing a single trail~$\mathrm g$, as we did in that section, we now fix two trails~$\mathrm g,\mathrm g'$. Applying formula~\eqref{eq_def_mu_g} to~$\mathrm g$ and~$\mathrm{g}'$, respectively, we define random configurations~$\beta_0^\mathrm{g} \sim \mu_{\mathrm g}$ and~$\beta_0^{\mathrm{g}'} \sim \mu_{\mathrm g'}$. We then use these as starting configurations for the contact-and-barrier processes~$(\beta^\mathrm{g}_t)_{t \ge 0}$ and~$(\beta^{\mathrm{g}'}_t)_{t \ge 0}$, respectively, both governed by~$(\mathcal H,\mathcal I)$. Then, following definitions of Section~\ref{patchworkElements}, we define the two quadruples
\begin{equation}\label{eq_coupled_quadruples}
\Xi_{\rmg} = (T_{\rmg}, \Gamma_{\rmg},I_{\rmg},D_{\rmg}) \qquad \text{and}\qquad \Xi_{\rmg'} = (T_{\rmg'}, \Gamma_{\rmg'},I_{\rmg'},D_{\rmg'}).
\end{equation}

We can now state:
\begin{lemma}[Depth and influence]\label{lem_depth_and_influence}
    For all~$\rmg,\rmg' \in \mathcal R$, defining~$\Xi_\mathrm{g}$ and~$\Xi_{\mathrm g'}$ coupled as above, we have
\begin{equation}
    \label{eq_cases_Q_0}
    \mathbb P(\{D_\rmg = D_{\rmg'} = -\infty,\; \Xi_\rmg = \Xi_{\rmg'} \} \cup \{D_{\rmg} \ge 0, \; D_{\rmg'} \ge 0\}) = 1.    
\end{equation}

    Moreover, if~$\rmg \cap (\mathbb Z \times [-t,0]) = \rmg' \cap (\mathbb Z \times [-t,0]) $ for some~$t \ge 0$, then
    \begin{equation}
    \label{eq_cases_Q}
\mathbb P(\{D_\rmg = D_{\rmg'} < t,\; \Xi_\rmg = \Xi_{\rmg'} \} \cup \{D_{\rmg} \ge t, \; D_{\rmg'} \ge t\}) = 1.    
\end{equation}
\end{lemma}
\begin{proof}
The proofs of both statements are very similar, so we only prove the second one~\eqref{eq_cases_Q}. The ideas in this proof are very similar to those in the claims preceding Lemma~\ref{lem_change_H_patchwork}, so we only sketch the main steps.

Fix~$t > 0$ and~$\mathrm{g}, \mathrm{g}'$ with~$\rmg \cap (\mathbb Z \times [-t,0]) = \rmg' \cap (\mathbb Z \times [-t,0]) $.
Recalling~\eqref{eq_def_of_mathcal_S},~\eqref{eq_def_sigma_x}, and~\eqref{eq_def_bold_X}, here we write
\[\mathcal S_{\mathrm{g}},\; \mathcal S_{\mathrm g'},\quad \mathcal T_{\mathrm{g},x},\; \mathcal T_{\mathrm{g}',x},\quad \mathbf{X}_{\mathrm g},\; \mathbf{X}_{\mathrm{g}'}\]
to distinguish between the objects defined using the two trails.

Assume that~$D_{\mathrm g} < t$. Then, using the construction of~$\beta^{\mathrm g}_0$ and~$\beta^{\mathrm g'}_0$ with~\eqref{eq_def_mu_g} and the definition of~$\mathcal S_{\mathrm g}$ and~$\mathcal S_{\mathrm{g}'}$, we have
\[\beta_0^{\mathrm g} \cdot \mathds{1}_{\{1,\ldots,\mathbf{X}_{\mathrm g}\}} = \beta_0^{\mathrm g'}\cdot \mathds{1}_{\{1,\ldots,\mathbf{ X}_{\mathrm g}\}} \qquad \text{and}\qquad \mathcal S_\mathrm{g} \cap \{1,\ldots, \mathbf{X}_{\mathrm g}\} =  \mathcal S_\mathrm{g'} \cap \{1,\ldots, \mathbf{X}_{\mathrm g}\}. \]

Moreover, for each~$x \in \{1,\ldots, \mathbf{X}_{\mathrm g}\}$, the fact that~$\beta_0^{\mathrm g} \cdot \mathds{1}_{\{1,\ldots, x\}}=\beta_0^{\mathrm g'} \cdot \mathds{1}_{\{1,\ldots, x\}}$ implies that~$\mathcal T_{\mathrm g,x} = \mathcal T_{\mathrm{g}',x}$ because of Lemma~\ref{lem_SameInterface}, and, in case this quantity is finite, the leftmost barrier-free infection path in~$(\beta^\mathrm{g}_t)$ from~$(x,0)$ to~$(B^\mathrm{g}_{\mathcal T_{\mathrm g,x}}+1,\mathcal T_{\mathrm g,x})$ is equal to the leftmost barrier-free infection path in~$(\beta^{\mathrm{g}'}_t)$ from~$(x,0)$ to~$(B^{\mathrm{g}'}_{\mathcal T_{\mathrm g,x}}+1,\mathcal T_{\mathrm g,x})$. It follows from these considerations that
\[
\mathbf X_{\mathrm g'} = \min\{x \in \mathcal S_{\mathrm g'}: \; \mathcal{T}_{\mathrm g',x} < \infty\} = \min\{x \in \mathcal S_{\mathrm g}: \; \mathcal{T}_{\mathrm g,x} < \infty\} = \mathbf X_{\mathrm g},
\]
and then also that~$T_{\mathrm g} = T_{\mathrm g'}$, that~$D_{\mathrm g'} = D_{\mathrm g}$, and that the negative and positive portions of~$\Gamma_{\mathrm g}$ and~$\Gamma_{\mathrm{g}'}$ agree. 

We have thus proved that~$\{D_\mathrm g < t\} \subseteq \{D_{\mathrm g}=D_{\mathrm g'} < t,\; \Xi_{\mathrm g} = \Xi_{\mathrm g'} \}$. By symmetry we then also have~$\{D_{\mathrm g'} < t\} \subseteq \{D_{\mathrm g}=D_{\mathrm g'} < t,\; \Xi_{\mathrm g} = \Xi_{\mathrm g'} \}$, which then gives~\eqref{eq_cases_Q}.
\end{proof}

\begin{definition}\label{def_depth_and_influence}
We denote by~$Q_{\mathrm g,\mathrm{g}'}$ the law of the coupled quadruples~$(\Xi_\mathrm{g},\Xi_{\mathrm{g}'})$ defined in~\eqref{eq_coupled_quadruples}.
\end{definition}

We are now ready to prove Lemma \ref{lem_coupled_patchworks}. 

\begin{proof}[Proof of Lemma~\ref{lem_coupled_patchworks}]
    We start with a definition. Given~$\mathrm{h}, \mathrm{h}' \in \mathcal R$ and~$t \ge 0$, define
\begin{equation}\label{eq_def_var_phi}
    \varphi(\mathrm{h}, \mathrm{h}', t):= Q_{\mathrm h, \mathrm h'}(\{D_{\mathrm h} = D_{\mathrm h'} < t,\; \Xi_{\mathrm h} = \Xi_{\mathrm h'}\} \cup \{D_{\mathrm h} \ge t, \; D_{\mathrm h'} \ge t\}).
\end{equation}
Note that Lemma~\ref{lem_depth_and_influence} says that
\begin{equation}\label{eq_reform_depth_and_influence}
    \mathrm h \cap [-t,0] = \mathrm h' \cap [-t,0] \quad \Longrightarrow \quad \varphi (\mathrm h, \mathrm h', t) = 1.
\end{equation}

In what follows, for any~$\rmg \in \mathcal R$, we let~$\mathbb P_{\mathrm g}$ denote a probability measure under which a patchwork sequence~$(\Xi^n)_{n \ge 0}$ corresponding to~$\mathrm g$ is defined. We denote by~$\mathbb E_{\mathrm g}$ the associated expectation operator.

Fix a bounded and measurable function~$f$; we want to prove that, for all~$\rmg,\rmg' \in \mathcal R$,
\begin{equation}\label{eq_want_with_indicator}
    \mathbb E_{\rmg}[f(\Xi^0,\Xi^1,\ldots) \cdot \mathds{1}\{\kappa_0 = \infty\}]=\mathbb E_{\rmg'}[f(\Xi^0,\Xi^1,\ldots) \cdot \mathds{1}\{\kappa_0 = \infty\}].
\end{equation}

We now construct a coupling of patchwork sequences. Given~$\mathrm g, \mathrm g' \in \mathcal R$, we let~$\mathbb P_{\mathrm g, \mathrm g'}$ denote a probability measure under which a random sequence
\[(\dot{\Xi}^n,\ddot{\Xi}^n)_{n \ge 0}=((\dot T^n, \dot \Gamma^n, \dot I^n, \dot D^n),(\ddot T^n, \ddot \Gamma^n, \ddot I^n, \ddot D^n))_{n \ge 0}\]
is defined, with law as follows:
\begin{itemize}
    \item $(\dot{\Xi}^0,\ddot{\Xi}^0) \sim Q_{\mathrm g, \mathrm g'}$;
    \item the law of~$(\dot{\Xi}^{n+1},\ddot{\Xi}^{n+1})$ conditionally on~$(\dot{\Xi}^0,\ddot{\Xi}^0,\ldots, \dot{\Xi}^n,\ddot{\Xi}^n)$ is
    \[Q_{\mathrm g \sqcup \dot{\Gamma}^1 \sqcup \cdots \sqcup \dot{\Gamma}^n,\;\mathrm g \sqcup \ddot{\Gamma}^1 \sqcup \cdots \sqcup \ddot{\Gamma}^n}.\]
    \end{itemize}
    We denote by~$\mathbb E_{\mathrm g, \mathrm g'}$ the associated expectation operator.
    Clearly, the marginal sequences~$(\dot{\Xi}^n)_{n \ge 0}$ and~$(\ddot{\Xi}^n)_{n \ge 0}$ are patchwork sequences corresponding to~$\mathrm g$ and~$\mathrm g'$, respectively.

    We also define stopping times~$\dot \kappa_n$ and~$\ddot \kappa_n$ associated to the sequences~$(\dot{\Xi}^n)_{n \ge 0}$ and~$(\ddot{\Xi}^n)_{n \ge 0}$, respectively, as in~\eqref{eq_def_kappa_n}, that is:
    \begin{align*}
            &\dot\kappa_n := \inf\{n' \ge n:\; \dot D^{n'} \ge \dot T^n + \cdots + \dot T^{n'-1}\},\\[.2cm]
            &\ddot\kappa_n := \inf\{n' \ge n:\; \ddot D^{n'} \ge \ddot T^n + \cdots + \ddot T^{n'-1}\}.
    \end{align*}

    We now claim that
\begin{equation}\label{eq_equality_with_dots}
    \mathbb P_{\rmg, \rmg'}(\{\dot \kappa_0 = \ddot \kappa_0 = \infty,\; \dot \Xi^n = \ddot \Xi^n \;\forall n\} \cup \{\dot \kappa_0 = \ddot \kappa_0 < \infty \})=1.
\end{equation}
Note that~\eqref{eq_equality_with_dots} readily gives~\eqref{eq_want_with_indicator}, as follows:
\begin{align*}
    &\mathbb E_{\rmg}[f(\Xi^0,\Xi^1,\ldots) \cdot \mathds{1}\{\kappa_0 = \infty\}] = \mathbb E_{\rmg,\rmg'}[f(\dot \Xi^0,\dot \Xi^1,\ldots) \cdot \mathds{1}\{\dot \kappa_0 = \infty\}] \\[.2cm]&\stackrel{\eqref{eq_equality_with_dots}}{=} \mathbb E_{\rmg,\rmg'}[f(\ddot \Xi^0,\ddot \Xi^1,\ldots) \cdot \mathds{1}\{\ddot \kappa_0 = \infty\}] = \mathbb E_{\rmg'}[f(\Xi^0,\Xi^1,\ldots) \cdot \mathds{1}\{\kappa_0 = \infty\}].
\end{align*}

It remains to prove~\eqref{eq_equality_with_dots}. It is easy to see, using an approximation argument, that~\eqref{eq_equality_with_dots} follows from proving that, for every~$n \in \mathbb N_0$,
\begin{equation}
    \label{eq_equality_with_dots0}
    \mathbb P_{\rmg, \rmg'}(A_n \cup B_n) = 1,
\end{equation}
where we abbreviate
\begin{align*}&A_n:=\{\dot \kappa_0 = \ddot \kappa_0 > n,\; (\dot \Xi^0,\ldots,\dot \Xi^n) = (\ddot \Xi^0,\ldots, \ddot \Xi^n)\},\\ &B_n:= \{\dot \kappa_0 = \ddot \kappa_0 \le n \} .\end{align*}

We prove~\eqref{eq_equality_with_dots0} by induction on~$n$. The case~$n=0$ readily follows from~\eqref{eq_cases_Q_0}. Now assume the statement has been proved for~$n$.  Let~$\mathscr F_n$ be the~$\sigma$-algebra generated by~$(\dot \Xi_0,\ddot{\Xi}_0,\ldots, \dot{\Xi}_n,\ddot{\Xi}_n)$. Since~$B_{n} \subseteq B_{n+1}$, we have that
\begin{equation}\label{eq_justify_AB1}
    \text{on } B_n,\; \mathbb P_{\rmg,\rmg'}( B_{n+1} \mid \mathscr F_n) = 1, \text{ so } \mathbb P_{\rmg,\rmg'}(A_{n+1} \cup B_{n+1} \mid \mathscr F_n) = 1. 
\end{equation}
Next, recalling the function~$\varphi$ in~\eqref{eq_def_var_phi}, by the definition of the law of the coupling, we have that
\begin{equation}\label{eq_justify_AB2}
\begin{split}
    \text{on }A_n,\; &\mathbb P_{\rmg,\rmg'}(A_{n+1} \cup B_{n+1} \mid \mathscr F_n) \\
    &= \varphi(\rmg \sqcup \dot \Gamma^0 \sqcup \cdots \sqcup \dot \Gamma^n,\;\rmg' \sqcup \ddot \Gamma^0 \sqcup \cdots \sqcup \ddot \Gamma^n,\; \dot T^0+\cdots + \dot T^n) = 1
    \end{split}
\end{equation}
by~\eqref{eq_reform_depth_and_influence}, since on~$A_n$ we have~$(\dot T^0, \dot \Gamma^0,\ldots, \dot T^n,\dot \Gamma^n) = (\ddot T^0, \ddot \Gamma^0,\ldots, \ddot T^n,\ddot \Gamma^n)$, and so
\begin{align*}
    &(\rmg \sqcup \dot \Gamma^0 \sqcup \cdots \sqcup \dot \Gamma^n) \cap (\mathbb Z \times [-(\dot T^0+ \cdots \dot T^n),0])\\
    &= (\rmg \sqcup \ddot \Gamma^0 \sqcup \cdots \sqcup \ddot \Gamma^n) \cap (\mathbb Z \times [-(\dot T^0+ \cdots \dot T^n),0]).
\end{align*}

Now, we write
\begin{align*}
    &\mathbb P_{\rmg, \rmg'}(A_{n+1} \cup B_{n+1}) \\[.1cm]
    &= \mathbb E_{\rmg, \rmg'}[\mathbb P_{\rmg,\rmg'}(A_{n+1} \cup B_{n+1} \mid \mathscr F_n)]\\[.1cm]
    &\ge \mathbb E_{\rmg, \rmg'}[\mathbb P_{\rmg,\rmg'}(A_{n+1} \cup B_{n+1} \mid \mathscr F_n) \cdot \mathds{1}_{A_n}] +  \mathbb E_{\rmg, \rmg'}[\mathbb P_{\rmg,\rmg'}(A_{n+1} \cup B_{n+1} \mid \mathscr F_n) \cdot \mathds{1}_{B_n}] \\[.1cm]
    &\stackrel{\eqref{eq_justify_AB1},\eqref{eq_justify_AB2}}{\ge} \mathbb P_{\rmg, \rmg'}(A_n) + \mathbb P_{\rmg, \rmg'}(B_n) = 1,
\end{align*}
where the last equality follows from the induction hypothesis. 
\end{proof}

\subsection{Proof of Theorems~\ref{thmCBP_Speed},~\ref{ThmCBP_Tightness} and \ref{ThmCBP_CLT_Speed}} \label{sec_proofsThmsCBP}

\paragraph{} In this section, we prove the main results regarding the contact-and-barrier process. To do so, we fix the following. Let~$\mathbb P$ be a probability measure under which is defined a patchwork sequence~$(\Xi^n)_{n\in\mathbb N_0}$ for the contact-and-barrier process started from the Heaviside configuration. Note that we can obtain the Heaviside configuration by considering a configuration drawn with~$\mu_{\mathrm g_{\mathrm h}}$ as in~\eqref{eq_def_mu_g} where the trail~$\mathrm g_{\mathrm h}$ given by
\[
\mathrm{g}_{\mathrm h} = \bigcup_{x\in \mathbb N}\left\{\{x\}\times \{0\}\right\}
\]

Let~$(\kappa_n)_{n\in\mathbb N_0}$ be as in Definition~\ref{def_kappa_n}, and let~$N_0, N_1,\dots$ be the increasing sequence of indexes~$n\in\mathbb N_0$ for which we have~$\kappa_n=\infty$. Using the same notation as in Lemma~\ref{lem_renewal}, let~$Y_n:=\Xi^n$. Therefore, we can conclude that the random sequences
\begin{equation} \label{seq_iid_CBP}
    \left(N_1-N_0,\Xi^{N_0},\dots, \Xi^{N_1-1}\right),\,\left(N_2-N_1,\Xi^{N_1},\dots, \Xi^{N_2-1}\right),\,\dots
\end{equation}
are i.i.d. all distributed with the same law as~$(N_1,\Xi^0,\dots,\Xi^{N_1-1})$ conditioned on~$\{\kappa_0=0\}$.

For~$n \in \mathbb{N}_0$, let
\begin{equation} \label{eq_taun_CBP}
    \tau_n := \sum_{k=0}^{N_n-1} T^k
\end{equation}
be the time until the $n$-th renewal, with~$\tau_0 = 0$ whenever~$N_0 = 0$.  
Given~$t \in [0,\infty)$, define
\begin{equation} \label{eq_Nt_CBP}
    N(t) := \sup\{m \in \mathbb{N}_0 : \tau_m \leq t\}
\end{equation}
as the index of the last renewal before~$t$, taking~$N(t) = -\infty$ if no such renewal occurs.  
We then set
\begin{equation} \label{eq_sigmat_CBP}
    \sigma(t) := \tau_{N(t)}
\end{equation}
the time of the last renewal before~$t$, with~$\sigma(t) = 0$ if~$N(t) = -\infty$. At last, consider the conditional law
\begin{equation} \label{eq_conditionalLaw_CBP}
    \Bar{\mathbb P} \left(\cdot\right) = \mathbb P\left(\cdot\mid\{\kappa_0=\infty\}\right)
\end{equation}

\begin{lemma} \label{lem_timeSinceLastRenewal_CBP}
    On the above conditions, there exist constants~$c,C,p>0$ such that for all~$t\in[0,\infty)$ it follows that:
    \begin{equation}\label{eq_BehaviourLastRenewal_CBP}
        \Bar{\mathbb P}\left(\max\left\{(t-\sigma(t)),\sup\{|i_s-i_{\sigma(t)}|\,:\,s\in[\sigma(t),t]\}\right\}>L\right)\leq Ce^{-cL^p} \quad \forall L\geq 0
    \end{equation}
\end{lemma}

\begin{proof}
    The proof we give here follows very closely the lines of proof of Lemma 2.5 in~\cite{Valesin2010Multitype}. For~$k\in\mathbb N_0$, define~$M_k := \sup\{|i_s-i_{\tau_k}|\,:\,s\in[\tau_k,\tau_{k+1}]\}$ and~$\varphi(k):=\max\{(\tau_{k+1}-\tau_k),M_k\}$. In that way, by decomposing on the set of possible values of~$N(t)$ and using a union bound, we can rewrite the left-hand side of~\eqref{eq_BehaviourLastRenewal_CBP} as
\begin{align}
\sum_{k=0}^{\infty}\Bar{\mathbb P}\!\left(\tau_{k}\le t,\,\tau_{k+1}>t,\,\varphi(k)>L\right)
&= \sum_{k=0}^{\infty}\int_0^t
\Bar{\mathbb P}\!\left(\{\tau_{k+1}-\tau_k \ge t-u\}\cap \{\varphi(k)>L\}\right)
\Bar{\mathbb P}\!\left(\tau_k\in \mathrm du\right) \nonumber \\
&= \sum_{k=0}^{\infty}\int_0^t
\Bar{\mathbb P}\!\left(\{\tau_{1} \ge t-u\}\cap \{\max\{\tau_1,M(0)\}>L\}\right)
\Bar{\mathbb P}\!\left(\tau_k\in \mathrm du\right). 
\label{eq:last_line}
\end{align}

where in the second equality we have used the conclusion of Lemma~\ref{lem_renewal} to argue that the time increments in between renewals and the interface displacement between renewals are i.i.d. At last, note that~$\Bar{\mathbb P}\left(\max \{\tau_1,M(0)\}>L\right)$ is bounded by:
\begin{align*}
     \underbrace{\Bar{\mathbb P }\left(\left\{\max \{\tau_1,M(0)\}>L\right\}\cap \left\{\tau_1<\frac{L}{2(r_{\rightarrow}+r_{\leftarrow} + \lambda)}\right\}\right)}_{(1)} 
    \quad + \underbrace{\Bar{\mathbb P} \left(\tau_1\geq\frac{L}{2(r_{\rightarrow}+r_{\leftarrow} + \lambda)}\right)}_{(2)}
\end{align*}

The term~$(2)$ can be bounded by~$Ce^{-cL^P}$ because of Lemma~\ref{time_and_space_betweenRenewals}. To bound~$(1)$, we proceed as follows. Let~$I_0 := \left[0,T^0\right]$ and for~$k\in\{1,\dots,N_1-1\}$ let~$I_k:=\left[\sum_{j=0}^{k-1}T^j,\sum_{j=0}^kT^j\right]$, so that the time intervals~$I_k$ for~$k\in\{0,\dots,N_1-1\}$ form a partition of~$[0,\tau_1]$, i.e.,~$[0,\tau_1] = \dot{\cup}_{k=0}^{N_1-1}I_k$. 

For each $k\in\{0,\dots,N_1-1\}$, the number of jumps by the barrier or the leftmost particle is bounded by an independent Poisson variable with parameter $(r_{\rightarrow}+r_{\leftarrow}+\lambda)|I_k|$. Hence the total number of jumps in $[0,\tau_1]$ is stochastically dominated by a Poisson$(L/2)$ variable, so the term in~(1) is bounded by $\mathbb P(\mathrm{Poisson}(L/2)>L)\le Ce^{-cL}$. This bounds~\eqref{eq:last_line} by:

\begin{align}
    &\sum_{k=0}^{\infty}\sum_{i=1}^{\lceil t\rceil }\int_{i-1}^i \left[\Bar{\mathbb P}\left(\tau_1 \geq t-u\right) \wedge\Bar{\mathbb P}\left(\max\{\tau_1,M(0)\}>L\right)\right]\Bar{\mathbb P }\left(\tau_k\in \mathrm du\right) \nonumber \\
   & \leq \sum_{i=1}^{\lceil t\rceil }\left(Ce^{-c(t-i)^p}\wedge Ce^{-cL^p}\right)\sum_{k=0}^{\infty} \Bar{\mathbb P }\left(\tau_k\in [i-1,i]\right) \label{eq:last_eq_finalbound}
\end{align}
and note that~$\sum_{k=0}^\infty \Bar{\mathbb{P}}\left(\tau_k\in[i-1,i]\right) =\Bar{\mathbb{E}}\left[|\{n\in\mathbb N_0\,:\,\tau_n\in[i-1,i]\}|\right]\leq 1$ since~$\tau_{n+1}-\tau_n\geq 1$ for all~$n\in\mathbb N_0$ by definition. Thus, we can bound~\eqref{eq:last_eq_finalbound} by
\[
C\sum_{i=1}^\infty e^{-ci^p}\wedge e^{-cL^p}\leq C\lceil L \rceil e^{-cL^p} + C\sum_{i=\lceil L \rceil +1}^\infty e^{-ci^p}\leq De^{-dL^p} \qedhere
\]
\end{proof}

\begin{proof}[Proof of Theorem~\ref{thmCBP_Speed}]
    Consider~$\Bar{\mathbb E}$ the expectation operator associated to the probability measure~$\Bar{\mathbb P}$ as in~\eqref{eq_conditionalLaw_CBP}. By letting~$\mu_T:=\Bar{\mathbb E}[\tau_1]$ and~$\mu_S:=\Bar{\mathbb E}[B_{\tau_1}]$, we will show that~\eqref{real_speed_barrier} holds for~$\mathbf B=\mu_S/\mu_T$. First we note that~$B_{\tau_n}/\tau_n\rightarrow \mathbf B$ as~$n\rightarrow \infty$ almost surely. 
    
    Indeed, note that we can write~$\tau_n =\tau_0+ \sum_{k=1}^n(\tau_k-\tau_{k-1})$ and since~$\tau_0$ is finite almost surely, by the strong law of large numbers we have that~$\tau_n/n\rightarrow \mu_T$ as~$n\rightarrow \infty$ since the increments~$\tau_k-\tau_{k-1}$ are i.i.d. because of Lemma~\ref{lem_renewal}. Since~$r_{\tau_0}$ is by consequence also finite, a similar argument gives that~$B_{\tau_n}/n\rightarrow \mu_S$ as~$n\rightarrow \infty$ almost surely. Since~$\tau_n\geq n$ by definition, we have that~$\tau_n/n$ is bounded away from zero and thus~$(B_{\tau_n}/n)/(\tau_n/n)\rightarrow \mathbf B$ as~$n\rightarrow \infty$. 

    We now transfer this convergence along integer times, i.e., we claim that~$B_n/n\rightarrow \mathbf B$ as~$n\rightarrow \infty$ a.s. To do so, we will prove that for any~$\epsilon>0$, we have that there exist constants~$c,C,p>0$ independent of~$n$ for which one has:
    \begin{equation}\label{eq_speedIntegerTimes}
        \mathbb P \left(\left|\frac{B_n}{n} - \frac{B_{\sigma(n)}}{\sigma(n)}\right|>\epsilon\right) \leq Ce^{-cn^p}
    \end{equation}

    Since~$B_{\sigma(n)}/\sigma(n)\rightarrow\mathbf B$ as~$n\rightarrow \infty$ as it is a subsequence of a convergent sequence,~\eqref{eq_speedIntegerTimes} with the Borel-Cantelli lemma give us the desired claim. First, we note that we can bound the left-hand side of~\eqref{eq_speedIntegerTimes} by:

\begin{equation*}
    \underbrace{\mathbb P\left(\left|\frac{B_n}{n}-\frac{B_n}{\sigma(n)}\right|>\frac{\epsilon}{2}\right)}_{(1)}
    + 
    \underbrace{\mathbb P \left(\left|\frac{B_n}{\sigma(n)} - \frac{B_{\sigma(n)}}{\sigma(n)}\right|>\frac{\epsilon}{2}\right)}_{(2)}
\end{equation*}

We first bound the term~$(2)$. The event inside that probability implies that~$|B_n - B_{\sigma(n)}|>\frac{\epsilon}{2}\sigma(n)$, and therefore we can bound~$(2)$ by the following expression:
\begin{equation*}
    \mathbb P \left(\sigma(n) <\frac{n}{2}\right) + \mathbb P \left(|B_n- B_{\sigma(n)}|> \frac{\epsilon}{4}n\right)
\end{equation*}
and both terms are bounded by~$Ce^{-cn^p}$ due to Lemma~\ref{time_and_space_betweenRenewals}. At last, we bound~$(1)$. The event inside that probability implies that~$|B_n(n-\sigma(n))|>\frac{\epsilon}{2}n\sigma(n)$, and therefore we can bound it by the following sum:
\begin{equation}
    \underbrace{\mathbb P \left(\sigma(n)<\frac{n}{2}\right)}_{(\mathrm a)} + \underbrace{\mathbb P \left(|B_n|(n-\sigma_n)>\frac{\epsilon}{4}n^2\right)}_{(\mathrm b)}
\end{equation}

The term~$(a)$ is bounded by~$Ce^{-cn^p}$ again due to Lemma~\ref{time_and_space_betweenRenewals}. At last, for~$(b)$, we can bound it by:

\begin{equation}
\underbrace{\mathbb{P}\left( \left\{ |B_n| (n-\sigma(n)) > \frac{\epsilon}{4} n^2 \right\} \cap \left\{|B_n|\leq 2(r_{\rightarrow}+r_{\leftarrow})n\right\} \right) }_{(\mathrm b_1)} + \underbrace{\mathbb P\left(|B_n|> 2(r_{\rightarrow}+r_{\leftarrow})n\right)}_{(\mathrm b_2)}
\end{equation}

The term in~$(\mathrm b_1)$ is bounded by~$\mathbb P \left(n-\sigma(n)>\frac{\epsilon}{8(r_{\rightarrow} + r_{\leftarrow})}n\right)$, which is bounded by~$Ce^{-cn^p}$ once more due to Lemma~\ref{time_and_space_betweenRenewals}. The term in~$(\mathrm b_1)$ is bounded by~$Ce^{-cn}$ since the random variable~$|B_n|$ is bounded by the total number of arrivals in~$\mathcal I$ before time~$n$ and this number is bounded by a random variable with Poisson distribution with parameter~$(r_{\rightarrow}+r_{\leftarrow})n$.  

Finally, to conclude that almost surely~$B_t/t\rightarrow \mathbf B$ as~$t\rightarrow \infty$ we proceed as in the proof of Theorem 2.19 of~\cite{liggett1985interacting} noting that:
    \[
\mathbb P\!\left(\max_{n\le t\le n+1}(B_t-B_n)\ge \epsilon n\right)
\le Ce^{-cn^p},
\qquad
\mathbb P\!\left(\max_{n-1\le t\le n}(B_n-B_t)\ge \epsilon n\right)
\le Ce^{-cn^p}.
\]
    since both terms are bounded by the probability of a Poisson random variable with parameter~$\lambda$ to be larger than~$\epsilon n$. Because once more of the Borel-Cantelli lemma, we are done. 

    It remains to show that the same convergence also holds in~$L^1$. We claim that it is enough to show that the family~$\{B_n/n\,:\,n\in\mathbb N\}$ is uniformly integrable. Indeed, if that was the case, we would have that~$B_n/n\rightarrow \mathbf B$ as~$n\rightarrow \infty$ in~$L^1$ from the almost sure convergence we have just shown. Then, we could conclude also that the same convergence would hold for~$B_t/t$. To do that, we would have to show that for any~$\epsilon>0$ there exists~$t_0\in[0,\infty)$ large enough so that
    \begin{equation} \label{eq_target_L1}
        \mathbb E \left[\frac{B_t}{t}-\mathbf B\right]<\epsilon \quad \text{ for all }t\geq t_0
    \end{equation}
    
    But we have that~\eqref{eq_target_L1} is bounded above by:
    \begin{equation}
    \underbrace{\mathbb E \left[\left|\frac{B_t}{t}-\frac{B_{\lfloor t\rfloor}}{t}\right|\right]}_{(\mathrm c_1)} + \underbrace{\mathbb E \left[\left|\frac{B_{\lfloor t\rfloor}}{t} - \frac{B_{\lfloor t \rfloor}}{t}\right|\right]}_{(\mathrm c_2)}
         + \underbrace{\mathbb E\left[\frac{B_{\lfloor t \rfloor}}{t}- \mathbf B\right]}_{(\mathrm c_3)}
    \end{equation}
    and we prove that each of those terms can be made smaller than~$\epsilon /3 $ by choosing~$t$ sufficiently large. Indeed, the term in~$(\mathrm c_1)$ is bounded by~$\frac{1}{t}\mathbb E[B_t-B_{\lfloor t \rfloor}]\leq \frac{r_{\rightarrow}+r_{\leftarrow}}{t}$; the term in~$(\mathrm c_2)$ is bounded by~$\frac{\mathbb E[B_{\lfloor t\rfloor}]}{t\lfloor t \rfloor}\leq \frac{(r_{\rightarrow}+r_{\leftarrow})t}{t\lfloor t \rfloor}$; finally, the term in~$(\mathrm c_3)$ can be made smaller than~$\frac{\epsilon}{3}$ choosing~$t$ sufficiently large because~$B_n/n\rightarrow \mathbf B$ in~$L_1$ as~$n\rightarrow \infty$. 

    To prove that the family is uniformly integrable, we must show that for any~$\epsilon>0$ there exists~$L\in[0,\infty)$ such that
    \[
    \sup_{n\in\mathbb N}\mathbb E\left[\frac{|B_n|}{n}\mathds{1}_{\{|B_n|/n>L\}}\right]<\epsilon.
    \]
    
    Fix~$n\in\mathbb N$. Note that, if we let~$Y_n$ denote the random variable counting the number of arrivals in~$\mathcal I$ before time~$n$, we have that~$|B_n|\leq Y_n$, and thus: 
   \begin{align*}
\mathbb E\!\left[\frac{|B_n|}{n}\mathds{1}_{\{|B_n|/n>L\}}\right]
&\le \left(\frac{\mathbb E[(Y_n)^2]}{n^2}\right)^{1/2}
\mathbb P(Y_n>Ln)^{1/2}  \\
&\le \left(\frac{\tilde c n^2}{n^2}\right)^{1/2}
\left(\frac{(r_{\rightarrow}+r_{\leftarrow})n}{Ln}\right)^{1/2}.
\end{align*}
where the inequalities follow from Hölder’s inequality, the comparison with $Y_n$, and Markov’s inequality. In particular, the right-hand side can be made arbitrarily small by choosing $L$ sufficiently large (depending on $r_{\rightarrow}+r_{\leftarrow}$).
\end{proof}

\begin{proof}[Proof of Theorem~\ref{ThmCBP_Tightness}]
   Fix $t\ge0$ and $\epsilon>0$. It suffices to establish the result for $t\ge t_0$ with $t_0$ large. Indeed, for $t\le t_0$, the quantity $\ell_t-B_t$ is stochastically dominated by a Poisson random variable with parameter $(r_{\leftarrow}+r_{\rightarrow}+\lambda)t_0$, so $\mathbb P(\ell_t-B_t>L_1)<\epsilon/2$ for $L_1$ large enough. Choosing $L=\max\{L_1,L_2\}$, where $L_2$ controls the case $t\ge t_0$, yields the claim.
   
   Let~$\sigma(t)=\sup_{m\in\mathbb N}\{\tau_m\,:\,\tau_m\leq t\}$ be the time of the last renewal before~$t$. Because of Lemma~\ref{lem_timeSinceLastRenewal_CBP}, for any~$s\geq 0$ we have that~$\mathbb P\left(t-\sigma(t)>s\right)\leq Ce^{-cs^p}$ and let~$s$ be large enough so that this probability is smaller than~$\frac{\epsilon}{2}$ for all~$t\geq t_0$, so that~$t_0\geq s$. Let~$R:=r_{\leftarrow}+r_{\rightarrow}^1+\lambda$ and~$K=K(s)$ be large enough so that the probability of a Poisson random variable with parameter~$Rs$ being larger than~$KRs$ is smaller than~$\frac{\epsilon}{2}$. Let~$L_2=KRs$. For any~$t\geq t_0$:
 \begin{equation} \label{Target_Tightness}
       \mathbb P \left(\ell_t-B_t<L_2\right) \leq \mathbb P \left(\ell_t-B_t<L_2 \cap \{t-\sigma(t)\leq s\}\right) + \mathbb P \left(t-\sigma(t) > s\right)
   \end{equation}
   and both terms on the right hand side of~\eqref{Target_Tightness} are bounded by~$\frac{\epsilon}{2}$ by the choice of~$s$.
\end{proof}

\begin{proof}[Proof of Theorem~\ref{ThmCBP_CLT_Speed}]
    Let~$(Z^0_t)_{0\leq t< \tau_0}$  for~$t\in[0,\tau_0)$ and let~$(Z^n_t)_{\tau^{n-1}\leq t<\tau^n}$ for~$n\in\mathbb N$ be given by~$Z^n_t = i_t$ for~$n\in\mathbb N_0$. Because of Lemma~\ref{lem_renewal}, we are under conditions~$(1)$ and~$(2)$ of Lemma~\ref{lem_CLT_RenewalProcess}. Condition~$(3)$ of Lemma~\ref{lem_CLT_RenewalProcess} is satisfied because of Lemma~\ref{time_and_space_betweenRenewals}. Thus, the conclusion of Lemma~\ref{lem_CLT_RenewalProcess} follows. 
\end{proof} 
\section{Multitype Contact Process} \label{sec_MCP}

\paragraph{} The structure of this section largely mirrors that of Section~\ref{sec_CBP}. Given its complexity, we begin with a brief overview. In Section~\ref{sec_GCMCP}, we introduce the multitype contact process using a graphical construction. Section~\ref{section_patchworkMCP} then develops the patchwork construction for the interface process. This section follows a layout similar to Section~\ref{section_PatchworkCBP}, where we developed the patchwork construction for the contact-and-barrier process, but incorporates additional challenges. For instance, we must now account for a \emph{pair of trails} for the initial configuration instead of a single trail, and also for two special infection paths instead of one, along with several other technical modifications. Finally, Section~\ref{sec_proofMCP} presents the proofs of the main theorems for the multitype contact process.

\subsection{Graphical construction} \label{sec_GCMCP}

\paragraph{}In order to construct the multitype contact process, we will consider a graphical construction~$\mathcal H_{\mathrm{pair}}$ given by a pair~$\mathcal H_{\mathrm{pair}} =( \mathcal H_1, \mathcal H_2)$ where~$\mathcal H_1 = \left((D^x_1)_{x \in \mathbb{Z}},\ (D^{x,y}_1)_{x,y \in \mathbb{Z},\, |x - y| = 1}\right)$ and~$\mathcal H_2=\left((D_2^x)_{x \in \mathbb{Z}},\ (D_2^{x,y})_{x,y \in \mathbb{Z},\, |x - y| = 1}\right)$ are graphical constructions on~$\mathbb Z\times [0,\infty)$ for contact processes with parameter~$\lambda_1$ and~$\lambda_2$, respectively. Given an initial configuration~$\xi_0$, we will construct the multitype contact process~$(\xi_t)_{t\geq 0}$ started from~$\xi_0$ as a limiting process of multitype contact processes started from the initial configuration restricted to a box. 

To do so, fix~$n\in\mathbb N$ and consider~$\xi_0^n=\xi_0\mathds1_{[-n,n]}$ the restriction of~$\xi_0$ to the box~$[-n,n]$. For~$k\in\mathbb N_0$, let~$\mathcal J^k =\{t\in D_{1}^x\cup D_2^x\cup D_1^{x,y}\cup D^{x,y}_2\,:\,x,y\in [-n-k,n+k]\}$ be the collection of arrivals of the graphical construction~$\mathcal H'$ on sites in~$[-n-k, n+k]$. 

We will define a strictly increasing sequence of stopping times~$(\sigma_k)_{k\in\mathbb N_0}$. Let~$\sigma_0=0$ and define~$\sigma_k=\inf\{t>\sigma_{k-1}\,:\,t\in\mathcal J^k\}$ for~$k\in\mathbb N$. Let~$\xi_s^n=\xi_0^n$ for all~$s\in[0,\sigma_1)$. Suppose that the process is defined on~$[0,\sigma_k)$ for some~$k\in\mathbb N$. We will define the process in the time interval~$[\sigma_k,\sigma_{k+1})$, which then will make the process well-defined for all positive times~$t$ since the sequence of stopping times is strictly increasing. Given~$i\in\{0,1,2\}$,~$x\in\mathbb Z$ and~$\xi\in\{0,1,2\}^{\mathbb Z}$, let~$\xi^{i\rightarrow x}$ be the configuration obtained from~$\xi$ by assigning state~$i$ to the site~$x$ and leaving everything else unchanged, i.e., 
\begin{equation*}
    \xi^{i\rightarrow x}(y)=
    \begin{cases}
        i &\quad \text{if }y=x\\
        \xi(y) &\quad\text{if }y\neq x
    \end{cases}
\end{equation*}

We define~$\xi_{\sigma_k}^n$ by considering the following possibilities:

\begin{itemize}
    \item if~$\sigma_k\in D_1^x$ and~$\xi_{\sigma_k-}^n(x)=1$, then~$\xi_{\sigma_k}^n = \xi_{\sigma_{k}-}^{0\rightarrow x}$ 
    \item if~$\sigma_k\in D_2^x$ and~$\xi_{\sigma_k-}^n(x)=2$, then~$\xi_{\sigma_k}^n = \xi_{\sigma_{k}-}^{0\rightarrow x}$ 
    \item if~$\sigma_k\in D_1^{x,y}$,~$\xi_{\sigma_k-}^n(x)=1$ and~$\xi_{\sigma_k-}^n(y)=0$, then~$\xi_{\sigma_k}^n = \xi_{\sigma_{k}-}^{1\rightarrow y}$
    \item if~$\sigma_k\in D_2^{x,y}$,~$\xi_{\sigma_k-}^n(x)=2$ and~$\xi_{\sigma_k-}^n(y)=0$, then~$\xi_{\sigma_k}^n = \xi_{\sigma_{k}-}^{2\rightarrow y}$
\end{itemize}

At last, let~$\xi_s^n=\xi_{\sigma_k}^n$ for~$s\in[\sigma_k,\sigma_{k+1})$. Therefore, for a given graphical construction~$\mathcal H_{\mathrm{pair}}=(\mathcal H_1,\mathcal H_2)$ and a given initial configuration~$\xi_0$, we can consider a sequence~$(\xi_t^n)_{t\geq 0}$ of multitype processes with initial configuration~$\xi^n_0$ all constructed using~$\mathcal H_{\mathrm{pair}}$. Regarding those processes, we have the following result:

\begin{lemma} \label{lem_TruncatedLimit}
    Let~$x\in\mathbb Z$ and~$t\in[0,\infty)$ be given. Then,~$\lim_{n\rightarrow \infty}\xi_t^n(x)$ exists.
\end{lemma}

The proof of Lemma~\ref{lem_TruncatedLimit} is standard and omitted. In view of this lemma, we define the multitype contact process $(\xi_t)_{t\ge0}$ by $\xi_t(x)=\lim_{n\to\infty}\xi_t^n(x)$. Throughout, we assume that $(\xi_t)_{t\ge0}$ is constructed from an initial configuration $\xi_0$ via the graphical representation $\mathcal H_{\mathrm{pair}}=(\mathcal H_1,\mathcal H_2)$.

\begin{definition}[Active infection paths]
    Let~$i\in\{1,2\}$. We say that~$\gamma:I\rightarrow \mathbb Z$ is an active infection path (in~$\mathcal H_i$, in case we want to highlight its type) if the following holds:
    \begin{itemize}
        \item ~$\gamma$ is an infection path in the graphical construction~$\mathcal H_i$ in the sense of Definition~\ref{def_InfectionPath}
        \item ~$\xi_t(\gamma(t))= \xi_0(\gamma(0))=i$ for all~$t\in I$.
    \end{itemize} 
    
    For~$s,t\in[0,\infty)$ with~$s\leq t$ and~$x,y\in\mathbb Z$, we write~$(y,s)\overset{\mathcal{H}_i,\mathrm{ active}}{\rightsquigarrow}(x,t)$ to denote the existence of an active infection path connecting~$(y,s)$ to~$(x,t)$. 
\end{definition}

\begin{lemma}\label{lem_Typei_iff_activei}
    Let~$x\in\mathbb Z$,~$t\in[0,\infty)$ and ~$i\in\{1,2\}$ be given. Then~$\xi_t(x)=i$ if and only if~$(y,0)\overset{\mathcal H_i,\text{active}}{\rightsquigarrow} (x,t)$ for some~$y\in\mathbb Z$. 
\end{lemma}

\begin{proof}
    The result follows by establishing the claim for the truncated processes.
\end{proof}
Regarding active infection paths, we make the following important remark.  

\begin{remark}
    Consider~$(\xi_t)_{t\geq 0}$ a multitype contact process started from some~$\xi_0\in\mathcal C$ and let~$i\in\{1,2\}$. The following is true:
    \begin{itemize}
        \item If~$\gamma_1:I\rightarrow\mathbb Z$ is an active infection path in~$\mathcal H_1$ and~$\gamma_1':I\rightarrow\mathbb Z$ is an infection path in~$\mathcal H_1$ with~$\gamma_1(t)\leq \gamma_1'(t)$ for all~$t\in I$, then~$\gamma_1'$ is also active. 
        \item If~$\gamma_2:I\rightarrow\mathbb Z$ is an active infection path in~$\mathcal H_2$ and~$\gamma_2':I\rightarrow\mathbb Z$ is an infection path in~$\mathcal H_2$ with~$\gamma_2(t)\geq \gamma_2'(t)$ for all~$t\in I$, then~$\gamma_2'$ is also active. 
    \end{itemize}
\end{remark}

The graphical construction of the multitype contact process also exhibits the following property, which will prove useful later.

\begin{lemma} \label{lem_sameInterfaceMCP}
    Let~$\xi_1,\xi_2\in\{0,1,2\}^{\mathbb Z}$ be two configurations such that:
    \begin{itemize}
        \item ~$\xi_1(x),\xi_2(x)\in\{0,1\}$ if~$x\leq 0$ and~$\xi_1(x),\xi_2(x)\in\{0,2\}$ if~$x\geq 1$;
        \item there exists~$k_1,k_2\in\mathbb N$ such that~$\xi_1\mathds{1}_{[-k_1,k_2]}=\xi_2\mathds{1}_{[-k_1,k_2]}$.
    \end{itemize}
    
    Let~$(\xi^1_t)_{t\geq 0}$ and~$(\xi_t^2)_{t\geq 0}$ be two multitype contact processes started from~$\xi_1$ and~$\xi_2$, respectively, constructed using the same graphical construction~$\mathcal H_{\mathrm{pair}} = (\mathcal H_1,\mathcal H_2)$. Let~$(I_t^1)_{t\geq 0}=(r_t^1,\ell_t^1)_{s\geq 0}$ and~$(I_t^2)_{t\geq 0}=(r_t^2,\ell_t^2)_{s\geq 0}$ be the interface process associated with~$(\xi_t^1)_{t\geq 0}$ and~$(\xi_t^2)_{t\geq 0}$, respectively. Let~$t'\in[0,\infty)$ be given. Suppose that:
    \begin{itemize}
        \item there exists~$x\in[-k_1,0]$ such that~$(x,0)\overset{\mathcal H_1,\mathrm{active}}{\rightsquigarrow}\mathbb Z\times \{t'\}$ for the process~$(\xi^1_t)_{t\geq 0}$;
        \item there exists~$y\in[1,k_2]$ such that~$(y,0)\overset{\mathcal H_2,\mathrm{active}}{\rightsquigarrow}\mathbb Z\times \{t'\}$ also for the process~$(\xi^2_t)_{t\geq 0}$
    \end{itemize}
     Then, it follows that~$(I_s^1)_{0\leq s \leq t'} = (I_s^2)_{0\leq s \leq t'}$ almost surely.  
\end{lemma}

\begin{proof}
    One may argue by contradiction by letting $P:=\inf\{s\in[0,t']:\, I_s^1\neq I_s^2\}$ denote the first time at which the property fails. The remainder of the argument is not particularly instructive and is therefore omitted.
\end{proof}

\begin{remark} \label{rk__sameInterfaceMCP}
    For~$0\leq t\leq t'$, if we let~$\gamma_r^1:[0,t]\rightarrow\mathbb Z$ be any active infection path in~$\mathcal H_1$ for~$(\xi_t^1)_{t\geq 0}$ with~$\gamma_r^1(0)\in[-k_1,0]$ and~$\gamma_r^1(t)=r_t^1$, it follows that~$\gamma_r^1$ is also active for the process~$(\xi_t^2)_{t\geq 0}$. Similarly, if we let~$\gamma_l^1:[0,t']\rightarrow\mathbb Z$ be an active infection path in~$\mathcal H_2$ for~$(\xi_t^2)_{t\geq 0}$ with~$\gamma_l^1(0)\in[1,k_2]$ and~$\gamma_l^1(t)=\ell_t^1$, it follows that~$\gamma_l^1$ is also active for the process~$(\xi_t^2)_{t\geq 0}$.
\end{remark}
\subsection{Patchwork construction of the interface process} \label{section_patchworkMCP}

\paragraph{} The structure of this section is the same as the one from Section~\ref{section_PatchworkCBP}, where in each of the following subsections we replicate the patchwork construction done for the contact-and-barrier process with the required modifications to fit the multitype contact process. 

\subsubsection{Pair of trails and interface measures}\label{sec_pairOfTrails}

\paragraph{}We begin by defining a pair of trails. As in the case of the contact-and-barrier process, these trails will be used to generate occupied sites. The key difference here is that each trail produces only one type of individual, since we must now distinguish between the two types.
\begin{definition}[Pair of trails]
    Let~$\mathscr A'$ be the collection of pairs of sets~$(A_1,A_2)$ where both~$A_1,A_2\in\mathscr A$. 
\[
\mathcal R':= \left\{\begin{array}{l}
(\mathrm g_1,\mathrm g_2)\in \left(\mathrm P(\mathbb Z \times (-\infty,0])\right)^2\,:\mathrm g_i \text{ is of the form }\mathrm g_i = \bigcup_{x \in \mathbb Z} (\{x\} \times A^i_x) \\ \text{ for }i\in\{1,2\}\text{ with } (A^1_x,A^2_x) \in \mathscr A' \text{ for all }x,  \text{ with  }(0,0)\in\mathrm g_1\text{ and } (1,0) \in \mathrm g_2
\end{array}\right\}.
\]
\end{definition}

Similarly to what we have done in Definition~\ref{def_trail}, we endow~$\mathscr A'$ with the product~$\sigma$-algebra and then~$\mathcal R'$ with the infinite product~$\sigma$-algebra.

\begin{definition}
    For two c\`adl\`ag functions~$\gamma_1,\gamma_2:(-\infty,t]\rightarrow \mathbb Z$ with~$t>0$, let~$(\mathrm g_1 \sqcup \gamma_1,\mathrm g_2\sqcup \gamma_2)$ be the pair defined by:
    \begin{align*}
        \mathrm g_1 \sqcup \gamma_1&=\{(x-\gamma_1(t),s-t)\,:\,(x,s)\in\gamma_1\cup \overline{\mathrm{Graph}(\gamma_1)}\} \\
        \mathrm g_2 \sqcup \gamma_2&=\{(x-\gamma_2(t)+1,s-t)\,:\,(x,s)\in\gamma_2\cup \overline{\mathrm{Graph}(\gamma_2)}\}
    \end{align*}
\end{definition}

In words, we describe the pair~$(\mathrm g_1\sqcup \gamma_1,\mathrm g_2\sqcup \gamma_2)$. The first entry~$\mathrm g_1\sqcup \gamma_1$ is the subset of~$\mathbb Z\times (-\infty, 0]$ obtained by appending the closure of the graph of~$\gamma$ to~$\mathrm g_1$ and translating it so that~$(\gamma_1(t),t)$ becomes the origin; the second entry~$\mathrm g_2\sqcup \gamma_2$ can be understood similarly, it is the subset of~$\mathbb Z \times(-\infty,0]$ obtained by appending the closure of the graph~$\gamma_2$ to~$\mathrm g_2$ and translating it so that~$(\gamma_2(t),t)$ is taken to~$(1,0)$. 

For a pair of trails~$(\mathrm{g}_1,\mathrm{g}_2)\in \mathcal R'$, we can construct a configuration~$\xi\in\{0,1,2\}^{\mathbb Z}$ in the following way. Let~$\mathcal H_{\mathrm{pair}}$ be a graphical construction of the multitype contact process on~$\mathbb Z \times (-\infty, 0]$. Let:

\begin{equation}\label{eq_def_mu_g1g2}
\xi(x)=\begin{cases}
\mathds{1}(\{-\infty\}\cup \mathrm g_1\overset{\mathcal H_1}{\rightsquigarrow} (x,0)\})&\text{if } x \leq 0;\\
2\times \mathds{1}(\{-\infty\}\cup \mathrm g_2\overset{\mathcal H_2}{\rightsquigarrow} (x,0)\})&\text{if } x \geq  1.
\end{cases}
\end{equation}

In particular, since for any pair of trails~$(\mathrm g_1,\mathrm g_2)\in\mathcal R'$ we have that~$(0,0)\in\mathrm g_1$ and~$(1,0)\in\mathrm g_2$, it follows that for~$\xi$ as in~\eqref{eq_def_mu_g1g2}, one has~$\xi(0)=1$ and~$\xi(1)=2$.

\begin{definition}
For~$(\mathrm g_1, \mathrm g_2)\in \mathcal R'$, we let~$\mu_{(\mathrm g_1,\mathrm{g}_2)}$ be the law of a configuration~$\xi$ obtained from~$(\mathrm g_1,\mathrm g_2)$ as in \eqref{eq_def_mu_g1g2}. 
\end{definition}

\begin{definition}[Law~$P_{\xi}^0$ of interface process]
    We let~$P^0_{\xi}$ be the distribution of the interface process~$(I_t)_{t \ge 0}$ for the multitype contact process~$(\xi_t)_{t \ge 0}$ started from~$\xi_0=\xi$.
\end{definition}

\begin{definition}[Law~$P_{(\mathrm g_1,\mathrm g_2)}$ of interface process induced by a pair of trails] \label{measuresMultitype}
 We let~$P_{(\mathrm g_1,\mathrm{g}_2)}$ be the distribution of the interface process~$(I_t)_{t \ge 0}$ for the contact-and-barrier process~$(\xi_t)_{t \ge 0}$ started from a random configuration~$\xi_0$ with distribution~$\mu_{(\mathrm g_1,\mathrm{g}_2)}$, so that\[
 P_{(\mathrm g_1,\mathrm g_2)} = \int P_\xi^0\mu_{(\mathrm g_1,\mathrm g_2)}(d\xi)
 \]
\end{definition}

\subsubsection{Patchwork elements} \label{patchworkElementsMCP}

\paragraph{} The goal of this section is to replicate the structure of Section~\ref{patchworkElements} for the multitype contact process. Throughout this section, we fix a pair of trails~$(\mathrm{g}_1,\mathrm{g}_2)\in\mathcal{R}'$. We take a probability measure~$\mathbb P$ under which a graphical construction~$\mathcal H_{\mathrm{pair}} = (\mathcal H_1,\mathcal H_2)$ is defined (for both positive and negative times) and we consider multitype contact process~$(\xi_t)_{t\geq 0}$ started from a random configuration distributed as~$\mu_{(\mathrm g_1, \mathrm g_2)}$ constructed using~$\mathcal H_{\mathrm{pair}}$ as in~\eqref{eq_def_mu_g1g2}. Let:
\begin{align} \label{eq_defS1}
    \mathcal S_1 &:=\{x\in\mathbb Z\,:\,x\leq  0\text{ and }-\infty \overset{\mathcal H_1}{\rightsquigarrow}(x,0)  \} \subseteq \{x\in\mathbb Z\,:\,\xi_0(x)=1\} \\ \label{eq_defS2}
    \mathcal S_2 &:=\{x\in\mathbb Z\,:\,x\geq 1\text{ and }-\infty \overset{\mathcal H_2}{\rightsquigarrow}(x,0)  \} \subseteq \{x\in\mathbb Z\,:\,\xi_0(x)=2\}
\end{align}
where the inclusion follows from~\eqref{eq_def_mu_g1g2}. 

\begin{definition}[Adjacency time $T$]\label{defTMCP}
Define:
\begin{equation}\label{eq_def_TMCP}
    T:=\inf \left\{ 
\begin{array}{l} 
t \geq 1 : \text{ there exist }x_1\in\mathcal S_1 \text{ and }x_2\in\mathcal S_2 \text{ such that } \\
(x_1,0)\overset{\mathcal{H}_1,\mathrm{ active}}{\rightsquigarrow}(\ell_t-1,t) \text{ and } (x_2,0)\overset{\mathcal{H}_2,\mathrm{ active}}{\rightsquigarrow}(r_t+1,t)
\end{array}
\right\}
\end{equation}
\end{definition}

For~$T$ as in Definition~\ref{defTMCP}, we also let
\begin{align} \label{eq_X1}
    \mathbf{X}_1&:=\sup\{x_1\in\mathcal S_1\,:\,(x_1,0)\overset{\mathcal{H}_1,\mathrm{ active}}{\rightsquigarrow} \mathbb Z\times \{T\}\} \\ \label{eq_X2}
    \mathbf X_2&:=\inf\{x_2\in\mathcal S_2\,:\,(x_2,0)\overset{\mathcal{H}_2,\mathrm{ active}}{\rightsquigarrow} \mathbb Z\times \{T\}\}
\end{align}

By a crossing paths argument, it is easy to see that,~$(\mathbf X_1,0)\overset{\mathcal{H}_1,\mathrm{ active}}{\rightsquigarrow} (r_T,T)$, i.e., that~$r_T$ is a descendant of~$\mathbf X_1$, and that~$(\mathbf X_2,0)\overset{\mathcal{H}_2,\mathrm{ active}}{\rightsquigarrow} (\ell_T,T)$, i.e., ~$\ell_T$ is a descendant of~$\mathbf X_2$. Regarding those random variables~$T$,~$\mathbf X_1$ and~$\mathbf X_2$ we have the following result:

\begin{lemma} \label{TBehavesWellMCP}
    There exists constants~$c,C,p>0$ (independent of the initial pair of trails~$(\mathrm g_1,\mathrm g_2)$) such that:

    \begin{equation}\label{target_TBheavesWellMCP}
        \mathbb P\left(\max\{T,-\mathbf X_1,\mathbf X_2\}>s\right)\leq Ce^{-cs} \quad \text{ for all }s\geq 0
    \end{equation}
\end{lemma}

The proof of the above lemma is postponed to Section~\ref{sec_proofElementsMCP}. 

\begin{definition}[Depth]\label{DefDepthMCP} Define the \emph{depth} as the random variable
    \[D:=\sup\{t\geq 0\,:\,\mathbb Z \times \{-t\}\overset{\mathcal H_1\text{ or }\mathcal H_2}{\rightsquigarrow}\left(\{\mathbf X_1,\dots,\mathbf X_2\}\setminus (\mathcal S_1 \cup \mathcal S_2)\right) \times \{0\}\}\]
with~$D=-\infty$ in case~$\{\mathbf X_1,\dots,\mathbf X_2\}\setminus (\mathcal S_1 \cup \mathcal S_2)=\varnothing$
\end{definition}

\begin{lemma} \label{DepthBehavesWellMCP}
    There exists a constant~$C$ (independent of the pair of trail~$(\mathrm g_1,\mathrm g_2)$) such that
    \[
    \mathbb P \left(D > d\right) \leq Ce^{-\sqrt d}\quad \text{ for all }d\geq 0
    \]
\end{lemma}

The proof of the above lemma is also postponed to Section~\ref{sec_proofElementsMCP}. 

\begin{definition}[Pair of special infection paths]\label{pairSpecialPaths}
    Let~$\Gamma_1:(-\infty,T]$ be the infection path in~$\mathcal H_1$ and~$\Gamma_2:(-\infty,T]$ be the infection path in~$\mathcal H_2$ defined as follows:
    \begin{itemize}
        \item in~$(-\infty,0]$,~$\Gamma_1$ is the rightmost infection path in~$\mathcal H_1$ from~$-\infty$ to~$(\mathbf X_1,0)$ and~$\Gamma_2$ is the leftmost infection path in~$\mathcal H_2$ from~$-\infty$ to~$(\mathbf X_2,0)$;
        \item in~$[0,T]$,~$\Gamma_1$ is the rightmost active infection path in~$\mathcal H_1$ from~$(\mathbf X_1,0)$ to~$(r_T,T)$ and~$\Gamma_2$ is the leftmost active infection path in~$\mathcal H_2$ from~$(\mathbf X_2,0)$ to~$(\ell_T,T)$. 
    \end{itemize}
\end{definition}

Note that the sites~$\Gamma_1(0)=\mathbf X_1$ and~$\Gamma_2(0)=\mathbf X_2$ would be occupied, respectively, by a particle of type~$1$ and by a particle of type~$2$ regardless of whether we changed the pair of trails~$(\mathrm g_1, \mathrm g_2)$ used to build the initial configuration. Moreover, note that~$\Gamma_1(t)<\Gamma_2(t)$ for all~$t\in[0,T]$ by a crossing paths argument.

\begin{proposition}\label{lemma_lawatTimeTMCP}
    The law of~$\{\xi_T(x - r_T) : x \in \mathbb{Z}\}$ conditioned on the~$\sigma$-algebra generated by~$T$,~$\Gamma_1$,~$\Gamma_2$,~$\mathrm{Left}_{\Gamma_1}(\mathcal H_1)$ and~$\mathrm{Right}_{\Gamma_2}(\mathcal H_2)$ is~$\mu_{(\mathrm g_1\sqcup \Gamma_1, \mathrm g_2\sqcup \Gamma_2)}$. 
\end{proposition}

The proof of Proposition~\ref{lemma_lawatTimeTMCP} is also postponed to Section~\ref{sec_proofElementsMCP}. 

\begin{definition}[Law~$\mathrm{Q_{(\mathrm{g}_1, \mathrm{g}_2)}}$]
We let~$\mathrm{Q_{(\mathrm{g}_1, \mathrm{g}_2)}}$ be the law of the 5-tuple

\begin{equation} \label{XiDefMCP}
    \Xi := (T,\Gamma_1, \Gamma_2,D,(I_t)_{0 \le t \le T})
\end{equation}
for the multitype contact process~$(\xi_t)_{t \ge 0}$ started from~$\xi_0 \sim \mu_{(\mathrm g_1, \mathrm g_2)}$.
\end{definition}

\begin{definition}
     Define
        \begin{equation}\label{eq_def_ItprimeMCP}
            I_t':=I_{T+t}-I_T, \quad t \ge 0.
    \end{equation}
\end{definition}

\begin{corollary} \label{FirstPatchworkInterfaceMCP}
   Conditionally on~$\Xi$, the law of~$(I_t')_{t \ge 0}$ is~$P_{(\mathrm g_1\sqcup \Gamma_1,\mathrm g_2\sqcup \Gamma_2)}$.
\end{corollary}

The proof of the above corollary is also postponed to Section~\ref{sec_proofElementsMCP}. 
\begin{definition}
    Let~$Q_{(\mathrm g_1, \mathrm g_2)}'$ denote the law of the pair~$(\Xi,(I_t')_{t \ge 0})$ for the multitype contact process~$(\xi_t)_{t \ge 0}$ started from~$\xi_0 \sim \mu_{(\mathrm g_1, \mathrm g_2)}$.
\end{definition}

\begin{lemma} \label{onePatchMCP}
    If $(\Xi,(I_t')_{t\ge 0}) = ((T,(\Gamma_1, \Gamma_2),D,(I_t)_{0 \le t \le T}),(I_t')_{t \ge 0}) \sim Q_{(\mathrm{g}_1, \mathrm g_2)}'$, then it follows that $\mathrm{Sew}((I_t)_{0 \le t \le T}, (I_t')_{t \ge 0}) \sim P_{(\mathrm{g}_1, \mathrm g_2)}$.  
\end{lemma}

\begin{proof}
    Follows readily from the definition of those objects as once more we have used the same graphical construction to build them. 
\end{proof}

\subsubsection{Proof of properties of patchwork elements} \label{sec_proofElementsMCP}

\begin{proof}[Proof of Lemma~\ref{TBehavesWellMCP}]
    Let~$\lambda_1, \lambda_2 > \lambda_c$ be given but fixed, and let~$\alpha_1 = \alpha_1(\lambda_1)$ and~$\alpha_2 = \alpha_2(\lambda_2)$ denote the associated speeds of the contact process with parameters~$\lambda_1$ and~$\lambda_2$, respectively, as defined in~\eqref{speedCP}. Let~$a>4\max\{\lambda_1,\lambda_2\}$ and let~$t:=\frac{2a}{\alpha_1+\alpha_2}$ so that~$-a+\alpha_1t = a - \alpha_2t$. Let
    \begin{align*}
        E_1 &:= \left\{\max\{x\in\mathbb Z\,:\,\exists\, s\in[0,1]\text{ such that }(-\infty,-a]\}\overset{\mathcal H_1}{\rightsquigarrow}(x,s)<-\frac{a}{2} \right\}, \\
        E_2 &:= \left\{\min \{x\in\mathbb Z\,:\,\exists\, s\in[0,1] \text{ such that }[a,\infty)\overset{\mathcal H_2}{\rightsquigarrow}(x,s)\}>\frac{a}{2}\right\},
    \end{align*}
    and note that for~$i\in\{1,2\}$, one has that~$\mathbb P(E_i^c)<Ce^{-ca}$ since the probability of this event is bounded by the probability of a Poisson random variable with parameter~$\lambda_i$ to be larger than~$2\lambda_i$. Consider the good event~$G_0:=E_1\cap E_2$, so that~$\mathbb P (G_0)>1-Ce^{-ca}$. Let also
    \begin{align*}
        G_1&:=\left\{\left(\mathcal S_1\cap [-2a,-a]\right)\times \{0\}\overset{\mathcal H_1}{\rightsquigarrow} \left[-a+\alpha_1t,\infty\right)\times\{t\}\right\} \\
        G_2&:=\left\{\left(\mathcal S_2\cap [a,2a]\right)\times \{0\}\overset{\mathcal H_2}{\rightsquigarrow} \left(-\infty,a-\alpha_2t\right]\times\{t\}\right\}
    \end{align*}

    We claim that~$G_0\cap G_1 \cap G_2 \subseteq \left\{\mathbf X_1\geq -2a,\, \mathbf X_2 \leq 2a,\, T<t\right\}$. Indeed, if~$G_1$ occurs, then there exists~$x_1\in\mathcal S_1\cap [-2a,-a]$ and~$\gamma_1:[0,t]\rightarrow \mathbb Z$ an infection path in~$\mathcal H_1$ such that~$\gamma(0)=x_1$ and~$\gamma_1(t)\geq -a+\alpha_1 t$. Similarly, if~$G_2$ occurs, then there exists~$x_2\in\mathcal S_1\cap [a,2a]$ and~$\gamma_2:[0,t]\rightarrow \mathbb Z$ an infection path in~$\mathcal H_2$ such that~$\gamma(0)=x_2$ and~$\gamma_2(t)\leq a-\alpha_2 t$. By the choice of~$t$ we have that~$\gamma_1(t)\geq \gamma_2(t)$, so we can consider~$t'$ the first moment where those infection paths are adjacent, i.e.,~$\gamma_1(t')+1=\gamma_2(t')$. Moreover, if~$G_0$ occurs, we have that the restriction of those infection paths to the time interval~$[0,1]$ are active, and therefore~$T\leq t'< t$. 

    For~$i\in\{1,2\}$, we have that~$\mathbb P(G_i) >1-Ce^{-ca}$ because of Corollary~\ref{CorRightmostWellBehaved}. Thus
    \[\mathbb P\left(\mathbf X_1\geq -2a,\,\mathbf X_2\leq 2a,\,T<\frac{2a}{\lambda_1+\lambda_2}\right)>1-Ce^{-ca}-e^{ca}\]
    and we get the desired result by a change of constant if necessary. 
\end{proof}

\begin{proof}[Proof of Lemma~\ref{DepthBehavesWellMCP}]
    The argument follows closely the proof of Lemma~\ref{DepthBehavesWell}. For each~$x\in\mathbb Z$, we define
    \[
    \sigma_x':=\sup\left\{t\geq 0\,:\,\mathbb Z\times \{-t\}\overset{\mathcal H_1 \text{ or }\mathcal H_2}{\rightsquigarrow}(x,0)\right\}
    \]
    
    Again by duality and~\eqref{eq_small_cluster}, we obtain that~$\mathbb P\left(t<\sigma_x'<\infty\right)<e^{-ct}$ for some~$c>0$, and all~$x$ and~$t\geq 0$. The event~$\{D>d\}$ is the same as the event where there exists~$x\in\{\mathbf X_1,\dots,\mathbf X_2\}$ such that~$t<\sigma_x'<\infty$. Because of an union bound and Lemma~\ref{TBehavesWellMCP}, we have
    \[\mathbb P(D>d)\leq \mathbb P\left(\mathbf X_1<-d \text{ or }\mathbf X_2>d\right) + 2de^{-cd} \leq Ce^{-c\sqrt d}\]
    for all~$d>0$ for some constant~$C>0$ large enough. 
\end{proof}

We carry on working under the same probability space where a graphical construction~$\mathcal H_{\mathrm{pair}}$ for the multitype contact process is defined. We consider a ``truncated'' version of the initial configuration given as in~\eqref{eq_def_mu_g1g2} by letting
\[
\xi_0^{(n)} = \begin{cases}
\mathds{1}(\{\mathbb Z\times \{-n\}\}\cup \mathrm g_1\overset{\mathcal H_1}{\rightsquigarrow} (x,0)\})&\text{if } x \leq 0;\\
2\times \mathds{1}(\{\mathbb Z\times \{-n\}\}\cup \mathrm g_2\overset{\mathcal H_2}{\rightsquigarrow} (x,0)\})&\text{if } x \geq 1 .
\end{cases}
\]

Similarly to~\eqref{eq_defS1} and~\eqref{eq_defS2}, we let
\begin{align*}
\mathcal S_1^{(n)}&:=\{x\in\mathbb Z\,:\,x\leq 0 \text{ and }\mathbb Z\times \{-n\}\overset{\mathcal H_1}{\rightsquigarrow}(x,0)\} \\
\mathcal S_2^{(n)}&:=\{x\in\mathbb Z\,:\,x\geq  1 \text{ and }\mathbb Z\times \{-n\}\overset{\mathcal H_2}{\rightsquigarrow}(x,0)\}
\end{align*}

The following is easily checked. 
\begin{claim}\label{cl_NkMCP}
    Almost surely, for all $k \in \mathbb N$ there exists $N(k)$ such that for all $n \ge N(k)$ we have
    \[\xi_0^{(n)} \cdot \mathds{1}_{[-k,k]} = \xi_0 \cdot \mathds{1}_{[-k,k]} \qquad \text{and} \qquad \mathcal S_i^{(n)} \cap [-k,k] = \mathcal S_i \cap [-k,k]\text{ for } i\in\{1,2\}.\] 
\end{claim}

Let~$(\xi_t^{(n)})_{t\geq 0}$ be the multitype contact process constructed evolved using the positive part of~$\mathcal H_{\mathrm{pair}}$ started from~$\xi_0^{(n)}$. We can then define the analogue of Definition~\ref{defTMCP} for this process started from this truncated configuration in the following way:

\[
T^{(n)}:=\inf \left\{ 
\begin{array}{l}
t \geq 1 : \text{ there exist }x_1\in\mathcal S_1^{(n)} \text{ and }x_2\in\mathcal S_2^{(n)} \text{ such that } \\
(x_1,0)\overset{\mathcal{H}_1,\mathrm{ active}}{\rightsquigarrow}(\ell_t^{(n)}-1,t) \text{ and } (x_2,0)\overset{\mathcal{H}_2,\mathrm{ active}}{\rightsquigarrow}(r_t^{(n)}+1,t)
\end{array}
\right\}
\]
where~$r_t^{(n)}$ and~$\ell_t^{(n)}$ are the positions of the rightmost particle of type~$1$ and the leftmost particle of type~$2$ for the process~$(\xi_t^{(n)})_{t\geq 0}$. Consider also the analogues of~\eqref{eq_X1} and~\eqref{eq_X2}:
\begin{align*}
    \mathbf X^{(n)}_1 = \sup \left\{x_1\in\mathcal S_1^{(n)}\,:\,(x_1,0)\overset{\mathcal H_1,\mathrm{ active}}{\rightsquigarrow}\mathbb Z\times \{T^{(n)}\}\right\} \\
    \mathbf X^{(n)}_2 = \sup \left\{x_2\in\mathcal S_2^{(n)}\,:\,(x_2,0)\overset{\mathcal H_2,\mathrm{ active}}{\rightsquigarrow}\mathbb Z\times \{T^{(n)}\}\right\}
\end{align*}

Consider~$(I_t^{(n)})_{t\geq 0} = (r_t^{(n)},\ell_t^{(n)})_{t\geq 0}$ the interface process for the multitype contact process started from the truncated initial configuration. At last, let~$\mathbf X=\max\{-\mathbf X_1,\mathbf X_2\}$. It is not hard to see that the following is true:

\begin{claim}\label{claim_MCPN(X)}
    Almost surely, if~$n\geq N(\mathbf X)$, it follows that:
\[
\mathbf X_1 = \mathbf X_1^{(n)},\quad \mathbf X_2 = \mathbf X_2^{(n)},\quad T^{(n)} = T,\quad \text{and } I_t^{(n)} = I_t \text{ for all } t \in [0, T].
\]

Moreover, the rightmost active infection path in~$\mathcal H_1$ for the process~$(\xi_t^{(n)})_{t\geq 0}$ from~$(\mathbf X^{(n)}_1,0)$ to~$(\ell_T-1,T)$ is the same as the rightmost active infection path in~$\mathcal H_1$ for the process~$(\xi_t)_{t\geq 0}$ from~$(\mathbf X_1,0)$ to~$(\ell_T-1,T)$, and the leftmost active infection path in~$\mathcal H_2$ for the process~$(\xi_t^{(n)})_{t\geq 0}$ from~$(\mathbf X^{(n)}_2,0)$ to~$(r_T+1,T)$ is the same as the leftmost active infection path in~$\mathcal H_2$ for the process~$(\xi_t)_{t\geq 0}$ from~$(\mathbf X_2,0)$ to~$(r_T+1,T)$. 
\end{claim}

Finally, we defined the analogue of Definition~\ref{pairSpecialPaths} for the process started from the truncated initial configuration. We let~$\Gamma_1^{(n)}:[-n,T^{(n)}]\rightarrow\mathbb Z$ and~$\Gamma_2^{(n)}:[-n,T^{(n)}]\rightarrow\mathbb Z$ be the infection paths defined by
\begin{itemize}
    \item in~$[-n,0]$,~$\Gamma_1^{(n)}$ is the rightmost infection path in~$\mathcal H_1$ from~$\mathbb Z \times \{-n\}$to~$(\mathbf X_1^{(n)},0)$ and~$\Gamma_2^{(n)}$ is the leftmost infection path in~$\mathcal H_2$ from~$\mathbb Z \times \{-n\}$ to~$(\mathbf X_2^{(n)},0)$
    \item in~$[0,T^{(n)}]$,~$\Gamma_1^{(n)}$ is the rightmost active (with respect to~$(\beta_t^{(n)})_{t\geq 0}$) infection path in~$\mathcal H_1$ from~$(\mathbf X_1^{(n)},0)$ to~$(\ell_{T^{(n)}}-1,T^{(n)})$ and~$\Gamma_2^{(n)}$ is the leftmost active (with respect to~$(\beta_t^{(n)})_{t\geq 0}$) infection path in~$\mathcal H_2$ from~$(\mathbf X_2^{(n)},0)$ to~$(r_{T^{(n)}}+1,T^{(n)})$. 
\end{itemize}

We finally have:
\begin{claim}
    Almost surely, for every~$s<0$, there exists~$N'>N(\mathbf X)$ such that for all~$n\geq N'$ we have~$\Gamma_1^{(n)}\vert_{[s,T]} = \Gamma_1\vert_{[s,T]}$ and~$\Gamma_2^{(n)}\vert_{[s,T]} = \Gamma_2\vert_{[s,T]}$. 
\end{claim}

The proof of the above claim is similar to that of Claim~\ref{clNk4} and is therefore omitted.

\begin{lemma} \label{changeGraphicalConstructionMCP}
    On the same probability space where~$\mathcal{H}_{\mathrm{pair}}$ is defined, let~$\mathcal H_{\mathrm{pair}}' = (\mathcal H_1',\mathcal H_2')$ be another graphical construction for a multitype contact process with the same parameters~$\lambda_1,\lambda_2>\lambda_c$ defined for all times in~$\mathbb R$ and independent of~$\mathcal H_{\mathrm{pair}}$. Let~$\mathcal{G}$ be the~$\sigma$-algebra generated by~$T$,~$\Gamma_1$,~$\Gamma_2$,~$\mathrm{Right}_{\Gamma_1}(\mathcal H_1)$ and~$\mathrm{Left}_{\Gamma_2}(\mathcal H_2)$. Then, almost surely:

    \begin{align*}
        \mathrm{Law}\big( \mathrm{Left}_{\Gamma_1}^+(\mathcal{H}_1) \,\big|\mathcal G \big) &= \mathrm{Law}\big( \mathrm{Left}_{\Gamma_1}^+(\mathcal{H}_1') \,\big|\mathcal G \big) \\
        \mathrm{Law}\big( \mathrm{Right}_{\Gamma_2}^+(\mathcal{H}_2) \,\big|\mathcal G \big) &= \mathrm{Law}\big( \mathrm{Right}_{\Gamma_2}^+(\mathcal{H}_2') \,\big|\mathcal G \big)
    \end{align*}
\end{lemma}

The proof of the above lemma is similar to that of Lemma~\ref{lem_change_H_patchwork} and is therefore omitted.

\begin{proof}[Proof of Proposition~\ref{lemma_lawatTimeTMCP}]
    The proof proceeds exactly as in Proposition~\ref{lem_first_law}, with only notational modifications.
\end{proof}

\begin{proof}[Proof of Corollary~\ref{FirstPatchworkInterfaceMCP}]
The proof proceeds as in Corollary~\ref{FirstPatchworkInterface}, with only minor notational adjustments, and we omit the details.
\end{proof}

\subsubsection{Sewing the patchwork} \label{sec_sewing_MCP}

\begin{proposition}[Patchwork construction of~$P_{(\mathrm{g}_1, \mathrm g_2)}$] \label{prop_patchwork_MCP}
Fix~$(\mathrm g_1, \mathrm g_2) \in \mathcal R'$. Let~$(\Xi^n)_{n \ge 0}$ be a sequence with~$\Xi^n=(T^n,(\Gamma^n_1, \Gamma^n_2),D^n,(I^n_t)_{t \le T^n})$, and distribution specified inductively as follows:
\begin{align}
&\mathrm{Law}(\Xi^0) = Q_{(\rmg_1, \rmg_2)};\\[.1cm]
&\mathrm{Law}(\Xi^{n+1} \mid \Xi^0,\ldots, \Xi^n) = Q_{(\mathrm g_1\sqcup \Gamma_1^0\sqcup\dots\sqcup\Gamma_1^n,\mathrm g_2\sqcup \Gamma_2^0\sqcup\dots\sqcup\Gamma_2^n)} \; \text{ for all }n \in \mathbb N_0. \label{eq_nome_xi_n_MCP}
\end{align}

Then,~$\mathrm{Sew}((I^0_t)_{t \le T^0}, (I^1_t)_{t \le T^1},\ldots)$ has law~$P_{(\mathrm g_1, \rmg_2)}$.
\end{proposition}

The proof follows the same lines as that of Proposition~\ref{prop_patchwork} and is therefore omitted.

\begin{definition}
    We call a sequence~$(\Xi^n)_{n \ge 0}$ with law as prescribed in the statement of Proposition~\ref{prop_patchwork_MCP}, corresponding to~$(\mathrm g_1,\rmg_2) \in \mathcal R'$, a \emph{patchwork sequence with initial pair of trails $(\mathrm g_1,\rmg_2)$}.
\end{definition}

\subsubsection{Renewals of patchwork construction}
\begin{definition} \label{def_kappa_n_MCP}
    Let~$(\mathrm g_1,\rmg_2) \in \mathcal R'$ and let~$(\Xi^n)_{n \ge 0}$ be a patchwork sequence with initial pair of trails~$(\mathrm g_1,\rmg_2)$. We define the stopping times~$\kappa_n$ as in \eqref{eq_def_kappa_n}. 
\end{definition}

\begin{lemma} \label{kappa_m_kappa_n_MCP}
    If~$m < n$ and~$\kappa_m \ge n$, then~$\kappa_m \ge \kappa_n$.
\end{lemma}

\begin{proof}
    Exactly the same as the proof of Lemma~\ref{kappa_m_kappa_n}. 
\end{proof}

\begin{lemma}\label{lem_kappaPoswithPosProbMCP}
    Let~$(\Xi^n)_{n\geq 0}$ be a patchwork sequence with initial pair of trails~$(\mathrm g_1, \rmg_2)$ which is defined under some probability measure~$\mathbb P$. Then,~$\mathbb P(\kappa_0=\infty)>0$. 
\end{lemma}

\begin{proof}
    Let~$\mathcal G$ be a space of the 4-tuples~$\Xi$ as in~\eqref{XiDefMCP} that are such that~$T=1$ and~$D=-\infty$. Let~$n\in\mathbb N$ be fixed and let~$G(n):=\{\Xi^0\in \mathcal G,\dots, \Xi^n\in\mathcal G\}$. One has that~$P(G(n))\geq p^{n+1}$ for some~$p>0$. Indeed, note that the event~$\{T=1\}\cap\{D=-\infty\}$ follows if the following conditions are verified: there exits~$\gamma_1:(-\infty,0]\rightarrow \mathbb Z$ infection path in~$\mathcal H_1$ such that~$\gamma_1(0)=0$,  there exits~$\gamma_2:(-\infty,0]\rightarrow \mathbb Z$ infection path in~$\mathcal H_2$ such that~$\gamma_2(0)=1$ and~$D^0_1\cap[0,1]=D^1_2\cap[0,1]=\emptyset$. The rest of the proof follows exactly as in the proof of Lemma~\ref{lem_kappaPoswithPosProb}. 
\end{proof}

\begin{lemma}\label{lem_coupled_patchworks_MCP}
    Let~$(\Xi^n)_{n \ge 0}$ be a patchwork sequence with initial pair of trails~$(\mathrm g,\rmg_2)\in\mathcal R'$, defined under a probability measure~$\mathbb P$. Then, for every bounded and measurable function~$f$, the value of
    \[\mathbb E[f(\Xi^0,\Xi^1,\ldots) \cdot \mathds{1}\{\kappa_0 = \infty\}]\]
    does not depend on the initial trail~$(\mathrm g_1,\rmg_2)$.
\end{lemma}

The proof is done in Subsection~\ref{proofLemmaCoupledMCP}. 

\begin{corollary} \label{cor_sameAsInitial_MCP}
    Let~$(\Xi^n)_{n \ge 0}$ be a patchwork sequence with initial pair of trails~$(\mathrm g_1,\rmg_2)\in\mathcal R'$, defined under a probability measure~$\mathbb P$. Then, for every bounded and measurable function~$f$ and every~$n \ge 1$,
\begin{align*}
        &\mathbb E[f(\Xi^{n},\Xi^{n+1},\ldots)\cdot \mathds{1}\{\kappa_{n}=\infty\} \mid \Xi^0,\ldots, \Xi^{n-1}] \\[.1cm]&\hspace{4cm}= \mathbb E[f(\Xi^0,\Xi^1,\ldots) \mid \kappa_0=\infty] \text{ a.s.}
\end{align*}
\end{corollary}

\begin{proof}
    The proof is exactly the same as the proof of Corollary~\ref{cor_sameAsInitial}.
\end{proof}

\begin{lemma} \label{time_and_space_betweenRenewals_MCP}
    Let~$(\Xi^n)_{n\in\mathbb N_0}$ be a patchwork construction with initial pair of trails~$(\rmg_1,\rmg_2)\in\mathcal R'$ defined under some law~$\mathbb P$. Let~$N_0<N_1<\dots$ denote the indexes~$n\in\mathbb N$ for which one has~$\kappa_n=\infty$. Then, there exist constants~$c,C,p>0$ independent of the pair of trails~$(\mathrm g_1,\mathrm g_2)$ such that the following is true for all~$t\geq 0$:
    \begin{equation}\label{timeBetweenRenBehavesWellMCP}
        \mathbb P\left(\sum_{n=N_m}^{N_{m+1}-1}T^n >t\right) \leq Ce^{-ct^p}
    \end{equation}

    Moreover, if we let~$\tau_m:=\sum_{n=0}^{N_m-1}T^n$, then it follows that there exists constants~$c,C,p>0$ (again independent of the pair of trails~$\rmgonegtwo$) such that the following is true for all~$x\geq 0$:

    \begin{equation} \label{spaceBetweenRenBehavesWellMCP}
        \mathbb P \left(\sup\left\{|i_t - i_{\tau_m}|\,:\,t\in[\tau_m,\tau_{m+1}]\right\}>x\right) \leq Ce^{-cx^p}
    \end{equation}
\end{lemma}

\begin{proof}
    The proof of this lemma is exactly the same as the proof of Lemma~\ref{time_and_space_betweenRenewals} with minor modifications.
\end{proof}

\subsubsection{Coupled patchworks: proof of Lemma~\ref{lem_coupled_patchworks_MCP}} \label{proofLemmaCoupledMCP}

\begin{definition}\label{def_depth_and_influence_MCP}

Let~$(\rmg_1,\rmg_2), (\rmg_1',\rmg_2')\in\mathcal R'$. Consider two multitype contact processes started from~$\mu_{(\rmg_1,\rmg_2)}$ and~$\mu_{(\rmg_1',\rmg_2')}$ constructed using the same graphical construction~$\mathcal H_{\mathrm{pair}}$. Let
\begin{align*}
    &\Xi^{(\rmg_1, \rmg_2)} := (T^{(\rmg_1, \rmg_2)}, \Gamma_1^{(\rmg_1, \rmg_2)},\Gamma_2^{(\rmg_1, \rmg_2)},I^{(\rmg_1, \rmg_2)},D^{(\rmg_1, \rmg_2)}) \\
    &\Xi^{(\rmg_1', \rmg_2')} := (T^{(\rmg_1', \rmg_2')}, \Gamma_1^{(\rmg_1', \rmg_2')},\Gamma_2^{(\rmg_1', \rmg_2')},I^{(\rmg_1', \rmg_2')},D^{(\rmg_1', \rmg_2')})
\end{align*}
the 5-tuples for which their entries are defined as in Section~\ref{patchworkElementsMCP} taken with respect to the processes started from~$\mu_{(\rmg_1, \rmg_2)}$ and~$\mu_{(\rmg_1', \rmg_2')}$, respectively. At last, let~$Q_{(\rmg_1,\rmg_2),(\rmg_1',\rmg_2')}$ denote the law of this coupled pair~$\left(\Xi_(\rmg_1, \rmg_2),\Xi_{(\rmg_1', \rmg_2')}\right)$. 
\end{definition}

\begin{lemma}[Depth and influence]\label{lem_depth_and_influence_MCP}
    For all~$(\rmg_1,\rmg_2),(\rmg_1',\rmg_2') \in \mathcal R'$, we have
\begin{equation}
\begin{aligned}
Q_{(\rmg_1,\rmg_2), (\rmg_1',\rmg_2')}\big(&\{D_{(\rmg_1,\rmg_2)} = D_{(\rmg_1',\rmg_2')} = -\infty,\;
\Xi_{(\rmg_1,\rmg_2)} = \Xi_{(\rmg_1',\rmg_2')}\} \\\label{easyCase}
&\cup \{D_{(\rmg_1, \rmg_2)} \ge 0,\; D_{(\rmg_1',\rmg_2')} \ge 0\} \big) = 1.
\end{aligned}
\end{equation}
    Moreover, if~$(\rmg_1,\rmg_2) \cap (\mathbb Z \times [-t,0])^2 = (\rmg_1',\rmg_2') \cap (\mathbb Z \times [-t,0])^2$ for some~$t \ge 0$, then
\begin{equation}
\begin{aligned}
Q_{(\rmg_1,\rmg_2), (\rmg_1',\rmg_2')}\big(&\{D_{(\rmg_1,\rmg_2)} = D_{(\rmg_1',\rmg_2')} < t,\; 
\Xi_{(\rmg_1,\rmg_2)} = \Xi_{(\rmg_1',\rmg_2')} \} \\ \label{hardCase}
&\cup \{D_{(\rmg_1,\rmg_2)} \ge t,\; D_{(\rmg_1',\rmg_2')} \ge t\} \big) = 1.
\end{aligned}
\end{equation}
\end{lemma}

\begin{proof}
    The proof follows the same lines as Lemma~\ref{lem_depth_and_influence}, with only minor modifications, and is therefore omitted.
\end{proof}

\begin{proof}[Proof of Lemma~\ref{lem_coupled_patchworks_MCP}]
    The proof is the same as that of Lemma~\ref{lem_coupled_patchworks}, with only minor modifications.
\end{proof}

\subsection{Proof of Theorems~\ref{Thm_MCP_Tight} and~\ref{Thm_MCP_CLT}} \label{sec_proofMCP}

\paragraph{} The proofs of the main theorems for the multitype contact process follow very closely the proofs done in Section~\ref{sec_proofsThmsCBP}. We first observe that we can obtain the Heaviside configuration from a configuration of the form~$\mu_{\rmgonegtwo}$ as in~\eqref{eq_def_mu_g1g2} by considering the pair of trails~$(\mathrm{g}_1^{\mathrm h},\mathrm{g}_2^{\mathrm h})\in\mathcal R'$ given by:
\[
\mathrm{g}_1^{\mathrm h} = \bigcup_{x\in \mathbb Z,\,x\leq 0}\left\{\{x\}\times \{0\}\right\} \quad \text{ and }\quad \mathrm{g}_2^{\mathrm h} = \bigcup_{x\in \mathbb Z,\,x\geq 1}\left\{\{x\}\times \{0\}\right\} 
\]

Throughout the rest of this section, we fix the following. Let~$\mathbb P$ be a probability measure under which is defined a patchwork sequence~$(\Xi^n)_{n\in\mathbb N_0}$ for the multitype contact process started from the Heaviside configuration constructed using~$\mu_{(\mathrm{g}_1^{\mathrm h},\mathrm{g}_2^{\mathrm h})}$. 

Let~$(\kappa_n)_{n\in\mathbb N_0}$ be as in Definition~\ref{def_kappa_n_MCP}, and let~$N_0, N_1,\dots$ be the increasing sequence of indexes~$n\in\mathbb N_0$ for which we have~$\kappa_n=\infty$. To see that such infinite sequence exist, we argue that we are in the conditions of Lemma~\ref{lem_renewal}. Indeed,~$\kappa_n\geq n$ by definition so that condition~$(i)$ is satisfied. Because of Lemmas~\ref{kappa_m_kappa_n_MCP} and~\ref{lem_kappaPoswithPosProbMCP}, we are also in condition~$(ii)$  and~$(iv)$ of Lemma~\ref{lem_renewal}. Moreover, Corollary~\ref{cor_sameAsInitial_MCP} guarantees that we are under condition~$(iii)$ of~\ref{lem_renewal}. 

Again using the same notation as in Lemma~\ref{lem_renewal}, let~$Y_n:=\Xi^n$. In that way, the sequence as in~\eqref{seq_iid_CBP} is again i.i.d. distributed with the same law as~$(N_1,\Xi^0,\dots,\Xi^{N_1-1})$ conditioned on~$\{\kappa_0=0\}$ now regarding the multitype contact process. The objects~$\tau_n$,~$N(t)$,~$\sigma(t)$ and~$\Bar{\mathbb P}$ are defined as in~\eqref{eq_taun_CBP},~\eqref{eq_Nt_CBP},~\eqref{eq_sigmat_CBP} and~\eqref{eq_conditionalLaw_CBP}, respectively. 

Regarding now the multitype contact process, the analogue result of Lemma~\ref{lem_timeSinceLastRenewal_CBP} is true:

\begin{lemma} \label{lem_timeSinceLastRenewal_MCP}
    On the above conditions, there exist constants~$c,C,p>0$ such that for all~$t\in[0,\infty)$ it follows that:
    \begin{equation*}
        \Bar{\mathbb P}\left(\max\left\{(t-\sigma(t)),\sup\{|i_s-i_{\sigma(t)}|\,:\,s\in[\sigma(t),t]\}\right\}>L\right)\leq Ce^{-cL^p} \quad \forall L\geq 0
    \end{equation*}
\end{lemma}

\begin{proof}
    The proof is the same as the proof of Lemma~\ref{lem_timeSinceLastRenewal_CBP} with only a minor modifications. 
\end{proof}

We are finally in conditions to prove the main results regarding the multitype contact process.

\begin{proof}[Proof of Theorems~\ref{Thm_MCP_Tight}]
    The proof follows the same lines as the proof of Theorem~\ref{ThmCBP_Tightness}, with a few minor modifications.  
\end{proof}

\begin{proof}[Proof of Theorems~\ref{Thm_MCP_CLT}]
    The proof is exactly the same as the proof of the Theorem~\ref{ThmCBP_CLT_Speed}, with the reference to Lemma~\ref{time_and_space_betweenRenewals} being replaced by a reference to Lemma~\ref{time_and_space_betweenRenewals_MCP}. 
\end{proof}

\section{Proofs of Lemma~\ref{lem_renewal} and~\ref{lem_CLT_RenewalProcess}} \label{sec_proofRenewalLemmas}

We first prove Lemma~\ref{lem_renewal}.

\begin{proof}[Proof of Lemma~\ref{lem_renewal}]
	We split part of the proof in a sequence of claims.
	\begin{claim} \label{claim_infinitelyManyInfinity}
Almost surely, there are infinitely many~$n$ for which~$\kappa_n = \infty$.
	\end{claim}
	\begin{proof}
It suffices to prove that
		\begin{equation}\label{eq_infinite_m}
			\mathbb P(\exists n \ge m:\; \kappa_n = \infty) = 1 \quad \text{for every }m \in \mathbb N.
		\end{equation}
	To do so, fix~$m \in \mathbb N$. Define a sequence~$(T_j)_{j \ge 1}$ by letting~$T_1 = \kappa_m$, and recursively,
    \[
    T_{j+1}:= \begin{cases}
        \kappa_{T_j + 1}& \text{if } T_j < \infty;\\
        \infty&\text{otherwise.}
    \end{cases}
    \]

	For any~$k \in \mathbb N$ and~$n_1,\ldots,n_{k-1} \in \mathbb N$ with~$m \le n_1 <\cdots < n_{k-1}$, we have
	\begin{align*}
		&\mathbb P(T_1=n_1,\;\ldots,\; T_{k-1}=n_{k-1},\; T_k < \infty)\\
		&= \mathbb P(T_1=n_1,\;\ldots,\; T_{k-1}=n_{k-1}) \cdot \mathbb P(\kappa_0 < \infty),
	\end{align*}
	by conditioning to~$\mathcal F_{n_{k-1}}$ and using property~$\mathrm{(iii)}$ with~$g \equiv 1$. Summing over~$n_1,\ldots, n_{k-1}$ then gives
	\[\mathbb P(T_k<\infty) \le \mathbb P(T_{k-1}<\infty) \cdot \mathbb P(\kappa_0 < \infty).\]
		Iterating gives~$\mathbb P(T_k < \infty) < \mathbb P(\kappa_0 < \infty)^k$. Using property~$\mathrm{(iv)}$, this tends to zero as~$k \to \infty$, so~\eqref{eq_infinite_m} follows.
	\end{proof}

	Before proceeding to the next claims, we give a definition.
For each~$n \in \mathbb N_0$, let~${\mathcal F}^+_n$ denote the~$\sigma$-algebra generated by~$\mathcal F_n$ and the event~$\{\kappa_{n+1} = \infty\}$. 
	\begin{claim}\label{cl_cons_ii}
		For any~$k \in \mathbb N$ and any natural numbers~$n_0 < \cdots < n_k$, the event~$\{N_0 = n_0,\ldots, N_k=n_k\}$ belongs to~${\mathcal F}^+_{n_k-1}$.
	\end{claim}
	\begin{proof}
		Letting~$[n_k]:=\{1,\ldots,n_k\}$ and~$\Lambda := \{n_0,n_1,\ldots, n_k\}$, we have
\begin{equation}\label{eq_new_event_kappa}
    \{N_0 = n_0,\ldots, N_k=n_k\} = \left\{\begin{array}{l}\kappa_i = \infty \; \forall i \in \Lambda,\\[.1cm] \kappa_i < \infty \; \forall i \in [n_k] \backslash \Lambda \end{array}\right\}.
\end{equation}
        
We now observe that property~$(\mathrm{ii})$ implies that, if~$m < n$ and~$\kappa_n = \infty$, then~$\kappa_m$ is either~$< n$ or~$=\infty$.
Using this, we see that the event on the right-hand side of~\eqref{eq_new_event_kappa} is the same as the event that:
		\begin{align*}
			&\kappa_i < n_0 \text{ for } i \in \{0,\ldots, n_0-1\},\; \kappa_{n_0} \ge n_k\\
			&\kappa_i < n_1 \text{ for } i \in \{n_0+1,\ldots, n_1-1\},\; \kappa_{n_1} \ge n_k, \ldots \\
			&\kappa_i < n_{k-1} \text{ for } i \in \{n_{k-2}+1,\ldots, n_{k-1}-1\},\; \kappa_{n_{k-1}} \ge n_k,\\
			&\kappa_i < n_k \text{ for }i \in \{n_{k-1}+1,\ldots, n_k-1\},\; \kappa_{n_k}=\infty.
		\end{align*}
    The statement of the claim then follows from the fact that the~$\kappa_i$'s are stopping times and from the definition of~${\mathcal F}^+_{n_{k}-1}$.
	\end{proof}

	\begin{claim} \label{cl_alt_iii} For every bounded and measurable $g$ and every~$n \ge 1$,  we have
	\begin{equation*} \begin{split}
		&\mathbb E[g(Y_{n},\hat \kappa_{n},Y_{n+1},\hat\kappa_{n+1},\ldots) \mid {\mathcal F}^+_{n-1}]= \mathbb E[g(Y_0,\hat{\kappa}_0,Y_1,\hat{\kappa}_1,\ldots) \mid \kappa_0 = \infty]
        \end{split}
	\end{equation*}
    almost surely on~$\{\kappa_n = \infty\}$.
	\end{claim}
	\begin{proof}
It will be useful to note that, applying property~$\mathrm{(iii)}$ with~$g \equiv 1$, we have
	\begin{equation}
		\label{eq_iii_1} \{\kappa_n = \infty\} \text{ is independent of } \mathcal F_{n-1}
	\end{equation}
	and
	\begin{equation}
		\label{eq_iii_2} \mathbb P(\kappa_n = \infty) = \mathbb P(\kappa_0 = \infty)
	\end{equation}
	for all~$n$.

	Fix~$g$ bounded and measurable, and abbreviate
		\[W_n:=g(Y_{n},\hat \kappa_n, Y_{n+1}, \hat{\kappa}_{n+1},\ldots),\quad A_n:=\{\kappa_n = \infty\}, \quad c:= \mathbb E[W_0  \mid A_0].\]
        We need to prove that
\begin{equation*}
    \mathbb E[W_n \mid \mathcal F_{n-1}^+] \cdot \mathds{1}_{A_n} = c \cdot \mathds{1}_{A_n},
\end{equation*}
or equivalently (since~$A_n \in \mathcal F_{n-1}^+$), that
\begin{equation*}
    \mathbb E[W_n \cdot \mathds{1}_{A_n} \mid \mathcal F_{n-1}^+]  = c \cdot \mathds{1}_{A_n}.
\end{equation*}
Recall that~$\mathbb E[W_n \cdot \mathds{1}_{A_n} \mid \mathcal F_{n-1}^+]$ is the almost surely unique random variable~$X$ that satisfies~$\mathbb E[X \cdot \mathds{1}_B] = \mathbb E[W_n \cdot \mathds{1}_{A_n} \cdot  \mathds{1}_B]$ for all~$B \in \mathcal F_{n-1}^+$. We then need to show that~$X:=c\cdot \mathds{1}_{A_n}$ does the job, that is, we need to check that~$\mathbb E[c \cdot \mathds{1}_{A_n} \cdot \mathds{1}_B] = \mathbb E[W_n \cdot \mathds{1}_{A_n} \cdot \mathds{1}_B]$ for all~$B \in \mathcal F^+_{n-1}$, that is,
\begin{equation}\label{eq_check_cond0}
    c \cdot \mathbb P(A_n \cap B) = \mathbb E[W_n \cdot \mathds{1}_{A_n \cap B}]\quad \text{for all }B \in \mathcal F^+_{n-1}.
\end{equation}
Since~$\mathcal F^+_{n-1}$ is the~$\sigma$-algebra generated by~$\mathcal F_{n-1}$ and~$A_n$, to prove~\eqref{eq_check_cond0}, it suffices to prove that
\begin{equation}\label{eq_check_cond}
    c \cdot \mathbb P(A_n \cap B) = \mathbb E[W_n \cdot \mathds{1}_{A_n \cap B}]\quad \text{for all }B \in \mathcal F_{n-1}.
\end{equation}

By the definition of~$c$ and~\eqref{eq_iii_1} and~\eqref{eq_iii_2}, we have
\[
c \cdot \mathbb P(A_n \cap B) = \mathbb E[W_0 \mid A_0] \cdot \mathbb P(A_0) \cdot \mathbb P(B) = \mathbb E[W_0 \cdot \mathds{1}_{A_0}] \cdot \mathbb P(B).
\]
On the other hand, we have
\begin{align*}
    \mathbb E[W_n \cdot \mathds{1}_{A_n \cap B}] = \mathbb E[\mathbb E[W_n \cdot \mathds{1}_{A_n} \mid \mathcal F_{n-1}] \cdot \mathds{1}_B] \stackrel{\eqref{eq_property_iii}}{=} \mathbb E[W_0  \cdot \mathds{1}_{A_0}] \cdot \mathbb P(B),
\end{align*}
so~\eqref{eq_check_cond} holds and the proof is complete.\qedhere
\end{proof}

	We are now ready to complete the proof of the lemma, by induction.
		For a bounded and measurable function~$f$, using~$\mathbb P(N_0<\infty)=1$, we have
	\begin{align*}
		\mathbb E[f(N_1-N_0, Y_{N_0},\ldots, Y_{N_1-1})] = \sum_n\mathbb E[\mathds{1}_{\{N_0=n\}}\cdot f(N_1-N_0, Y_{N_0},\ldots, Y_{N_1-1})].
	\end{align*}
   We now condition on~${\mathcal F}^+_{n-1}$ inside the expectation inside the sum.
    Note that
    \[\{N_0 = n\} = \{\kappa_m < \infty \; \forall m < n,\; \kappa_n = \infty\} \in {\mathcal F}^+_{n-1}\]
    by Claim~\ref{cl_cons_ii}. Moreover, by Claim~\ref{cl_alt_iii}, 
    \[\mathbb E[f(N_1-N_0, Y_{N_0},\ldots, Y_{N_1-1})\mid {\mathcal F}^+_{n-1}] = \mathbb E[f(N_1,Y_0,\ldots,Y_{N_1-1}) \mid \kappa_0=\infty]\]
almost surely on~$\{\kappa_n = \infty\}$, and hence also on~$\{N_0=n\}$, which is a smaller event. Hence, the above sum equals
	\[\sum_n \mathbb P(N_0=n) \cdot \mathbb E[f(N_1,Y_0,\ldots, Y_{N_1-1}) \mid \kappa_0 = \infty].\]
	Again using~$\mathbb P(N_0 < \infty)=1$, the above equals
    \[\mathbb E[f(N_1,Y_0,\ldots, Y_{N_1-1}) \mid \kappa_0 = \infty].\] This shows that the law of~$(N_1-N_0,Y_{N_0},\ldots, Y_{N_1-1})$ under~$\mathbb P$ is the same as the law of $(N_1,Y_0,\ldots, Y_{N_1-1})$ under~$\mathbb P(\cdot \mid \kappa_0=\infty)$.

The induction step is very similar.
	Let us abbreviate
	\[
\Psi_k:=(N_k-N_{k-1}, Y_{N_{k-1}},\ldots, Y_{N_k-1}), \quad k \in \mathbb N_0.
	\]
	Assume that we have proved that~$\Psi_0,\ldots,\Psi_k$ are independent and identically distributed, with the correct distribution. For bounded and measurable functions~$f$ and~$g$, similarly to the base case, we have
\begin{align*}
	\mathbb E[f(\Psi_0,\ldots, \Psi_k) \cdot g(\Psi_{k+1})] &= \sum_{n} \mathbb E[f(\Psi_0,\ldots, \Psi_k)\cdot \mathds{1}_{\{N_k=n\}} \cdot \mathbb E[g(\Psi_{k+1})\mid {\mathcal F}^+_{n-1}]]\\
	&= \sum_n \mathbb E[f(\Psi_0,\ldots, \Psi_k)\cdot \mathds{1}_{\{N_k=n\}}] \cdot  \mathbb E[g(\Psi_{0})\mid \kappa_0 = \infty]\\
	&=\mathbb E[f(\Psi_0,\ldots, \Psi_k)] \cdot  \mathbb E[g(\Psi_{0})\mid \kappa_0 = \infty],
\end{align*}
completing the proof.
\end{proof}

We finally prove Lemma~\ref{lem_CLT_RenewalProcess}. 

\begin{proof}[Proof of Lemma~\ref{lem_CLT_RenewalProcess}]
For $t\ge0$, define $N(t)=0$ if $t<\tau^1$ and
\[
N(t)=\sup\{n\in\mathbb N:\tau^n\le t\}
\quad\text{otherwise},
\]
and set $\tau^0=0$. Let
\[
\mu_1=\mathbf E(\tau^n-\tau^{n-1}),\qquad
\mu_2=\mathbf E\big[X(\tau^n)-X(\tau^{n-1})\big],
\qquad
\mu=\frac{\mu_2}{\mu_1}.
\]

For $t\ge0$, we have that~$\frac{X(t)-t\mu}{\sqrt t}$ is equal to
\[
\underbrace{\frac{X(t)-X(\tau^{N(t)})}{\sqrt t}}_{(1)}
+
\underbrace{\frac{X(\tau^{N(t)})-X(\tau^{\lfloor t/\mu_1\rfloor})}{\sqrt t}}_{(2)}
+
\underbrace{\frac{X(\tau^{\lfloor t/\mu_1\rfloor})-X(\tau^1)-t\mu}{\sqrt t}}_{(3)}
+
\underbrace{\frac{X(\tau^1)}{\sqrt t}}_{(4)}.
\]

We analyze the four terms separately as $t\to\infty$.

\medskip
\noindent
\textbf{(1)} For $\varepsilon>0$,~$(1)$ is bounded by
\begin{equation*}
    \mathbf P\!\left(
\left|\frac{N(t)}{t}-\frac1{\mu_1}\right|>\tfrac12
\right)+
\mathbf P\!\left(
|X(t)-X(\tau^{N(t)})|>\varepsilon\sqrt t,
\ \left|\frac{N(t)}{t}-\frac1{\mu_1}\right|\le\tfrac12
\right).
\end{equation*}
The first term vanishes by the renewal theorem. Writing
$n_1=\lfloor t/\mu_1-\delta t\rfloor$ and
$n_2=\lceil t/\mu_1+\delta t\rceil$, the second term is bounded by
\[
\mathbf P\!\left(
\bigcup_{n_1\le n\le n_2}
\sup_{s\in[\tau^n,\tau^{n+1})}
|X(s)-X(\tau^n)|
>\varepsilon\sqrt t
\right)
\le (2\delta t+1)Ce^{-c(\varepsilon\sqrt t)^p},
\]
which tends to $0$.

\medskip
\noindent
\textbf{(2)} Similarly,~$(2)$ is bounded by
\begin{equation*}
\mathbf P\!\left(
\left|\frac{N(t)}{t}-\frac1{\mu_1}\right|>\delta
\right)+
\mathbf P\!\left(
\max_{n_0\le n\le n_0+\delta t}
|X(\tau^n)-X(\tau^{n_0})|
>\varepsilon\sqrt t
\right)
\end{equation*}
where $n_0=\lfloor t/\mu_1\rfloor$.  
By Kolmogorov’s inequality, the second term is bounded by~$\frac{\delta\,\mathrm{Var}(X(\tau^1)-X(\tau^0))}{\varepsilon^2}$, which can be made arbitrarily small by choosing $\delta$ small.

\medskip
\noindent
\textbf{(3)} Since~$X(\tau^{\lfloor t/\mu_1\rfloor})-X(\tau^0)
=
\sum_{n=1}^{\lfloor t/\mu_1\rfloor}
\big(X(\tau^n)-X(\tau^{n-1})\big)$,
a sum of i.i.d.\ variables with mean $\mu_2$, the central limit theorem yields
\[
\frac{X(\tau^{\lfloor t/\mu_1\rfloor})-X(\tau^0)-t\mu}{\sqrt t}
\Rightarrow \mathcal N(0,\sigma^2).
\]

\medskip
\noindent
\textbf{(4)} Since $X(\tau^1)$ is finite a.s.,
$\frac{X(\tau^1)}{\sqrt t}\to0$ in probability.

\end{proof}

\newpage
\printbibliography
\end{document}